\documentclass{amsart}
\usepackage{amsmath,amsthm,amssymb}
\usepackage[T1]{fontenc}
\usepackage{url}
\usepackage{graphics}

\newtheorem{theorem}{Theorem}[section]
\newtheorem{proposition}[theorem]{Proposition}
\newtheorem{lemma}[theorem]{Lemma}
\newtheorem{corollary}[theorem]{Corollary}

\theoremstyle{remark}
\newtheorem{remark}[theorem]{Remark}
\newtheorem*{definition}{Definition}
\newtheorem{example}{Example}

\newtheorem*{notations}{Notations}

\numberwithin{equation}{section}

\newcommand{\bez}{\nopagebreak\hspace*{\fill}\nolinebreak$\Box$}
\newcommand{\vep}{\varepsilon}
\newcommand{\R}{{\mathbb{R}}}

\newcommand{\Q}{{\mathbb{Q}}}

\newcommand{\Z}{{\mathbb{Z}}}
\newcommand{\N}{{\mathbb{N}}}
\newcommand{\T}{{\mathbb{T}}}

\newcommand{\xbm}{(X,\mathcal{B},\mu)}

\newcommand{\cp}{\mathcal{P}}

\newcommand{\sgn}{\operatorname{sgn}}
\newcommand{\var}{\operatorname{Var}}
\newcommand{\bv}{\operatorname{BV}}
\newcommand{\pl}{\operatorname{PL}}
\newcommand{\Int}{\operatorname{Int}}

\newcommand{\SL}{\operatorname{SL}}

\newcommand{\supp}{\operatorname{supp}}

\begin{document}
\title[Cocycles over IETs and multivalued Hamiltonian flows]
{Cocycles over interval exchange transformations and multivalued
Hamiltonian flows}
\author[J.-P.\ Conze, K. Fr\k{a}czek]{Jean-Pierre Conze and Krzysztof Fr\k{a}czek}
\address{Universit\'e de Rennes I, IRMAR, CNRS UMR 6625, Campus de Beaulieu, 35042 Rennes Cedex, France}
\email{Jean-Pierre.Conze@univ-rennes1.fr}

\address{Faculty of Mathematics and Computer Science, Nicolaus
Copernicus University, ul. Chopina 12/18, 87-100 Toru\'n, Poland}
\address{Institute of Mathematics\\
Polish Academy of Science\\
ul. \'Sniadeckich 8\\
00-956 Warszawa, Poland } \email{fraczek@mat.umk.pl}
\date{}

\thanks{Research partially supported by MNiSzW grant N N201
384834 and Marie Curie "Transfer of Knowledge" program, project
MTKD-CT-2005-030042 (TODEQ)}

\begin{abstract}
We consider interval exchange transformations of periodic type and
construct different classes of recurrent ergodic cocycles of
dimension $\geq 1$ over this special class of IETs. Then using
Poincar\'e sections we apply this construction to obtain recurrence
and ergodicity for some smooth flows on non-compact manifolds which
are extensions of multivalued Hamiltonian flows on compact surfaces.
\end{abstract}

\maketitle

\tableofcontents

\section{Introduction}

Let $T:\xbm\to\xbm$ be an ergodic automorphism of a standard Borel
probability space and $G$ be a locally compact abelian group with
identity element denoted by $0$. We will consider essentially the
case $G = \R^\ell$, for $\ell \geq 1$.

\vskip 3mm Each measurable function $\varphi:X\rightarrow G$
determines a {\em cocycle} $\varphi^{(\,\cdot\,)}:\Z\times X \to
G$ for $T$ by the formula
\begin{displaymath}
\varphi^{(n)}(x)=\left\{
\begin{array}{ccl}
\varphi(x)+\varphi(Tx)+\ldots+\varphi(T^{n-1}x), & \mbox{if} & n>0 \\
0, & \mbox{if} & n=0,\\
-(\varphi(T^nx)+\varphi(T^{n+1}x)+\ldots+\varphi(T^{-1}x)), &
\mbox{if} & n<0.
\end{array}
\right.
\end{displaymath}
We consider the associated {\it skew product}
\begin{eqnarray}
T_{\varphi}:(X\times
G,{\mathcal{B}}\times{\mathcal{B}_{G}},{\mu}\times m_{G})
&\rightarrow&
(X\times G,{\mathcal{B}}\times{\mathcal{B}_{G}},{\mu}\times m_{G}), \nonumber \\
T_{\varphi}(x,g)&=&(Tx,g+{\varphi}(x)),
\end{eqnarray}
where $\mathcal{B}_{G}$ denotes the $\sigma$-algebra of Borel
subsets and $m_{G}$ the Haar measure of $G$.

\vskip 3mm The cocycle $(\varphi^{(\,\cdot\,)})$ can be viewed as
a {\it "stationary" walk} in $G$ over the dynamical system $(X,
\mu ,T)$. We say that it is {\it recurrent} if
$(\varphi^{(n)}(x))$ returns for a.e. $x$ infinitely often in any
neighborhood of the identity element. The transformation
$T_\varphi$ is then conservative for the invariant $\sigma$-finite
measure $\mu \times
 m_G$. If moreover the system $(X \times G, \mu \times
 m_G, T_\varphi)$ is {\it ergodic}, we say that the cocycle
$\varphi^{(\,\cdot\,)}$ is ergodic. For simplicity, the expression
"cocycle $\varphi$" refers to the cocycle
$(\varphi^{(\,\cdot\,)})$ generated by $\varphi$ over the
dynamical system $(X,\mathcal{B},\mu,T)$.

\vskip 3mm A problem is the construction of recurrent ergodic
cocycles defined over a given dynamical system by regular functions
$\varphi$ with values in $\R^\ell$. There is an important literature
on skew products over an irrational rotation on the circle, and
several classes of ergodic cocycles with values in $\R$ or $\R^\ell$
are known in that case (see \cite{Le-Pa-Vo}, \cite{Or} and
\cite{Pa1} for some classes of ergodic piecewise absolutely
continuous non-continuous $\R$-cocycles, \cite{Gr} for examples of
ergodic cocycles with values in a nilpotent group, \cite{Co-Gu} for
ergodic cocycles in $\Z^2$ associated to special directional
rectangular billiard flows in the plane).

\vskip 3mm Skew products appear in a natural way in the study of
the billiard flow in the plane with $\Z^2$ periodically
distributed obstacles. For instance when the obstacles are
rectangles, they can be modeled as skew products over interval
exchange transformations (abbreviated as IETs). Recurrence and
ergodicity of these models are mainly open questions. Nevertheless
a first step is the construction of recurrent ergodic cocycles
over some classes of IETs (see also a recent paper by P.\ Hubert
and B.\ Weiss \cite{HuWe} for cocycles associated to non-compact
translation surfaces).

\vskip 3mm For the rotations on the circle, a special class consists
in the rotations with bounded partial quotients. For IETs, it is
natural to consider the so-called interval exchange transformations
of periodic type. The aim of this paper is to construct different
classes of recurrent ergodic cocycles over IETs in this special
class.

\vskip 3mm This is done in Sections \ref{piecewise}, \ref{step}, and
\ref{correct}. In Section \ref{prelim} we recall basic facts about
IETs of periodic type, as well as from the ergodic theory of
cocycles. In the appendix proofs of the needed results on the growth
of cocycles of bounded variation (abbreviated as BV cocycles) are
given, mainly adapted from \cite{Ma-Mo-Yo}.

\vskip 3mm In Sections \ref{multihami} and \ref{constmulti} we
present smooth models for recurrent and ergodic systems based on
the previous sections. We deal with a class of smooth flows on
non-compact manifolds which are extensions of multivalued
Hamiltonian flows on compact surfaces of higher genus. These flows
have Poincar\'e sections for which the first recurrence map is
isomorphic to a skew product of an IET and a BV cocycle. This
allows us to prove a sufficient condition for recurrence and
ergodicity (see Section~\ref{multihami}) whenever the IET is of
periodic type. In Section~\ref{constmulti} we show how to
construct explicit non-compact ergodic extensions of some
Hamiltonian flows.

\vskip 4mm
\section{Preliminaries \label{prelim}}

\subsection{Interval exchange transformations}
\vskip 3mm

\hfill \break In this subsection, we recall standard facts on IET's,
with the presentation and notations from \cite{Vi0} and \cite{ViB}.
Let $\mathcal{A}$ be a $d$-element alphabet and let
$\pi=(\pi_0,\pi_1)$ be a pair of bijections
$\pi_\vep:\mathcal{A}\to\{1,\ldots,d\}$ for $\vep=0,1$. Denote by
$\mathcal{S}^0_{\mathcal{A}}$ the subset of irreducible pairs, i.e.\
such that $\pi_1\circ\pi_0^{-1}\{1,\ldots,k\}\neq\{1,\ldots,k\}$ for
$1\leq k<d$. We will denote by $\pi^{sym}_d$ any pair
$(\pi_0,\pi_1)$ such that $\pi_1\circ\pi_0^{-1}(j)=d+1-j$ for $1\leq
j\leq d$.

Let us consider
$\lambda=(\lambda_\alpha)_{\alpha\in\mathcal{A}}\in
\R_+^{\mathcal{A}}$, where $\R_+=(0,+\infty)$. Set
\[|\lambda|=\sum_{\alpha\in\mathcal{A}}\lambda_\alpha,\;\;\;
I=\left[0,|\lambda|\right)\] and
\[I_{\alpha}=[l_\alpha,r_\alpha),\text{ where
}l_\alpha=\sum_{\pi_0(\beta)<\pi_0(\alpha)}\lambda_\beta,\;\;\;r_\alpha
=\sum_{\pi_0(\beta)\leq\pi_0(\alpha)}\lambda_\beta.\] Then
$|I_\alpha|=\lambda_\alpha$. Denote by $\Omega_\pi$ the matrix
$[\Omega_{\alpha\,\beta}]_{\alpha,\beta\in\mathcal{A}}$ given by
\[\Omega_{\alpha\,\beta}=
\left\{\begin{array}{cl} +1 & \text{ if
}\pi_1(\alpha)>\pi_1(\beta)\text{ and
}\pi_0(\alpha)<\pi_0(\beta),\\
-1 & \text{ if }\pi_1(\alpha)<\pi_1(\beta)\text{ and
}\pi_0(\alpha)>\pi_0(\beta),\\
0& \text{ in all other cases.}
\end{array}\right.\]
Given $(\pi,\lambda)\in
\mathcal{S}^0_{\mathcal{A}}\times\R_+^\mathcal{A}$, let
$T_{(\pi,\lambda)}:[0,|\lambda|)\rightarrow[0,|\lambda|)$ stand for
the {\em interval exchange transformation} (IET) on $d$ intervals
$I_\alpha$, $\alpha\in\mathcal{A}$, which are rearranged according
to the permutation $\pi^{-1}_1\circ\pi_0$, i.e.\
$T_{(\pi,\lambda)}x=x+w_\alpha$ for $x\in I_\alpha$, where
$w=\Omega_\pi\lambda$.

Note that for every $\alpha\in\mathcal{A}$ with $\pi_0(\alpha)\neq
1$ there exists $\beta\in\mathcal{A}$ such that $\pi_0(\beta)\neq
d$ and $l_\alpha=r_\beta$. It follows that
\begin{equation}\label{zbzero}
\{l_\alpha:\alpha\in\mathcal{A},\;\pi_0(\alpha)\neq 1\}=
\{r_\alpha:\alpha\in\mathcal{A},\;\pi_0(\alpha)\neq d\}.
\end{equation}
By $\widehat{T}_{(\pi,\lambda)}:(0,|I|]\to(0,|I|]$  denote the
exchange of the intervals
$\widehat{I}_\alpha=(l_\alpha,r_\alpha]$, $\alpha\in\mathcal{A}$,
i.e.\ $T_{(\pi,\lambda)}x=x+w_\alpha$ for $x\in
\widehat{I}_\alpha$. Note that for every $\alpha\in\mathcal{A}$
with $\pi_1(\alpha)\neq 1$ there exists $\beta\in\mathcal{A}$ such
that $\pi_1(\beta)\neq d$ and
$T_{(\pi,\lambda)}l_\alpha=\widehat{T}_{(\pi,\lambda)}r_\beta$. It
follows that
\begin{equation}\label{zbjeden}
\{T_{(\pi,\lambda)}l_\alpha:\alpha\in\mathcal{A},\;\pi_1(\alpha)\neq
1\}=
\{\widehat{T}_{(\pi,\lambda)}r_\alpha:\alpha\in\mathcal{A},\;\pi_1(\alpha)\neq
d\}.
\end{equation}

A pair ${(\pi,\lambda)}$ satisfies the {\em Keane condition} if
$T_{(\pi,\lambda)}^m l_{\alpha}\neq l_{\beta}$ for all $m\geq 1$
and for all $\alpha,\beta\in\mathcal{A}$ with $\pi_0(\beta)\neq
1$.

\vskip 3mm Let $T=T_{(\pi,\lambda)}$, $(\pi,\lambda)
\in\mathcal{S}^0_{\mathcal{A}}\times\R_+^{\mathcal{A}}$, be an IET
satisfying Keane's condition. Then
$\lambda_{\pi_0^{-1}(d)}\neq\lambda_{\pi_1^{-1}(d)}$. Let
\[\tilde{I}=\left[0,\max\left({l}_{\pi_0^{-1}(d)},{l}_{\pi_1^{-1}(d)}\right)\right)\]
and denote by $\mathcal{R}(T)=\tilde{T}:\tilde{I}\to\tilde{I}$ the
first return map of $T$ to the interval $\tilde{I}$. Set
\[\vep(\pi,\lambda)=\left\{
\begin{array}{ccl}
0&\text{ if }&\lambda_{\pi_0^{-1}(d)}>\lambda_{\pi_1^{-1}(d)},\\
1&\text{ if }&\lambda_{\pi_0^{-1}(d)}<\lambda_{\pi_1^{-1}(d)}.
\end{array}
\right.\]  Let us consider a pair
$\tilde{\pi}=(\tilde{\pi}_0,\tilde{\pi}_1)\in\mathcal{S}^0_{\mathcal{A}}$,
where
\begin{eqnarray*}\tilde{\pi}_\vep(\alpha)&=&\pi_\vep(\alpha)
\text{ for all }\alpha\in\mathcal{A}\text{ and }\\
\tilde{\pi}_{1-\vep}(\alpha)&=&\left\{
\begin{array}{cll}
\pi_{1-\vep}(\alpha)& \text{ if
}&\pi_{1-\vep}(\alpha)\leq\pi_{1-\vep}\circ\pi^{-1}_\vep(d),\\
\pi_{1-\vep}(\alpha)+1& \text{ if
}&\pi_{1-\vep}\circ\pi^{-1}_\vep(d)<\pi_{1-\vep}(\alpha)<d,\\
\pi_{1-\vep}\pi^{-1}_\vep(d)+1& \text{ if
}&\pi_{1-\vep}(\alpha)=d.\end{array} \right.
\end{eqnarray*}
As it was shown by Rauzy in \cite{Ra}, $\tilde{T}$ is also an IET
on $d$-intervals
\[\tilde{T}=T_{(\tilde{\pi},\tilde{\lambda})}\text{ with }
\tilde{\lambda}=\Theta^{-1}(\pi,\lambda)\lambda,\] where
\[\Theta(T)=\Theta(\pi,\lambda)=I+E_{\pi_{\vep}^{-1}(d)\,
\pi_{1-\vep}^{-1}(d)}\in\SL(\Z^{\mathcal{A}}).\]
Moreover,
\begin{equation}\label{omega}
\Theta^t(\pi,\lambda)\Omega_{\pi}\Theta(\pi,\lambda)=\Omega_{\tilde{\pi}}.
\end{equation}

It follows that $\ker\Omega_{{\pi}} =
\Theta(\pi,\lambda)\ker\Omega_{\tilde{\pi}}$. We have also
$\Omega_{{\pi}}^t = - \Omega_{{\pi}}$. Thus taking
$H_{\pi}=\Omega_{{\pi}}(\R^{\mathcal{A}})=\ker\Omega_{{\pi}}^\perp$,
we get $H_{\tilde{\pi}}=\Theta^t(\pi,\lambda)H_{{\pi}}$. Moreover,
$\dim H_{{\pi}}=2g$ and $\dim \ker\Omega_{{\pi}}=\kappa-1$, where
$\kappa$ is the number of singularities and $g$ is the genus of the
translation surface  associated to $\pi$. For more details we refer
the reader to \cite{ViB}.

\vskip 3mm The IET $\tilde{T}$ fulfills the Keane condition as well.
Therefore we can iterate the renormalization procedure and generate
a sequence of IETs $(T^{(n)})_{n\geq 0}$, where
$T^{(n)}=\mathcal{R}^n(T)$ for $n\geq 0$. Denote by
$\pi^{(n)}=(\pi^{(n)}_0,\pi^{(n)}_1)\in\mathcal{S}^0_{\mathcal{A}}$
the pair and by
$\lambda^{(n)}=(\lambda^{(n)}_\alpha)_{\alpha\in\mathcal{A}}$ the
vector which determines $T^{(n)}$. Then $T^{(n)}$ is the first
return map of $T$ to the interval $I^{(n)}=[0,|\lambda^{(n)}|)$ and
\[\lambda=\Theta^{(n)}(T)\lambda^{(n)}\text{ with }\Theta^{(n)}(T)=
\Theta(T)\cdot\Theta(T^{(1)})\cdot\ldots\cdot\Theta(T^{(n-1)}).\]

\vskip 3mm
\subsection{IETs of periodic type}
\vskip 3mm

\begin{definition}[see \cite{Si-Ul}] An IET $T$ is of {\em periodic type}
if there exists $p>0$ (called a {\em period of $T$}) such that
$\Theta(T^{(n+p)})=\Theta(T^{(n)})$ for every $n\geq 0$ and
$\Theta^{(p)}(T)$ (called a {\em periodic matrix of $T$} and denoted
by $A$ in all that follows) has strictly positive entries.
\end{definition}

\begin{remark}
Suppose that $T=T_{(\pi,\lambda)}$ is of periodic type. It follows
that
\[\lambda=\Theta^{(pn)}(T)\lambda^{(pn)}=\Theta^{(p)}(T)^n\lambda^{(pn)}
\in\Theta^{(p)}(T)^n\R^{\mathcal{A}},\] and hence $\lambda$
belongs to $\bigcap_{n\geq 0}\Theta^{(p)}(T)^n\R^{\mathcal{A}}$
which is a one-dimensional convex cone (see \cite{Ve1}). Therefore
$\lambda$ is a positive right Perron-Frobenius eigenvector of the
matrix $\Theta^{(p)}(T)$. Since the set
$\mathcal{S}^0_{\mathcal{A}}$ is finite,  multiplying the period
$p$ if necessary, we can assume that $\pi^{(p)}=\pi$. It follows
that
$(\pi^{(p)},\lambda^{(p)}/|\lambda^{(p)}|)=(\pi,\lambda/|\lambda|)$
and $\rho:=|\lambda|/|\lambda^{(p)}|$ is the Perron-Frobenius
eigenvector of the matrix $\Theta^{(p)}(T)$. Recall that similar
arguments to those above show that every IET of periodic type is
uniquely ergodic.

\vskip 3mm A procedure giving an explicit construction of IETs of
periodic type was introduced in \cite{Si-Ul}. The construction is
based on choosing closed paths on the graph giving the Rauzy
classes. Every IET of periodic type can be obtained this way.
\end{remark}

\vskip 3mm Let $T=T_{(\pi,\lambda)}$ be an IET of periodic type and
$p$ be a period such that $\pi^{(p)}=\pi$. Let $A=\Theta^{(p)}(T)$.
By (\ref{omega}),
\[A^t\Omega_{\pi}A=\Omega_{{\pi}}\text{ and hence }
\ker\Omega_{{\pi}}=A\ker\Omega_{{\pi}}\text{ and
}H_{{\pi}}=A^tH_{{\pi}}.\] Multiplying the period $p$ if
necessary, we can assume that $A|_{\ker\Omega_{\pi}}=Id$ (see
Appendix~\ref{korekcja} for details). Denote by $Sp(A)$ the
collection of complex eigenvalues of $A$, including
multiplicities. Let us consider the collection of Lyapunov
exponents $\log|\rho|$, $\rho\in Sp(A)$. It consists of the
numbers
\[\theta_1>\theta_2\geq\theta_3\geq\ldots\geq\theta_g\geq0
=\ldots=0\geq-\theta_g\geq\ldots\geq-\theta_3\geq-\theta_2>-\theta_1,\]
where $2g=\dim H_{\pi}$ and $0$ occurs with the multiplicity
$\kappa-1=\dim\ker\Omega_{{\pi}}$ (see e.g.\ \cite{Zo0} and
\cite{Zo}). Moreover, $\rho_1:=\exp\theta_1$ is the
Perron-Frobenius eigenvalue of $A$. We will use sometimes the
symbol $\theta_i(T)$ instead of $\theta_i$ to emphasize that it is
associated to $T$.

\begin{definition}
An IET of periodic type $T_{(\pi,\lambda)}$  has {\em
non-degenerated spectrum} if $\theta_g>0$.
\end{definition}

\vskip 3mm \subsection{Growth of BV cocycles}\vskip 3mm

\hfill \break The recurrence of a cocycle $\varphi$ with values in
$\R^\ell$ is related to the growth of $\varphi^{(n)}$ when $n$
tends to $\infty$.

\vskip 3mm For an irrational rotation $T: x \rightarrow x + \alpha
{\rm \ mod \ }1$ (this can be viewed as an exchange of 2 intervals),
when $\varphi$ has a bounded variation, the growth of
$\varphi^{(n)}$ is controlled by the Denjoy-Koksma inequality: if
$\varphi$ is a zero mean function on $X = \R/\Z$ with bounded
variation $\var \varphi$, and $(q_n)$ the denominators (of the
convergents) given by the continued fraction expansion of $\alpha$,
then the following inequality holds:
\begin{eqnarray}
|\sum_{j=0}^{q_n-1} \varphi(x+ j \alpha)| \le \var \varphi,
\forall x \in X. \label{DK}
\end{eqnarray}
This inequality implies obviously recurrence of the cocycle
$\varphi^{(\,\cdot\,)}$ and if $\alpha$ has  bounded partial
quotients (we say for brevity {\it bpq})
$\sum_{j=0}^{n-1}\varphi(x+ j\alpha) = O(\log n)$ uniformly in
$x\in X$.

\vskip 3mm It is much more difficult to get a precise upper bound
for the growth of a cocycle over an IET. The following theorem
(proved in Appendix \ref{proofs-dev}) gives for an IET of periodic
type a control on the growth of a BV cocycle in terms of the
Lyapunov exponents of the matrix $A$.

\begin{theorem}\label{thmthetas}
Suppose that $T_{(\pi,\lambda)}:I\to I$ is an interval exchange
transformation of periodic type, $0\leq\theta_2<\theta_1$ are the
two largest Lyapunov exponents, and $M$ is the  maximal size of
Jordan blocks in the Jordan decomposition  of its periodic matrix
$A$. Then there exists $C>0$ such that
\[\|\varphi^{(n)}\|_{\sup}\leq
C\cdot\log^{M+1}n\cdot n^{\theta_2/\theta_1}\cdot\var\varphi\] for
every function $\varphi:I\to\R$ of bounded variation with zero
mean and for each natural $n$. \bez
\end{theorem}

For our purpose, this inequality is useful when
$\theta_2(T)/\theta_1(T)$ is small. In Appendix~\ref{theta1theta2}
we will give examples with arbitrary small values of this ratio.

\vskip 3mm
\subsection{Recurrence, essential values, and ergodicity of cocycles}
\vskip 3mm

\hfill \break In this subsection we recall some general facts about
cocycles. For relevant background material concerning skew products
and infinite measure-preserving dynamical systems, we refer the
reader to \cite{Sch} and \cite{Aa}.

\vskip 3mm Denote by $\overline{G}$ the one point compactification
of the group $G$. An element $g\in \overline{G}$ is said to be an
{\em essential value} of $\varphi$, if for every open neighbourhood
$V_g$ of $g$ in $\overline{G}$ and any set $B\in\mathcal{B}$,
$\mu(B)>0$, there exists $n\in\Z$ such that
\begin{eqnarray}
\mu(B\cap T^{-n}B\cap\{x\in X:\varphi^{(n)}(x)\in V_g\})>0.
\label{val-ess}
\end{eqnarray}

The set of essential values of $\varphi$ will be denoted by
$\overline{E}(\varphi)$. The set of finite essential values
${E}(\varphi):=G\cap\overline{E}(\varphi)$ is a closed subgroup of
$G$. We recall below some properties of $\overline{E}(\varphi)$ (see
\cite{Sch}).

\vskip 3mm Two cocycles $\varphi,\psi: X\to G$ are called {\em
cohomologous} for $T$ if there exists a measurable function
$g:X\to G$ such that $\varphi=\psi+g-g\circ T$. The corresponding
skew products $T_\varphi$ and $T_{\psi}$ are then
measure-theoretically isomorphic. A cocycle $\varphi:X\to G$ is a
{\em coboundary} if it is cohomologous to the zero cocycle.

\vskip 3mm If $\varphi$ and $\psi$ are cohomologous then
$\overline{E}(\varphi)=\overline{E}(\psi)$. Moreover, $\varphi$ is
a coboundary if and only if $\overline{E}(\varphi)=\{0\}$.

\vskip 3mm A cocycle $\varphi: X \to G$ is recurrent (as defined
in the introduction) if and only if, for each open neighborhood
$V_0$ of $0$, (\ref{val-ess}) holds for some $n\neq 0$. This is
equivalent to the conservativity of the skew product $T_\varphi$
(cf. \cite{Sch}). Let $\varphi:X\to\R^\ell$ be an integrable
function. If it is recurrent, then $\int_X\varphi\,d\mu=0$;
moreover, for $\ell=1$ this condition is sufficient for recurrence
when $T$ is ergodic.

\vskip 3mm The group ${E}(\varphi)$ coincides with the group of
{\em periods} of $T_{\varphi}$-invariant functions i.e.\ the set
of all $g_0\in G$ such that, if $f:X\times G\rightarrow\mathbb{R}$
is a $T_{\varphi}$-invariant measurable function, then
$f(x,g+g_0)=f(x,g)$ ${\mu}\times m_{G}$-a.e. In particular,
$T_{\varphi}$ is ergodic if and only if $E(\varphi)=G$.

\vskip 3mm A simple sufficient condition of recurrence is the
following:
\begin{proposition}[see Corollary 1.2 in
\cite{Ch-Co}]\label{cochco} If $\varphi:X\to\R^\ell$ is a square
integrable cocycle for an automorphism $T:\xbm\to\xbm$ such that
$\|\varphi^{(n)}\|_{L^2(\mu)}=o(n^{1/\ell})$, then it is
recurrent.
\bez
\end{proposition}

In view of Theorem \ref{thmthetas}, as a consequence we have the
following.

\begin{corollary}\label{correc}
If $T:I\to I$ is an IET of periodic type such that
$\theta_2(T)/\theta_1(T)<1/\ell$ for an integer $\ell \geq 1$, then
every cocycle $\varphi:I\to\R^\ell$ over $T$ of bounded variation
with zero mean is recurrent. If, for $j=1,\ldots,\ell$,
$T_j:I^{(j)}\to I^{(j)}$ are interval exchange transformations of
periodic type such that $\theta_2(T_j)/\theta_1(T_j)<1/\ell $, then
every "product" cocycle $\varphi = (\varphi_1,\ldots ,
\varphi_\ell): I^{(1)}\times\ldots\times I^{(\ell)}\to\R^\ell$ of
bounded variation with zero mean over $T_1 \times ... \times T_\ell$
is recurrent. \bez
\end{corollary}

We continue these preliminaries by some useful observations for
proving the ergodicity of cocycles. Let $(X,d)$ be a compact metric
space. Let $\mathcal{B}$ stand for the $\sigma$--algebra of all
Borel sets and let $\mu$ be a probability Borel measure on $X$. By
$\chi_B$ we will denote the indicator function of a set $B$. Suppose
that $T:\xbm\to\xbm$ is an ergodic measure--preserving automorphism
and there exist an increasing sequence of natural numbers $(q_n)$
and a sequence of Borel sets $(C_n)$ such that
\[\mu(C_n)\to\alpha>0,\;\;\mu(C_n\triangle T^{-1}C_n)\to 0\;\;\mbox{ and }
\sup_{x\in C_n}d(x,T^{q_n}x)\to 0.\] Assume that $G\subset
\R^\ell$ for some $\ell\geq 1$. Let $\varphi:X\to G$ be a Borel
integrable cocycle for $T$ with zero mean. Suppose that the
sequence $(\int_{C_n}|\varphi^{(q_n)}(x)|d\mu(x))_{n \geq 1}$ is
bounded. As the distributions
\[(\mu(C_n)^{-1}(\varphi^{(q_n)}|_{C_n})_*(\mu|_{C_n}),n\in\N)\]
are uniformly tight, by passing to a further subsequence if
necessary we can assume that they converge weakly to a probability
Borel measure $P$ on $G$.

\begin{lemma}
The topological support of the measure $P$ is included in the
group $E(\varphi)$ of essential values  of the cocycle $\varphi$.
\end{lemma}
\begin{proof}
Suppose that $g\in\supp(P)$. Let $V_g$ be an open neighborhood of
$g$. Let $\psi:G\to[0,1]$ be a continuous function such that
$\psi(g)=1$ and $\psi(h)=0$ for $h\in G\setminus V_g$. Thus
$\int_{G}\psi(g)\,dP(g)>0$. By Lemma~5 in \cite{Fr-Le2}, for every
$B\in\mathcal{B}$ with $\mu(B)>0$ we have
\begin{align*}
\mu(&B\cap T^{-q_n}B\cap(\varphi^{(q_n)}\in V_g))
\geq \int_{C_n}\psi\left(\varphi^{(q_n)}(x)\right)\chi_B(x)\chi_B(T^{q_n}x)\,d\mu(x)\\
 &\to
\alpha\int_{X}\int_{G}\psi(g)\chi_B(x)\,dP(g)\,d\mu(x)=\alpha\mu(B)\int_{G}\psi(g)\,dP(g)>0,
\end{align*}
and hence $g\in E(\varphi)$.
\end{proof}

\begin{corollary}[see also \cite{Co}]\label{ergodic}
If $\varphi^{(q_n)}(x)=g_n$ for all $x\in C_n$ and $g_n\to g$, then
$g\in E(\varphi)$.
\end{corollary}

\vskip 3mm
\begin{proposition}[see Proposition 3.8 in \cite{Sch}]\label{compact}
Let $T:\xbm\to\xbm$ be an ergodic automorphism and let $\varphi:X\to
G$ be a measurable cocycle for $T$. If $K\subset G$ is a compact set
such that $K\cap E(\varphi)=\emptyset$, then there exists
$B\in\mathcal{B}$ such that $\mu(B)>0$ and
\[\mu(B\cap T^{-n}B\cap(\varphi^{(n)}\in K))=0\text{ for every }n\in\Z.\]
\end{proposition}

\begin{lemma}\label{lemsym}
Let $K\subset G$ be a compact set. If for every $B\in\mathcal{B}$
with $\mu(B)>0$ and every neighborhood $V_0 \subset G$ of zero there
exists $n\in\Z$ such that
\[\mu(B\cap T^{-n}B\cap(\varphi^{(n)}\in K+V_0))>0,\]
then $K\cap E(\varphi)\neq\emptyset$. In particular, when $K =
\{g,-g\}$, where $g$ is an element of $G$, then $g\in E(\varphi)$.
\end{lemma}

\begin{proof}
Suppose that $K\cap E(\varphi)=\emptyset$. Since $K$ is compact and
$E(\varphi)$ is closed, there exists a neighborhood $V_0$ of zero
such that $\overline{V_0}$ is compact and $(K+\overline{V_0})\cap
E(\varphi)=\emptyset$. As $K+\overline{V_0}$ is also compact, by
Proposition~\ref{compact}, there exists $B\in\mathcal{B}$ such that
$\mu(B)>0$ and
\[\mu(B\cap T^{-n}B\cap(\varphi^{(n)}\in (K+\overline{V_0})))=0
\text{ for every }n\in\Z, \] contrary to assumption. The last claim
is clear.
\end{proof}

Consider the quotient cocycle $\varphi^*:X\to G/{E}(\varphi)$
given by $\varphi^*(x)=\varphi(x)+{E}(\varphi)$. Then
$E(\varphi^*)=\{0\}$. The cocycle $\varphi$ is called {\em
regular} if $\overline{E}(\varphi^*)=\{0\}$ and  {\em
non--regular} if $\overline{E}(\varphi^*)=\{0,\infty\}$. Recall
that if $\varphi$ is regular then it is cohomologous to a cocycle
$\psi:X\to E(\varphi)$ such that $E(\psi)=E(\varphi)$.

\begin{lemma}\label{lemergo} If $H$ is a closed subgroup of $E(\varphi)$
such that the quotient cocycle $\varphi_H:X\to G/H$,
$\varphi_H(x)=\varphi(x)+H$ is ergodic, then $\varphi:X\to G$ is
ergodic as well.
\end{lemma}

\begin{proof} Let $f(x,g)$ be a measurable $T_\varphi$-invariant
function. Then, since $H\subset E(\varphi)$, $f$ is $H$-invariant.
Since $\varphi_H$ is ergodic, $f$ is constant.
\end{proof}

\vskip 6mm
\section{Ergodicity of piecewise linear cocycles \label{piecewise}}
\vskip 3mm

\begin{notations} We denote by $\bv(\sqcup_{\alpha\in \mathcal{A}}
I^{(k)}_{\alpha})$ the space of functions $\varphi:I^{(k)}\to
\mathbb{R}$ such that the restriction $\varphi:I^{(k)}_{\alpha}\to
\mathbb{R}$ is of bounded variation for every $\alpha\in
\mathcal{A}$, and by $\bv_0(\sqcup_{\alpha\in \mathcal{A}}
I^{(k)}_{\alpha})$ the subspace of functions in
$\bv(\sqcup_{\alpha\in \mathcal{A}} I^{(k)}_{\alpha})$ with zero
mean. We adopt the notation from \cite{Ma-Mo-Yo}. The space
$\bv(\sqcup_{\alpha\in \mathcal{A}} I^{(k)}_{\alpha})$ is equipped
with the norm $\|\varphi\|_{\bv}=\|\varphi\|_{\sup}+\var\varphi$,
where
\[\var \varphi=\sum_{\alpha\in \mathcal{A}}\var \varphi|_{I^{(k)}_ {\alpha}}.\]

For $\varphi\in\bv(\sqcup_{\alpha\in \mathcal{A}} I_{\alpha})$ and
$x\in I$, $\varphi_+(x)$ and $\varphi_-(x)$ denote the right-handed
and left-handed limit of $\varphi$ at $x$ respectively. We denote by
$\bv^1(\sqcup_{\alpha\in \mathcal{A}} I_{\alpha})$ the space of
functions $\varphi:I\to \mathbb{R}$ which are absolutely continuous
on each  $I_{\alpha}$, $\alpha\in \mathcal{A}$ and such that
$\varphi'\in \bv(\sqcup_{\alpha\in \mathcal{A}} I_{\alpha})$. For
$\varphi\in\bv^1(\sqcup_{\alpha\in \mathcal{A}} I_{\alpha})$ let
\[s(\varphi)=\int_I\varphi'(x)\,dx=\sum_{\alpha\in\mathcal{A}}
(\varphi_{-}(r_\alpha)-\varphi_{+}(l_\alpha)).\]

We denote by $\bv_*^1(\sqcup_{\alpha\in \mathcal{A}} I_{\alpha})$
the subspace of functions $\varphi \in \bv^1(\sqcup_{\alpha\in
\mathcal{A}} I_{\alpha})$ for which $s(\varphi)=0$, and by
$\pl(\sqcup_{\alpha\in \mathcal{A}} I_{\alpha})$ the set of
piecewise linear (with constant slope) functions $\varphi:I\to\R$
such that $\varphi(x)=sx+c_{\alpha}$ for $x\in I_{\alpha}$.
\end{notations}

\begin{proposition}[see \cite{Ma-Mo-Yo}]\label{cohcon}
If $T:I\to I$ satisfies a Roth type condition, then each cocycle
$\varphi\in \bv_*^1(\sqcup_{\alpha\in \mathcal{A}} I_{\alpha})$ for
$T$ is cohomologous to a cocycle which is constant on each interval
$I_{\alpha}$, $\alpha\in \mathcal{A}$. Moreover, the set of IETs
satisfying this Roth type condition has full measure and contains
all IETs of periodic type.
\end{proposition}

As a consequence of Proposition~\ref{cohcon} we have the following.

\begin{lemma}\label{lempl}
If $T:I\to I$ is of periodic type, then each cocycle
$\varphi\in\bv^1(\sqcup_{\alpha\in \mathcal{A}} I_{\alpha})$ is
cohomologous to a cocycle $\varphi_{pl}\in\pl(\sqcup_{\alpha\in
\mathcal{A}} I_{\alpha})$ with $s({\varphi}_{pl})=s(\varphi)$.
\end{lemma}

\subsection{Piecewise linear cocycles}
\vskip 3mm

\hfill \break Now we will focus on the case where the slope of a
piecewise linear cocycle is non-zero and show ergodicity. We will
need an information on the distribution of discontinuities of
$\varphi^{(n)}$.

Let $T:I\to I$ be an arbitrary IET satisfying Keane's condition.
Denote by $\mu$ the Lebesgue measure on $I$. Each finite subset
$D\subset I$ determines a partition $\cp(D)$ of $I$ into left-closed
and right-open intervals. Denote by $\min\cp(D)$ and $\max\cp(D)$
the length of the shortest and the longest interval of the partition
$\cp(D)$ respectively. For every $n\geq 0$ let $\cp_n(T)$ stand for
the partition given by the subset
$\{T^{-k}l_\alpha:\alpha\in\mathcal{A},0\leq k<n\}$. Then $T^n$ is a
translation on each interval of the partition $\mathcal{P}_n(T)$.
The following result shows that the discontinuities for iterations
of IETs of periodic type are well distributed.

\begin{proposition}[see \cite{Ku}]\label{kulaga}
For every IET $T$ of periodic type there exists $c\geq 1$ such
that for every $n\geq 1$ we have
\begin{equation}\label{condist}
\frac{1}{cn}\leq \min\mathcal{P}_n(T) \leq \max\mathcal{P}_n(T)
\leq\frac{c}{n}
\end{equation}
\end{proposition}

We begin by a preliminary result which will be proved later in a
general version (see Theorem~\ref{kawallin2} and \ref{kawallin1} for
$\ell=1$).

\begin{theorem}\label{kawallin}
Let $T:I\to I$ be an IET of periodic type. If
$\varphi\in\pl(\sqcup_{\alpha\in \mathcal{A}} I_{\alpha})$ is a
piecewise linear cocycle with zero mean and $s(\varphi)\neq 0$, then
the skew product $T_{\varphi}$ is ergodic.
\end{theorem}

Now we consider cocycles taking values in $\R^\ell$, $\ell\geq 1$.
Suppose that $\varphi:I\to\R^\ell$ is a piecewise linear cocycle
with zero mean such that the slope $s(\varphi)\in\R^\ell$ is
non-zero. Then, by an appropriate choice of coordinates,  we obtain
$s(\varphi_1)\neq 0$ and $s(\varphi_2)=0$, where
$\varphi=(\varphi_1,\varphi_2)$ and $\varphi_1:I\to\R$,
$\varphi_2:I\to\R^{\ell-1}$. Thus $\varphi_2$ is piecewise constant
and, roughly speaking, the ergodicity of $\varphi_2$ implies the
ergodicity of $\varphi$. The ergodicity of piecewise constant
cocycles will be studied in Sections~\ref{secstepcoc} and
\ref{correct}.

\begin{theorem}\label{kawallin2}
Suppose that $T:I\to I$ is an IET of periodic type such that
$\theta_2(T)/\theta_1(T)<1/\ell$. Let
$\varphi_1\in\pl(\sqcup_{\alpha\in \mathcal{A}} I_{\alpha},\R)$,
$\varphi_2\in\pl(\sqcup_{\alpha\in \mathcal{A}}
I_{\alpha},\R^{\ell-1})$ be piecewise linear cocycles with zero mean
such that $s(\varphi_1)\neq 0$ and $s(\varphi_2)=0$. If the cocycle
$\varphi_2:I\to\R^{\ell-1}$ is ergodic, then the cocycle
$\varphi=(\varphi_1,\varphi_2):I\to\R^\ell$ is ergodic as well.
\end{theorem}

\begin{proof}
Without loss of generality we can assume that $s(\varphi_1)= 1$. It
suffices to show that for every $0 <a < {1 \over 4c}$, the pair
$(a,0)$ belongs to $E(\varphi_1,\varphi_2)$. Indeed this implies
that $\R\times\{0\}\subset E(\varphi_1,\varphi_2)$, and since the
cocycle $\varphi_2$ is ergodic, by Lemma~\ref{lemergo}, it follows
that $(\varphi_1,\varphi_2):I\to\R^\ell$ is ergodic as well.

\vskip 3mm Fix $0 < a < {1 \over 4c}$, where $c$ is given by
Proposition \ref{kulaga}. By a density point argument, for every
measurable $B\subset I$ with $\mu(B)>0$ and every $\varepsilon \in
(0, {a \over 2})$, there are $B'\subset B$ with $\mu(B')>0$ and
$n_0\geq 1$ such that for $n \geq n_0$,
\begin{equation}\label{densarg1}
\mu\left(\left(x-\frac{c}{n},x+\frac{c}{n}\right)\setminus
B\right)<\frac{\vep}{n}\text{ for every }x\in B'.
\end{equation}
Since $\theta_2(T)/\theta_1(T)<1/\ell$, by Corollary~\ref{correc},
$(\varphi_1,\varphi_2)$ is recurrent, and hence there exists
$n\geq n_0$ such that
\[\mu(B'\cap T^{-n}B'\cap(|\varphi_1^{(n)}|<\vep)\cap(\|\varphi_2^{(n)}\|<\vep))>0.\]
Let $x_0 \in I$ be such that $x_0,T^nx_0\in B'$,
$|\varphi_1^{(n)}(x_0)|<\vep$ and $\|\varphi_2^{(n)}(x_0)\|<\vep$.
Denote by $J(x_0)\subset I$ the interval of the partition
$\mathcal{P}_n(T)$ which contains $x_0$. Then $\varphi_1^{(n)}$ is
a linear function on $J(x_0)$ with slope $n$. Since
$2\vep<a<1/(2c)-2\vep$ and $|J(x_0)|>1/(cn)$ (by (\ref{condist})),
there exists $y_0$ such that $(y_0-\vep/n,y_0+\vep/n)\subset
J(x_0)$ and
\[|\varphi_1^{(n)}(y)|\in a+(-\vep,\vep)\text{ for all } y\in
(y_0-\vep/n,y_0+\vep/n).\] Since $\varphi_2^{(n)}$ is constant on
$J(x_0)$, we have
\[\|\varphi_2^{(n)}(x)\|<\vep\text{ for all }x\in(y_0-\vep/n,y_0+\vep/n).\]
Therefore
\begin{align}\label{kakaka}
\begin{split}
&\mu\left(B\cap
T^{-n}B\cap(\varphi_1^{(n)}\in\{-a,a\}+(-\vep,\vep))
\cap(\varphi_2^{(n)}\in(-\vep,\vep)^{\ell-1})\right)\\
&\;\;\;\;\;\;\;\geq
\mu\left(\left(y_0-{\vep}/{n},y_0+{\vep}/{n}\right)\cap B\cap
T^{-n}B\right).
\end{split}
\end{align}
By (\ref{condist}) we have $|J(x_0)|<{c}/{n}$, and hence
$J(x_0)\subset(x_0-c/n,x_0+c/n)$. Moreover, $T^nJ(x_0)$ is an
interval such that $|T^nJ(x_0)|=|J(x_0)|<c/n$, so that
\begin{equation*}
T^nJ(x_0)\subset\left(T^nx_0-\frac{c}{n},T^nx_0+\frac{c}{n}\right).
\end{equation*}
Since $x_0,T^nx_0\in B'$, by (\ref{densarg1}), $\mu(J(x_0)\setminus
B)<\vep/n$ and $\mu(T^nJ(x_0)\setminus B)<\vep/n$. Therefore,
$\mu(J(x_0)\setminus(B\cap T^{-n}B))<2\vep/n$, and hence
\[\mu\left((y_0-{\vep}/{n},y_0+{\vep}/{n})\setminus(B\cap T^{-n}B)\right)<2\vep/n.\]
Thus
\[\mu\left((y_0-{\vep}/{n},y_0+{\vep}/{n})\cap B\cap T^{-n}B\right)>0.\]
In view of (\ref{kakaka}), it follows that
\[\mu\left(B\cap
T^{-n}B\cap(\varphi_1^{(n)}\in\{-a,a\}+(-\vep,\vep))\cap(\varphi_2^{(n)}
\in(-\vep,\vep)^{\ell-1})\right)>0.\] By Lemma~\ref{lemsym}, we
conclude that $(a,0)\in E(\varphi_1,\varphi_2)$, which completes the
proof.
\end{proof}

\subsection{Product cocycles}\vskip 3mm

\hfill \break The method used in Theorem~\ref{kawallin} allows us to
prove the ergodicity for Cartesian products of certain skew
products. As an example, first we apply this method for cocycles
taking values in $\Z$ over irrational rotations on the circle. This
will give a class of ergodic $\Z^2$-cocycles driven by 2-dimensional
rotations

\vskip 3mm Let $T(x,y)=(x+\alpha_1, x+\alpha_2)$ be a 2-dimensional
rotation and $\varphi$ be a zero mean function on $\T^2$ of the form
$\varphi(x,y) = (\varphi_1(x), \varphi_2(y))$ with $\varphi_1$ and
$\varphi_2$ BV functions. If $\alpha_1$ and $\alpha_2$ have bounded
partial quotients, then (\ref{DK}) implies
$\|\varphi^{(n)}\|_{\sup}=O(\log n)$, and therefore, by Proposition
\ref{cochco}, the cocycle $\varphi$ is recurrent.

Consider the function $\varphi(x,y) = (2\cdot\chi_{[0,{\frac12})}(x)
- 1,2\cdot\chi_{[0, {\frac12})}(y) - 1)$ or more generally assume
that $\varphi_i$, $i=1,2$, are step functions in one variable with
values in $\Z$. For $i=1,2$, we denote by $D_i\subset\T$ the finite
set of discontinuities of $\varphi_i$ and by $J_i\subset\Z$ the
corresponding set of jumps of the functions $\varphi_i$. It defines
a recurrent $\Z^2$-cocycle driven by a 2-dimensional rotation. A
question is then the ergodicity (with respect to the measure $\mu
\times m$ the product of the uniform measure on $\T^2$ by the
counting measure on $\Z^2$) of the skew-product
$$T_\varphi:\T^2\times\Z^2\to\T^2\times\Z^2,\;\;\;T_\varphi
(x, y, \bar{n})= (x+\alpha_1,y+\alpha_2, \bar{n}+\varphi (x,y)).$$

\begin{theorem} \label{2dim} Let $\alpha_1$ and $\alpha_2$ be two rationally
independent irrational bpq numbers, and let $\varphi(x,y) =
(\varphi_1(x), \varphi_2(y))$ be a function on the torus with step
functions components $\varphi_i:\T\to\Z$, $i=1,2$, such that
$D_1,D_2\subset\Q$ and the sets of the jumps $J_1\times\{0\}$,
$\{0\}\times J_2$ generate $\Z^2$. Then the system $(\T^2 \times
\Z^2, \mu \times m, T_\varphi)$ is ergodic.
\end{theorem}
\begin{proof} \ We have seen that the cocycle $\varphi^{(n)}$ is recurrent.
We prove that the group of its finite essential values is $\Z^2$.

Let $n$ be a fixed integer and let $(\gamma_{n,k}^i)_{k= 1,\ldots,
d_in}$ be the ordered set of the $d_in$ discontinuities of
$\varphi_i^{(n)}$ in $[0, 1)$ (where $d_i:=\# D_i$). In the sequence
of denominators of $\alpha_i$, let $q_{r_i(n)}^i$ be such that
$q_{r_i(n)}^i \leq n < q_{r_i(n)+1}^i$. We write simply $q_{r_i}^i$
for $q_{r_i(n)}^i$. As $\alpha_i$ is bpq, the ratio $q_{r_i+1}^i/
q_{r_i}^i$ is bounded by a constant independent from $n$.

Since $\alpha_i$ is bpq and the discontinuity points of $\varphi_i$
are rational, the distances between consecutive discontinuities of
$\varphi^{(n)}$ are of the same order: there are two positive
constants $c_1,c_2$ such that
\begin{equation}\label{defcs}
{\frac{c_1}{n}} \leq \gamma_{n, k+1}^i - \gamma_{n,k}^i \leq
\frac{c_2}{ n}, \ k= 1,\ldots, d_in, \;i=1,2.
\end{equation}
Recall that, for each $t \in D_i$ and each $0 \leq \ell < q_{r_i}$,
there is (mod 1) a point $t- k \alpha_i$, $0 \leq k < q_{r_i}$ in
each interval $[t+\ell/q_{r_i}^i, t+(\ell+1) /q_{r_i}^i]$.
Therefore, in each interval of length greater than $2 /q_r^i$ and
for each $t \in D_i$, there is at least one discontinuity of
$\varphi_i^{(n)}$ of the form $t-k\alpha_i$, $0\leq k<n$.

It implies that if we move a point $x$ on the unit interval by a
displacement greater than $2 /q_{r_i}^i$, we cross discontinuities
of $\varphi_i^{(n)}$ corresponding to each different discontinuity
$t \in D_i$ of $\varphi_i$.

For $x \in \T$, consider the interval $[\gamma_{n,k}^i,
\gamma_{n,k+1}^i)$ which contains $x$ and denote it by $I_n^i(x)$.
The intervals $[\gamma_{n, k+\ell}^i, \gamma_{n, k+\ell+1}^i)$,
where $k+\ell$ is taken mod $d_1n$, are denoted by
$I_{n,\ell}^i(x)$. This gives two collections of rectangles
\[R_{k,\ell}^n(x,y): = I_{n,k}^1(x) \times I_{n,\ell}^2(y)\text{ and
}\tilde R_{k,\ell}^n(x,y) := T^nR_{k,\ell}^n(T^{-n}(x,y))\] for each
$(x,y)\in\T^2$. By (\ref{defcs}), we have
\begin{equation}\label{boundmeas}
\mu(R_{k,\ell}^n(x,y)) \  \text{ and } \ \mu(\tilde
R_{k,\ell}^n(x,y)) \in \left[\frac{c_1^2}{n^2},
\frac{c_2^2}{n^2}\right].
\end{equation}
Let $M$ be a natural number such that $c_1M>1$. Then, by
(\ref{defcs}), the length of $\bigcup_{k=-M}^M I_{n,k}^i(x)$ is
greater than $ 2/q_{r_i}^i$, $i=1,2$.  Let $\delta > 0$ be such
that $\delta c^2 (2M+1)^2 < {1/2}$ with $c=c_2^2/c_1^2$. Set
\begin{eqnarray*} R_M^n(x,y) &:=& \bigcup_{k=-M}^{k=M}
\bigcup_{\ell=-M}^{\ell=M} R_{k,\ell}^n(x,y), \\
\tilde R_M^n(x,y) &:=&
T^nR_{M}^n(T^{-n}(x,y))=\bigcup_{k=-M}^{k=M}
\bigcup_{\ell=-M}^{\ell=M} \tilde R_{k,\ell}^n(x,y).
\end{eqnarray*}
In view of (\ref{defcs}),
\begin{equation}\label{rectangles}
\frac{\operatorname{length}(R_M^n(x,y))}
{\operatorname{width}(R_M^n(x,y))} \  \text{ and } \
\frac{\operatorname{length}(\tilde{R}_M^n(x,y))}
{\operatorname{width}(\tilde{R}_M^n(x,y))} \in
\left[\frac{c_1}{c_2}, \frac{c_2}{c_1}\right].
\end{equation}
The cocycle $\varphi^{(n)}$ has a constant value on each rectangle
$R_{k,\ell}^n(x,y)$ and the difference between its value on
$R_{k+1,\ell}^n(x,y)$ and $R_{k,\ell}^n(x,y)$ (resp.
$R_{k,\ell+1}^n(x,y)$ and $R_{k,\ell}^n(x,y)$) belongs to
$J_1\times\{0\}$ (resp. $\{0\}\times J_2$). Denote by
$\kappa^n_{k,\ell}(x,y)$ the value of $\varphi^{(n)}$ on
$R_{k,\ell}^n(x,y)$. Since the length of $R_M^n(x,y)$ is greater
than $2/q^1_{r_1}$ and the width of $R_M^n(x,y)$ is greater than
$2/q^2_{r_2}$ we have
\begin{eqnarray}\label{kappas}
\begin{aligned}
&\{\kappa^n_{k+1,\ell}(x,y)-\kappa^n_{k,\ell}(x,y):-M\leq k<M\}=
J_1\times\{0\}, \\
&\{\kappa^n_{k,\ell+1}(x,y)-\kappa^n_{k,\ell}(x,y):-M\leq l<M\}=
\{0\}\times J_2. \end{aligned}\end{eqnarray}

Let
\[K:=\overbrace{(J_1\cup\{0\}+\ldots+J_1\cup\{0\})}^{M}
\times\overbrace{(J_2\cup\{0\}+\ldots+J_2\cup\{0\})}^{M}.\] Let
$K_1$ be the subset of elements of $K$ which are not essential
values of $\varphi$, and suppose $K_1 \not = \emptyset$. By
Proposition \ref{compact}, there exists $B\subset\T^2$ such that
$\mu(B)>0$ and
\begin{equation}\label{absessval}
\mu(B\cap T^{-n}B\cap(\varphi^{(n)}\in K_1))=0\text{ for every
}n\in\Z.
\end{equation} Since the areas of $R_M^n(x,y)$, $\tilde
R_M^n(x,y)$ tend to $0$ as $n\to\infty$ and the rectangles satisfy
(\ref{rectangles}), by a density point argument, there is a Borel
subset $B'$ of $B$ of positive measure and there is $n_0 \in \N$
such that for $n \geq n_0$ and $(x,y) \in B'$:
\[
\mu(B \cap R_M^n(x,y)) \geq (1 - \delta) \mu(R_M^n(x,y)), \ \mu(B
\cap \tilde R_M^n(x,y)) \geq (1 - \delta) \mu(\tilde R_M^n(x,y)).
\]
By (\ref{boundmeas}), the areas of the small rectangles being
comparable, and hence
$$\mu(R_M^n(x,y)) \leq (2M+1)^2 c^2 \mu(R_{k,\ell}^n(x,y))\text{ for all }k,\ell \in [-M, M].$$
Therefore, by the choice of $\delta$, for  each $(x,y) \in B'$ we
have
\begin{multline*} \mu(B \cap R_{k,\ell}^n(x,y)) \ge \mu(R_{k,\ell}^n(x,y)) -
\mu(B^c \cap R_{k,\ell}^n(x,y)) \\
\begin{aligned}
&\ge \mu(R_{k,\ell}^n(x,y)) - \mu(B^c \cap R_{M}^n(x,y)) \ge
\mu(R_{k,\ell}^n(x,y)) - \delta \mu(R_M^n(x,y))
\\ &\ge \mu(R_{k,\ell}^n(x,y))  - \delta (2M+1)^2 c^2
\mu(R_{k,\ell}^n(x,y))
> {\frac12} \mu(R_{k,\ell}^n(x,y)).
\end{aligned}
\end{multline*}
In the same way, if $ T^n (x,y) \in B'$, then $\mu(B \cap \tilde
R_{k,\ell}^n(T^n(x,y))) > {\frac12} \mu( \tilde
R_{k,\ell}^n(T^n(x,y)))$. Since $\tilde
R_{k,\ell}^n(T^n(x,y))=T^nR_{k,\ell}^n(x,y)$, we have
$$\mu( T^{-n} B \cap R_{k,\ell}^n(x,y)) > {\frac12}
\mu(R_{k,\ell}^n(x,y)).$$ The preceding inequalities imply
\begin{eqnarray}
\mu( B \cap  T^{-n} B \cap R_{k,\ell}^n(x,y)) > 0, \ \forall
k,\ell \in [-M, M]. \label{intersec}
\end{eqnarray}
By the recurrence property, there is $n > n_0$ such that
$$\mu(B'
\cap  T^{-n} B' \cap \{ \varphi^{(n)}(\,\cdot\,) = (0,0)\})
>0.$$
If $(x,y)\in B' \cap  T^{-n} B' \cap \{ \varphi^{(n)}(\,\cdot\,) =
(0,0)\}$, then $\varphi^{(n)}$ is equal to $(0,0)$ on
$R_{0,0}^n(x,y)$. Moreover, on each rectangle $R_{k,\ell}^n(x,y)$,
$k,\ell \in [-M, M]$, the cocycle $\varphi^{(n)}$ is constant and is
equal to $\kappa_{k,\ell}(x,y)\in K$. In view of (\ref{intersec}),
it follows that
\[\mu( B \cap  T^{-n} B \cap \{ \varphi^{(n)}(\,\cdot\,) = \kappa_{k,\ell}(x,y)\}) > 0, \ \forall
k,\ell \in [-M, M].\] Therefore, by (\ref{absessval}) and the
definition of $K_1$, $\kappa_{k,l}(x,y)\not \in K_1$, and so it
belongs to $E(\varphi)$ for all $k,\ell \in [-M, M]$. In view of
(\ref{kappas}), it follows that $J_1\times\{0\},\{0\}\times
J_2\subset E(\varphi)$, and hence $E(\varphi)=\Z^2$.
\end{proof}

\begin{remark}
The ergodicity of $T_\varphi$ can be proven also for the more
general case where $\alpha_i$ is bpq and
$(D_i-D_i)\setminus\{0\}\subset
(\Q+\Q\alpha_i)\setminus(\Z+\Z\alpha_i)$ for $i=1,2$. To extend
the result of Theorem \ref{2dim}, we use that the discontinuities
of the cocycle are "well distributed" (the condition
(\ref{defcs})) which is a consequence of Lemma~2.3 in
\cite{Fr-Le-Le}.
\end{remark}

Now by a similar method we show the ergodicity of Cartesian
products of skew products that appeared in Theorem~\ref{kawallin}.
We need an elementary algebraic result:

\begin{remark}\label{macierz} Let $R$ be a real  $m\times k$--matrix.
Then the subgroup $R(\Z^k)$ is dense in $\R^m$ if and only if
\begin{eqnarray}
\forall a \in \R^m, \ R^t(a) \in \Z^k \Rightarrow a= 0.
\label{denseRm}
\end{eqnarray}
For instance, if $R=[r_{ij}]$ is a $m\times (m+1)$--matrix such
that $r_{ij}=\pm\delta_{ij}$ for $1\leq i,j\leq m$ and
$1,r_{1\,m+1},\ldots,r_{m\,m+1}$ are independent over $\Q$, then
(\ref{denseRm}) holds.
\end{remark}

\begin{theorem}\label{kawallin1}
Let $T_j:I^{(j)}\to I^{(j)}$ be an interval exchange transformation
of periodic type such that $\theta_2(T_j)/\theta_1(T_j)<1/\ell $ for
$j=1,\ldots,\ell $. Suppose that the Cartesian product
$T_1\times\ldots\times T_\ell$ is ergodic. If
$\varphi_j\in\pl(\sqcup_{\alpha\in \mathcal{A_j}} I^{(j)}_{\alpha})$
is a piecewise linear cocycle with zero mean and $s(\varphi_j)\neq
0$ for $j=1,\ldots,\ell $, then the Cartesian product
$(T_1)_{\varphi_1}\times \ldots\times(T_\ell)_{\varphi_\ell}$ is
ergodic.
\end{theorem}

\begin{proof}
Since $T_1,\ldots,T_\ell$ have periodic type, by Lemma~\ref{kulaga}
there exists $c>0$ such that
\begin{equation}\label{niernac}
\frac{1}{cn}\leq \min\mathcal{P}_n(T_j) \leq
\max\mathcal{P}_n(T_j)\leq\frac{c}{n}\text{ for all
}j=1,\ldots,\ell \text{ and }n>0.
\end{equation} Let
$\bar{I}=I^{(1)}\times\ldots\times I^{(l)}$,
$\bar{T}=T_1\times\ldots\times T_\ell$ and let
$\bar{\varphi}:\bar{I}\to\R^\ell $ be given by
$$\bar{\varphi}(x_1,\ldots,x_\ell)=(\varphi_1(x_1),\ldots,\varphi_\ell(x_\ell)).$$
Then $(T_1)_{\varphi_1}\times
\ldots\times(T_\ell)_{\varphi_\ell}=\bar{T}_{\bar{\varphi}}$. Denote
by $\bar{\mu}$ the Lebesgue measure on $\bar{I}$. Without loss of
generality we can assume that $s(\varphi_j)=\pm 1$ for
$j=1,\ldots,\ell $. By Corollary~\ref{correc}, the cocycle
$\bar{\varphi}$  for $\bar{T}$ is recurrent.

\vskip 3mm To prove the result, it suffices to show that, for every
$r=(r_1,\ldots,r_\ell)\in[0, {1 \over 4c})^\ell$, the set
$E(\bar{\varphi})$ has nontrivial intersection with
\[\{s\bullet r:=(s_1r_1,\ldots,s_\ell r_\ell):s=(s_1,\ldots,s_\ell)\in\{-1,1\}^\ell \}.\]
Indeed, for a fixed rational $0<r< {1 \over 4c}$, let us consider
a collection of vectors $r^{(i)}=(r_{1i},\ldots,r_{\ell i})\in
[0,1/(4c))^\ell $, $1\leq i\leq\ell +1$   such that
$r_{ij}=r\delta_{ij}$ for all $1\leq i,j\leq\ell $ and
$1,r_{1\,\ell +1},\ldots,r_{\ell \,\ell +1}$ are independent over
$\Q$. By Remark~\ref{macierz}, for any choice $s^{(i)}\in
\{-1,1\}^\ell $, $1\leq i\leq\ell +1$ the subgroup generated by
vectors $s^{(i)}\bullet r^{(i)}$, $1\leq i\leq\ell +1$ is dense in
$\R^\ell $. Since $E(\bar{\varphi})\subset\R^\ell $ is a closed
subgroup and for every $1\leq i\leq\ell +1$ there exists
$s^{(i)}\in\{-1,1\}^\ell $ such that $s^{(i)}\bullet r^{(i)}\in
E(\varphi)$, it follows that $E(\bar{\varphi})=\R^\ell $, and
hence $\bar{T}_{\bar{\varphi}}$ is ergodic.

\vskip 3mm Fix  $r=(r_1,\ldots,r_\ell)\in [0,{1 \over 4c})^\ell $.
We have to show that for every measurable set $B\subset \bar{I}$
with $\bar{\mu}(B)>0$ and $0<\vep<1/c$ there exists $n>0$ such that
the set of all $\bar{x}=(x_1,\ldots,x_\ell)\in B$ such that
\[(T_1^nx_1,\ldots,T_\ell^nx_\ell)\in B,\;\varphi_j^{(n)}(x_j)\in
\{-r_j,r_j\}+(-\vep,\vep)\text{ for  } 1\leq j\leq\ell \] has positive $\bar{\mu}$ measure.
By a density point argument, there exists $B'\subset B$ and
$n_0\geq 1$ such that $\bar{\mu}(B')>0$ and for every
$(x_1,\ldots,x_\ell)\in B'$ and $n\geq n_0$ we have
\begin{equation}\label{pauza}
\bar{\mu} (\prod_{j=1}^\ell
\left(x_j-\frac{c}{n},x_j+\frac{c}{n}\right)\setminus
B)<\frac{\vep}{4(2n)^\ell }.
\end{equation}
Since $\bar{\varphi}$ (as a cocycle for $\bar{T}$) is recurrent,
there exists $n\geq n_0$ such that
\[\bar{\mu}\left(B'\cap \bar{T}^{-n}B'\cap(\bar{\varphi}^{(n)}\in (-\vep/2,\vep/2)^\ell )\right)>0.\]
Next choose $x^0=(x^0_1,\ldots,x_\ell^0)\in B'$ so that
$(T_1^nx^0_1,\ldots,T_\ell^nx_\ell^0)\in B'$,
$|\varphi_j^{(n)}(x_j^0)|<\vep/2$ for $1\leq j\leq\ell $. For each
$1\leq j\leq\ell $ denote by $J_{j,n}(x_j^0)\subset I_j$ the
interval of the partition $\mathcal{P}_n(T_j)$ such that $x_j^0\in
J_{j,n}(x_j^0)$. By assumption,  $\varphi_j^{(n)}$ is continuous on
every interval of $\mathcal{P}^n(T_j)$. Therefore, for every $1\leq
j\leq\ell, $ the function $\varphi_j^{(n)}$ is continuous on
$J_{j,n}(x_j^0)$, and hence $\varphi_j^{(n)}(x)=\pm nx+d_{n,j}$ for
$x\in J_{j,n}(x_j^0)$. In view of (\ref{niernac}),
$\frac{1}{cn}<|J_{j,n}(x_j^0)|<\frac{c}{n}$, and hence
$J_{j,n}(x_j^0)\subset(x^0_j-c/n,x^0_j+c/n)$ for every $1\leq
j\leq\ell $. Moreover, $T^n_jJ_{j,n}(x_j^0)$ is an interval such
that $|T^n_jJ_{j,n}(x_j^0)|=|J_{j,n}(x_j^0)|<c/n$, so
\begin{equation}\label{przesu}
T^n_jJ_{j,n}(x_j^0)\subset\left(T^n_jx_j^0-\frac{c}{n},T^n_jx_j^0+\frac{c}{n}\right).
\end{equation}
Since $|\varphi_j^{(n)}(x_j^0)|<\vep/2$, $\varphi_j^{(n)}$ is
linear on $J_{j,n}(x_j^0)$ with slope $\pm n$ and  $0\leq
r_j<\frac{1}{4c}<\frac{1}{2c}-\frac{\vep}{4}$, we can find
$(y^0_j-\vep/(4n),y^0_j+\vep/(4n))\subset J_{j,n}(x_j^0)$ such
that
\begin{equation}\label{oky1}
|\varphi_j^{(n)}(x)|\in r_j+(-\vep,\vep)\text{ for all } x\in
(y^0_j-\vep/(4n),y^0_j+\vep/(4n)).
\end{equation}
Let $y^0=(y^0_1,\ldots,y^0_\ell)\in\prod_{j=1}^\ell
J_{j,n}(x^0_j)$. Since
\[\prod_{j=1}^\ell \left(y^0_j-\frac{\vep}{4n},y^0_j+\frac{\vep}{4n}\right)
\subset\prod_{j=1}^\ell J_{j,n}(x^0_j)\subset\prod_{j=1}^\ell
\left(x_j^0-\frac{c}{n},x_j^0+\frac{c}{n}\right),\] $x^0\in B'$
and $n\geq n_0$, by (\ref{pauza}), we have
\[\bar{\mu}\left(\prod_{j=1}^\ell \left(y^0_j-\frac{\vep}{4n},y^0_j+\frac{\vep}{4n}\right)\setminus B\right)
<\frac{\vep}{4(2n)^\ell }.\] Moreover, by (\ref{przesu}),
\begin{eqnarray*}
\bar{T}^n\prod_{j=1}^\ell
\left(y^0_j-\frac{\vep}{4n},y^0_j+\frac{\vep}{4n}\right)\subset\prod_{j=1}^\ell
T^n_jJ_{j,n}(x^0_j)\subset\prod_{j=1}^\ell
\left(T^n_jx_j^0-\frac{c}{n},T^n_jx_j^0+\frac{c}{n}\right).
\end{eqnarray*}
Since $(T^n_1x_1^0,\ldots,T^n_\ell x_\ell^0)\in B'$ and $n\geq
n_0$, by (\ref{pauza}), it follows that
\[\bar{\mu}(\prod_{j=1}^\ell (y^0_j-\frac{\vep}{4n},y^0_j+\frac{\vep}{4n})\setminus
\bar{T}^{-n}B) = \bar{\mu}(\bar{T}^n\prod_{j=1}^\ell
(y^0_j-\frac{\vep}{4n},y^0_j+\frac{\vep}{4n})\setminus
B)<\frac{\vep}{4(2n)^\ell }.\] Hence
\[\bar{\mu}(\prod_{j=1}^\ell \left(y^0_j-\frac{\vep}{4n},y^0_j+\frac{\vep}{4n})\cap
(B\cap\bar{T}^{-n}B)\right)>\frac{\vep}{2(2n)^\ell }>0.\] By
(\ref{oky1}),
\[\bar{\varphi}^{(n)}(x)\in\prod_{j=1}^\ell (\{-r_j,r_j\}+(-\vep,\vep))\text{ if
}x\in\prod_{j=1}^\ell
\left(y^0_j-\frac{\vep}{4n},y^0_j+\frac{\vep}{4n}\right).\] Thus
\[\bar{\mu} (B\cap \bar{T}^{-n}B\cap (\bar{\varphi}^{(n)}\in\prod_{j=1}^\ell
(\{-r_j,r_j\}+(-\vep,\vep))))>\frac{\vep}{2(2n)^\ell }>0.\] By
Lemma~\ref{lemsym}, it follows that $\left(\prod_{j=1}^\ell
\{-r_j,r_j\}\right)\cap E(\bar{\varphi})\neq \emptyset$. This
completes the proof.
\end{proof}

\section{Ergodicity of certain step cocycles}\label{secstepcoc}\label{step}

In this section we apply Corollary~\ref{ergodic} to prove the
ergodicity of step cocycles over IETs of  periodic type.

\subsection{Step cocycles}\vskip 3mm

\hfill \break Let $T:I\to I$ be an arbitrary  IET satisfying
Keane's condition. Suppose that $(n_k)_{k\geq 0}$ is an increasing
sequence of natural numbers such $n_0=0$ and the matrix
\[Z(k+1)=\Theta(T^{(n_k)})\cdot\Theta(T^{(n_k+1)})\cdot\ldots\cdot\Theta(T^{(n_{k+1}-1)})\]
has positive entries for each $k\geq 0$. In what follows, we
denote by $(\pi^{(k)},\lambda^{(k)})$ the pair defining
$T^{(n_k)}$. By abuse of notation, we continue to write $T^{(k)}$
for $T^{(n_k)}$. With this notation,
\[\lambda^{(k)}=Z(k+1)\lambda^{(k+1)}.\]
We adopt the notation from \cite{Ma-Mo-Yo}. For each $k<l$ let
\[Q(k,l)=Z(k+1)\cdot Z(k+2)\cdot\ldots\cdot Z(l).\]
Then
\[\lambda^{(k)}=Q(k,l)\lambda^{(l)}.\]
We will write $Q(l)$ for $Q(0,l)$. By definition,
$T^{(l)}:I^{(l)}\to I^{(l)}$ is the first return map of
$T^{(k)}:I^{(k)}\to I^{(k)}$ to the interval $I^{(k)}\subset
I^{(l)}$. Moreover, $Q_{\alpha\beta}(k,l)$ is the time spent by any
point of $I^{(l)}_{\beta}$ in $I^{(k)}_{\alpha}$ until it returns to
$I^{(l)}$. It follows that
\[Q_{\beta}(k,l)=\sum_{\alpha\in\mathcal{A}}Q_{\alpha\beta}(k,l)\]
is the first return time of points of $I^{(l)}_{\beta}$ to
$I^{(l)}$.

Suppose that $T=T_{(\pi,\lambda)}$ is of periodic type and $p$ is
a period such that $\pi^{(p)}=\pi$. Let $A=\Theta^{(p)}(T)$.
Considering the sequence $(n_k)_{k\geq 0}$, $n_k=pk$ we get
$Z(l)=A$ and $Q(k,l)=A^{l-k}$ for all $0\leq k\leq l$.

The norm of a vector is defined as the largest absolute value of the
coefficients. We set
$\|B\|=\max_{\beta\in\mathcal{A}}\sum_{\alpha\in\mathcal{A}}|B_{\alpha\beta}|$
for $B=[B_{\alpha\beta}]_{\alpha,\beta\in\mathcal{A}}$. Following
\cite{Ve2}, for every matrix
$B=[B_{\alpha\beta}]_{\alpha,\beta\in\mathcal{A}}$ with positive
entries, we set
$$\nu(B)=\max_{\alpha,\beta,\gamma\in\mathcal{A}}\frac{B_{\alpha\beta}}{B_{\alpha\gamma}}.$$
Then
\begin{equation}
\sum_{\alpha\in\mathcal{A}}B_{\alpha\beta}\leq\nu(B)
\sum_{\alpha\in\mathcal{A}}B_{\alpha\gamma}\text{ for all }
\beta,\gamma\in\mathcal{A}\text{ and }\nu(CB)\leq\nu(B),
\end{equation}
for any nonnegative nonsingular matrix $C$. It follows that
$\nu(B^m)\leq\nu(B)$, and hence
\begin{equation}\label{maxmin1}
\|B^m\|=\max_{\beta\in\mathcal{A}}\sum_{\alpha\in\mathcal{A}}
B^m_{\alpha\beta}\leq\nu(B)\min_{\beta\in\mathcal{A}}
\sum_{\alpha\in\mathcal{A}}B^m_{\alpha\beta}.
\end{equation}
Denote by $\Gamma^{(k)}$ the space of functions
$\varphi:I^{(k)}\to \mathbb{R}$ constant on each interval
$I^{(k)}_{\alpha}$, $\alpha\in\mathcal{A}$ and denote by
$\Gamma^{(k)}_0$ the subspace of functions with zero mean. Every
function $\varphi=\sum_{\alpha\in \mathcal{A}}
h_{\alpha}\chi_{I^{(k)}_{\alpha}}$ in $\Gamma^{(k)}$ can be
identified with the vector $h=(h_{\alpha})_{\alpha\in
\mathcal{A}}\in\R^{\mathcal{A}}$. Moreover,
\begin{equation}\label{warkockst}
\varphi^{(Q(k,l)_\alpha)}(x)=(Q(k,l)^th)_\alpha\text{ for every
}x\in I^{(l)}_\alpha,\,\alpha\in\mathcal{A}.
\end{equation}
The induced IET $T^{(n)}:I^{(n)}\to I^{(n)}$ determines a
partition of $I$ into disjoint towers $H^{(n)}_{\alpha},$
$\alpha\in\mathcal{A}$, where
\[H^{(n)}_\alpha=\{T^k I^{(n)}_\alpha:0\leq k<h_\alpha^{(n)}:=Q_\alpha(n)\}.\]
Denote by $h^{(n)}_{\max}$ and $h^{(n)}_{\min}$ the height of the
highest and the lowest tower respectively.

Assume that $I^{(n+1)}\subset I^{(n)}_{\alpha_1}$, where
$\pi_0^{(n)}(\alpha_1)=1$. For every $\alpha\in\mathcal{A}$ denote
by $C^{(n)}_{\alpha}$ the tower $\{T^iI^{(n+1)}_\alpha:0\leq
i<h_{\alpha_1}^{(n)}\}$.

\begin{lemma} \label{ciasny2} For every
$\alpha\in\mathcal{A}$ we have
\begin{equation}\label{ciasny}
\mu(C^{(n)}_{\alpha}\triangle TC^{(n)}_{\alpha})\to 0\text{ and
}\sup_{x\in C^{(n)}_{\alpha}}|T^{h_\alpha^{(n+1)}}x-x|\to 0 \text{
as }n\to+\infty.
\end{equation}
If
$\varphi=\sum_{\alpha\in\mathcal{A}}v_\alpha\chi_{I^{(0)}_\alpha}$
for some  $v=(v_\alpha)_{\alpha\in\mathcal{A}}\in \Gamma^{(0)}_0$,
then
\begin{equation}\label{wartosci0}
\varphi^{(h_\alpha^{(n+1)})}(x)=(Q(n+1)^tv)_{\alpha}\text{ for all
}x\in C^{(n)}_{\alpha}.
\end{equation}
If additionally $T$ is of periodic type then
\begin{equation}\label{dolcn}
\liminf_{n\to\infty}\mu(C^{(n)}_{\alpha})>0
\end{equation}
and
\begin{equation}\label{wartosci}
\varphi^{(h_\alpha^{(n+1)})}(x)=((A^t)^{n+1}v)_{\alpha}\text{ for
all }x\in C^{(n)}_{\alpha}.
\end{equation}
\end{lemma}

\begin{proof} Since $C^{(n)}_{\alpha}\triangle TC^{(n)}_{\alpha}
\subset T^{h_{\alpha_1}^{(n+1)}}I^{(n+1)}_\alpha\cup
I^{(n+1)}_\alpha$, we have
\[\mu(C^{(n)}_{\alpha}\triangle TC^{(n)}_{\alpha})\leq
2\mu(I^{(n+1)}_\alpha)\to 0\text{ as }n\to+\infty.\] Suppose that
$x\in T^iI^{(n+1)}_\alpha$ for some  $0\leq i<h_{\alpha_1}^{(n)}$.
Then
\[T^{h_\alpha^{(n+1)}}x\in T^iT^{h_\alpha^{(n+1)}}I^{(n+1)}_\alpha
\subset T^iI^{(n+1)}\subset T^iI^{(n)}_{\alpha_1}.\]
It follows that
\begin{equation}\label{stalewa}
x,T^{h_\alpha^{(n+1)}}x\in T^iI^{(n)}_{\alpha_1}\subset
I_\beta\text{ for some }\beta\in\mathcal{A}.
\end{equation}
Therefore
\begin{equation*}
|x-T^{h_\alpha^{(n+1)}}x|\leq |I^{(n)}_{\alpha_1}|\text{ for all
}x\in C^{(n)}_{\alpha}.
\end{equation*}
Next, by (\ref{warkockst}),
$\varphi^{(h_\alpha^{(n+1)})}(x)=(Q(n+1)^tv)_{\alpha}$  for every
$x\in I^{(n+1)}_\alpha$. Moreover, if $x\in C^{(n)}_{\alpha}$, say
$x= T^ix_0$ with $x_0\in I^{(n+1)}_\alpha$ and  $0\leq
i<h_{\alpha_1}^{(n)}$, then
\[\varphi^{(h_\alpha^{(n+1)})}(T^ix_0)-\varphi^{(h_\alpha^{(n+1)})}
(x_0)=\sum_{0\leq j< i}\varphi(T^{h_\alpha^{(n+1)}}T^jx_0)-\varphi(T^jx_0).\]
By (\ref{stalewa}),
$\varphi(T^{h_\alpha^{(n+1)}}T^jx_0)=\varphi(T^jx_0)$ for every
$0\leq j<h_{\alpha_1}^{(n)}$, and hence
\begin{equation*}
\varphi^{(h_\alpha^{(n+1)})}(x)=\varphi^{(h_\alpha^{(n+1)})}(x_0)=(Q(n+1)^tv)_{\alpha}\text{
for all }x\in C^{(n)}_{\alpha}.
\end{equation*}

Assume that $T=T_{(\pi,\lambda)}$ is of periodic type and $A$ is
its periodic matrix. Denote by $\rho_1$ the Perron-Frobenius
eigenvalue of $A$. Then there exists $C>0$ such that $\frac{1}{C}
\rho_1^n\leq\|A^n\|\leq C \rho_1^n$. Since
$h^{(n)}_{\max}=\|A^n\|=\max_{\alpha\in \mathcal{A}}A^n_\alpha$
and $h^{(n)}_{\min}=\min_{\alpha\in \mathcal{A}}A^n_\alpha$, by
(\ref{maxmin1}), it follows that
\begin{equation}\label{minimax}
\frac{1}{C\nu(A)} \rho_1^n\leq h^{(n)}_{\min}<h^{(n)}_{\max}\leq C
\rho_1^n.
\end{equation}
As $|I^{(n+1)}_\alpha|=\rho_1^{-(n+1)}|I^{(0)}_\alpha|$, we have
\begin{equation*}
\mu(C^{(n)}_{\alpha})=|I^{(n+1)}_\alpha|h_{\alpha_1}^{(n)}=|I^{(0)}_\alpha|h_{\min}^{(n)}/\rho_1^{n+1}\geq
\frac{|I^{(0)}_\alpha|}{C\nu(A)\rho_1}>0.
\end{equation*}
Multiplying the period of $T$, if necessary, we have
$I^{(n+1)}\subset I^{(n)}_{\alpha_1}$ for every natural $n$, and
hence
\begin{equation*}
\varphi^{(h_\alpha^{(n+1)})}(x)=(Q(n+1)^tv)_{\alpha}=((A^t)^{n+1}v)_{\alpha}\text{
for all }x\in C^{(n)}_{\alpha}.
\end{equation*}
\end{proof}

\subsection{Ergodic cocycles in case $\kappa>1$}\label{nonhyperbolic}\vskip 3mm

\hfill \break Assume that $T=T_{(\pi,\lambda)}$ is of periodic
type and $\kappa=\kappa(\pi)>1$. Then $\dim
\ker\Omega_{\pi}=\kappa-1>0$. As we already mentioned $A$ is the
identity on $\ker\Omega_{\pi}$. Let
\[F(T)=\{v\in\R^{\mathcal{A}}:A^tv=v\}.\]
Then $F(T)$ is a linear subspace with $\dim F(T)=k\geq\kappa-1$.
Since
\[\langle v,\lambda\rangle=\langle A^tv,\lambda\rangle=
\langle v,A\lambda\rangle=\rho_1\langle v,\lambda\rangle\text{ for
each }v\in F(T),\] we have $F(T)\subset \Gamma^{(0)}_0$. Moreover,
we can choose a basis of the linear space $F(T)$ such that each of
its element belongs to $\Z^{\mathcal{A}}$. It follows that
$\Z^{\mathcal{A}}\cap F(T)$ is a free abelian group of rank $k$.

\begin{lemma}\label{algebra} Let
$v_i=(v_{i\alpha})_{\alpha\in\mathcal{A}}$, $1\leq i\leq k$, be a
basis of the group $\Z^{\mathcal{A}}\cap F(T)$. Then the
collection of vectors
$w_{\alpha}=(v_{i\alpha})_{i=1}^{k}\in\Z^{k}$,
$\alpha\in\mathcal{A}$, generates the group $\Z^{k}$.\bez
\end{lemma}

\begin{theorem}
Let $v_i=(v_{i\alpha})_{\alpha\in\mathcal{A}}$, $1\leq i\leq k$ be
a basis of the group $\Z^{\mathcal{A}}\cap F(T)$. Then the cocycle
$\varphi:I\to \Z^{k}$ given by
$\varphi=(\varphi_1,\ldots,\varphi_{k})$ with
$\varphi_i=\sum_{\alpha\in\mathcal{A}}v_{i\alpha}\chi_{I_{\alpha}}$
for $i=1,\ldots,k$ is ergodic.

If $R$ is a $(k-1)\times k$-real matrix satisfying (\ref{denseRm}),
then the cocycle $\widetilde{\varphi}:I\to\R^{k-1}$ given by
$\widetilde{\varphi}(x)=R\varphi(x)$, which is constant over
exchanged intervals, is ergodic.
\end{theorem}

\begin{proof}
By (\ref{wartosci}), for every $\alpha\in\mathcal{A}$ we have
\[\varphi^{(h_\alpha^{(n+1)})}(x)=(((A^t)^{n+1}v_1)_{\alpha},
\ldots,((A^t)^{n+1}v_{k})_{\alpha})
=((v_1)_{\alpha},\ldots,(v_{k})_{\alpha})=w_{\alpha}\] for $x\in
C^{(n)}_{\alpha}$. In view of Lemma~\ref{ciasny2}, we can apply
Corollary~\ref{ergodic}. Thus $w_\alpha\in E(\varphi)$ for all
$\alpha\in\mathcal{A}$. Since $E(\varphi)$ is a group, by
Lemma~\ref{algebra}, we obtain $E(\varphi)=\Z^{k}$.

\vskip 3mm It is easy to show that $RE(\varphi)\subset
E(R\varphi)$. Since $E(\varphi)=\Z^{k}$ and $E(R\varphi)$ is
closed, by Remark~\ref{macierz}, we obtain
$E(\widetilde{\varphi})=E(R\varphi)\supset\overline{R\Z^{k}}=\R^{k-1}$.
\end{proof}

\begin{remark} Note that Remark~\ref{macierz} indicates how to
construct matrices $R$ satisfying (\ref{denseRm}).
\end{remark}

\vskip 3mm
\section{Ergodicity of corrected cocycles}\label{correct}

In this section, using a method from \cite{Ma-Mo-Yo}, we present a
procedure of correction of functions in $\bv_0(\sqcup_{\alpha\in
\mathcal{A}} I^{(0)}_{\alpha})$ by piecewise constant functions (in
$\Gamma^{(0)}_0$) in order to obtain better control on the growth of
Birkhoff sums. It will allow us to prove the ergodicity of some
corrected cocycles.

\vskip 3mm
\subsection{Rauzy-Veech induction for cocycles}
\vskip 3mm

\hfill \break For every cocycle $\varphi:I^{(k)}\to\R$ for the IET
$T^{(k)}:I^{(k)}\to I^{(k)}$ and $l>k$ denote by
$S(k,l)\varphi:I^{(l)}\to\R$ the renormalized cocycle for
$T^{(l)}$ given by
\[S(k,l)\varphi(x)=\sum_{0\leq i<Q_{\beta}(k,l)}\varphi((T^{(k)})^ix)\text{ for }x\in I^{(l)}_\beta.\]
Note that the operator $S(k,l)$ maps $\bv(\sqcup_{\alpha\in
\mathcal{A}} I^{(k)}_{\alpha})$ into $\bv(\sqcup_{\alpha\in
\mathcal{A}} I^{(l)}_{\alpha})$ and
\begin{equation}\label{nave}
\var S(k,l)\varphi\leq \var\varphi,
\end{equation}
\begin{equation}\label{nase}
\| S(k,l)\varphi\|_{\sup}\leq \|
Q(k,l)\|\|\varphi\|_{\sup}\;\;\text{ and }
\end{equation}
\begin{equation}\label{zachcal}
\int_{I^{(l)}}S(k,l)\varphi(x)\,dx= \int_{I^{(k)}}\varphi(x)\,dx
\end{equation}
for all $\varphi\in \bv(\sqcup_{\alpha\in \mathcal{A}}
I^{(k)}_{\alpha})$. In view of (\ref{zachcal}), $S(k,l)$ maps
$\bv_0(\sqcup_{\alpha\in \mathcal{A}} I^{(k)}_{\alpha})$ into
$\bv_0(\sqcup_{\alpha\in \mathcal{A}} I^{(l)}_{\alpha})$.

Recall that $\Gamma^{(k)}$ is the space of functions
$\varphi:I^{(k)}\to \mathbb{R}$ which are constant on each
interval $I^{(k)}_{\alpha}$, $\alpha\in\mathcal{A}$ and
$\Gamma^{(k)}_0$ is the subspace of functions with zero mean. Then
\[S(k,l)\Gamma^{(k)}=\Gamma^{(l)}\;\;\text{ and }\;\;S(k,l)\Gamma^{(k)}_0=\Gamma^{(l)}_0.\]
Moreover, every function $\sum_{\alpha\in \mathcal{A}}
h_{\alpha}\chi_{I^{(k)}_{\alpha}}$ from $\Gamma^{(k)}$ can be
identified with the vector $h=(h_{\alpha})_{\alpha\in
\mathcal{A}}\in\R^{\mathcal{A}}$. Under this identification,
\[\Gamma^{(k)}_0=Ann(\lambda^{(k)}):=\{h=(h_{\alpha})_{\alpha\in \mathcal{A}}
\in\R^{\mathcal{A}}:\langle h, \lambda^{(k)}\rangle=0\}\] and the
operator $S(k,l)$ is the linear automorphism of $\R^{\mathcal{A}}$
whose matrix in the canonical basis is $Q(k,l)^t$. Moreover, the
norm on $\Gamma^{(k)}$ inherited from the supremum norm coincides
with the norm of vectors.

\subsection{Correction of functions of bounded variation}
\vskip 3mm

\hfill \break Suppose now that $T$ is of periodic type.
 Let us consider the linear subspaces
\[\Gamma^{(k)}_{cs}=\{h\in\Gamma^{(k)}:\limsup_{l\to\infty}\frac{1}{l}
\log\|S(k,l)h\|=\limsup_{l\to\infty}\frac{1}{l}\log\|Q(k,l)^t
h\|\leq 0\},\]
\[\Gamma^{(k)}_{u}=\{h\in\Gamma^{(k)}:\limsup_{l\to\infty}\frac{1}{l}
\log\|S(k,l)h\|=\limsup_{l\to\infty}\frac{1}{l}\log\|Q(k,l)^t h\|>
0\}.\] Denote by
\[U^{(k)}:\bv(\sqcup_{\alpha\in \mathcal{A}} I^{(k)}_{\alpha})\to
\bv(\sqcup_{\alpha\in \mathcal{A}}
I^{(k)}_{\alpha})/\Gamma^{(k)}_{cs}\] the projection on the quotient
space. Let us consider the linear operator
$P^{(k)}_0:\bv_0(\sqcup_{\alpha\in \mathcal{A}}
I^{(k)}_{\alpha})\to\bv_0(\sqcup_{\alpha\in \mathcal{A}}
I^{(k)}_{\alpha})$ given by
\[P^{(k)}_0\varphi(x)=\varphi(x)-\frac{1}{|I^{(k)}_{\alpha}|}
\int_{I^{(k)}_{\alpha}}\varphi(t)dt\text{ if }x\in I^{(k)}_{\alpha}.\]

\begin{theorem}\label{thmcorre}
For every $\varphi\in \bv_0(\sqcup_{\alpha\in \mathcal{A}}
I^{(k)}_{\alpha})$  the sequence
\begin{equation}\label{ciagdop}\{U^{(k)}\circ S(k,l)^{-1}\circ P_0^{(l)}\circ
S(k,l)\varphi\}_{l\geq k}
\end{equation}
converges in the quotient norm on $\bv_0(\sqcup_{\alpha\in
\mathcal{A}} I^{(k)}_{\alpha})/\Gamma_{cs}^{(k)}$ induced by
$\|\,\cdot\,\|_{\bv}$.
\end{theorem}

\begin{notations}
Let $P^{(k)}:\bv_0(\sqcup_{\alpha\in \mathcal{A}}
I^{(k)}_{\alpha})\to\bv_0(\sqcup_{\alpha\in \mathcal{A}}
I^{(k)}_{\alpha})/\Gamma^{(k)}_{cs}$ stand for the limit operator.
Note that if $\varphi\in\Gamma^{(k)}_0$ then $P_0^{(k)}\varphi=0$,
and hence $P^{(k)}\varphi=0$.

We denote by $\bv^{\lozenge}(\sqcup_{\alpha\in \mathcal{A}}
I_{\alpha})$ the subspace of functions
$\varphi\in\bv(\sqcup_{\alpha\in \mathcal{A}} I_{\alpha})$ such that
$\varphi_-(x)=\varphi_+(x)$ for every $x=T^nl_\alpha$,
$\alpha\in\mathcal{A}$, $\pi_0(\alpha)\neq 1$,
$n\in\Z\setminus\{0\}$.
\end{notations}

Recall that, in general, the growth of $(S(k)\varphi)_{k\geq 1}$ is
exponential with exponent $\theta_2/\theta_1$ (see
Theorem~\ref{thmthetas}). Nevertheless, the growth can be reduced by
correcting the function $\varphi$ by a function $h$ constant on the
exchanged intervals.
\begin{theorem}\label{thmcorrecgener}
Suppose now that $T=T_{(\pi,\lambda)}$ is of periodic type and $M$
is the maximal size of Jordan blocks in the Jordan decomposition
of its periodic matrix. Let $\varphi\in\bv_0(\sqcup_{\alpha\in
\mathcal{A}} I^{(0)}_{\alpha})$. There exist $C_1,C_2>0$ such that
if $\widehat{\varphi}+\Gamma_{cs}^{(0)}= P^{(0)}\varphi$, then
$\widehat{\varphi}-\varphi\in\Gamma_0^{(0)}$ and
\begin{equation}\label{polygr}
\|S(k)(\widehat{\varphi})\|_{\sup}\leq
C_1k^{M}\var\varphi+C_2k^{M-1}\|\widehat{\varphi}\|_{\sup}\text{
for every natural }k.
\end{equation} For every
$\varphi\in\bv_0(\sqcup_{\alpha\in \mathcal{A}} I^{(0)}_{\alpha})$
there exists $h\in \Gamma^{(0)}_u\cap\Gamma^{(0)}_0$ such that
$\varphi+h+\Gamma^{(0)}_{cs}=P^{(0)}\varphi$. Moreover, the vector
$h\in \Gamma^{(0)}_u\cap\Gamma^{(0)}_0$ is unique.

If additionally $T$ has non-degenerated spectrum and
$\varphi\in\bv_0^{\lozenge}(\sqcup_{\alpha\in \mathcal{A}}
I^{(0)}_{\alpha})$ then
\[\|S(k)(\widehat{\varphi})\|_{\sup}\leq C_1\var\varphi
+C_2\|\widehat{\varphi}\|_{\sup}\text{ for every natural }k.\]
\end{theorem}

For completeness the proofs of these theorems will be given in
Appendix~\ref{korekcja}.

\begin{remark} If we restrict the choice of $h$ to the subspace
$\Gamma_u^{(0)}\cap\Gamma^{(0)}_0$, then the correction
$h\in\Gamma_u^{(0)}\cap\Gamma^{(0)}_0$ is unique. In what follows,
$\widehat{\varphi}$ will stand for the function $\varphi$ corrected
by the unique correction $h\in\Gamma_u^{(0)}\cap\Gamma^{(0)}_0$
(i.e. $\widehat{\varphi}=\varphi+h$).

\vskip 3mm If $\varphi:I\to\R^\ell$ with $\varphi=
(\varphi_1,\ldots,\varphi_\ell)$, we deal with the corrected
function $\widehat{\varphi}
:=(\widehat{\varphi_1},\ldots,\widehat{\varphi_\ell})$, and we have
\[\|S(k)(\widehat{\varphi})\|_{\sup}\leq C_1\max_{1\leq i\leq
\ell}\var\varphi_i+C_2\|\widehat{\varphi}\|_{\sup}\text{ for every natural }k.\]
\end{remark}

\subsection{Ergodicity of corrected step functions}\label{subsechyper}\vskip 3mm

\hfill \break
 We now consider
piecewise constant zero mean cocycles $\varphi:I\to\R^\ell$,
$\ell\geq 1$ which are also discontinuous in the interior of the
exchanged intervals. Suppose that $\gamma_i\in I$, $i=1,\ldots,s$
are discontinuities of $\varphi$ different from $l_{\alpha}$,
$\alpha\in\mathcal{A}$. Denote by $\bar{d}_i\in\R^\ell$ the vector
describing the jumps of coordinate functions of $\varphi$ at
$\gamma_i$, this is,
$\bar{d}_i=\varphi_+(\gamma_i)-\varphi_-(\gamma_i)\in\R^\ell$. In
this section we will prove the ergodicity of $\widehat{\varphi}$
for almost every choice of discontinuities. Note that the
corrected cocycle $\widehat{\varphi}$ is also piecewise constant
and it is discontinuous at $\gamma_i$ with the jump vector
$\bar{d}_i$ for $i=1,\ldots,s$, and hence it is still non-trivial.

\begin{theorem}\label{thmfullmeasure}
 Suppose that $T=T_{(\pi,\lambda)}$ is an IET of
periodic type and it has non-degenerated spectrum.
There exists a
set $D\subset I^s$ of full Lebesgue measure such that if
\begin{itemize}
\item[(i)] $(\gamma_1,\ldots,\gamma_s)\in D$;
\item[(ii)] the subgroup $\Z(\bar{d}_1,\ldots,\bar{d}_s)\subset\R^\ell$
generated by $\bar{d}_1,\ldots,\bar{d}_s$ is dense in $\R^\ell$,
\end{itemize}
then the cocycle $\widehat{\varphi}:I\to\R^\ell$ is ergodic.
\end{theorem}

\begin{proof}
As we already mentioned we can assume that $I^{(n+1)}\subset
I^{(n)}_{\alpha_1}$ for every natural $n$, where
$\alpha_1=(\pi_0^{(n)})^{-1}(1)=\pi_0^{-1}(1)$. Fix
$\alpha\in\mathcal{A}$ and choose
$b_0<a_1<b_1<\ldots<a_s<b_s<a_{s+1}$ so that
$[b_0,a_{s+1})=I_{\alpha}$. Let
\[F^{(n)}_i=\bigcup_{h^{(n)}_{\alpha_1}\leq j<h^{(n+1)}_{\alpha}}
T^j(a_i/\rho_1^{n+1},b_i/\rho_1^{n+1}), \text{ for }1\leq i\leq s,\]
\[C^{(n)}_i=\bigcup_{0\leq j<h^{(n)}_{\alpha_1}}
T^j(b_i/\rho_1^{n+1},a_{i+1}/\rho_1^{n+1}), \text{ for }0\leq i\leq
s\] ($\rho_1$ is the Perron-Frobenius eigenvalue of the periodic
matrix $A$ of $T$). Since
$[b_0/\rho_1^{n+1},a_{s+1}/\rho_1^{n+1})=I^{(n+1)}_{\alpha}$, the
sets $C^{(n)}_i$, $F^{(n)}_i$ are towers for which each level is an
interval. Moreover, $C^{(n)}_i\subset C^{(n)}_{\alpha}$ for $0\leq
i\leq s$ and
$$h^{(n+1)}_{\alpha}-h^{(n)}_{\alpha_1}\geq\sum_{\beta\in\mathcal{A}}
h^{(n)}_{\beta}-h^{(n)}_{\alpha_1}\geq
h^{(n)}_{\min}.$$ In view of (\ref{minimax}), it follows that
\[\mu(C^{(n)}_i)=(a_{i+1}-b_i)\frac{h^{(n)}_{\alpha_1}}{\rho_1^{n+1}}\geq (a_{i+1}-b_i)
\frac{h_{\min}^{(n)}}{\rho_1^{n+1}}\geq \frac{a_{i+1}-b_i}{C\nu(A)\rho_1}>0,\]
\[\mu(F^{(n)}_i)=(b_i-a_i)\frac{h^{(n+1)}_{\alpha}-h^{(n)}_{\alpha_1}}{\rho_1^{n+1}}
\geq (b_i-a_i)\frac{h^{(n)}_{\min}}{\rho_1^{n+1}}\geq \frac{b_i-a_i}{C\nu(A)\rho_1}>0.\]

Recall  that if $T:\xbm\to \xbm$ is ergodic and  $(\Xi_n)_{n\geq 1}$
is a sequence of towers for $T$ for which
\[\liminf_{n\rightarrow\infty}\mu(\Xi_n)>0\text{ and
height}(\Xi_n)\rightarrow\infty,\]
then (see King \cite{Ki}, Lemma 3.4)
\begin{equation}\label{king}
\mu(B\cap \Xi_n)-\mu(B)\mu(\Xi_n)\rightarrow 0 \text{ for all
}B\subset\mathcal{B}.
\end{equation}
It follows that, for $\mu$-almost every $x\in X$, the point $x$
belongs to $\Xi_n$ for infinitely many $n$.

Applying this fact for subsequences of  $(F^{(n)}_i)_{n\geq 1}$
successively for $i=1,\ldots ,s$, we conclude that
for a.e.\ $(\gamma_1,\ldots,\gamma_s)\in I^s$ there exists
a subsequence $(k_n)_{n\geq 1}$ such that
\[\gamma_i\in F^{(k_n)}_i\text{ for all }1\leq i\leq s\text{ and }n\geq 1.\]
Denote by $D\subset I^s$ the subset of all such
$(\gamma_1,\ldots,\gamma_s)$ for which $\gamma_i$ does not belong to
the union of orbits of $l_\alpha$, $\alpha\in\mathcal{A}$, for
$i=1,\ldots,s$. Therefore $\varphi\in
\bv^{\lozenge}(\sqcup_{\alpha\in\mathcal{A}}I_\alpha,\R^\ell)$.

Suppose that for some $n\geq 1$ we have $\gamma_i\in F^{(n)}_i$
for all $1\leq i\leq s$. Then the sets $T^j C^{(n)}_i$, $0\leq
j<h^{(n+1)}_{\alpha}$, $0\leq i\leq s$ do not contain
discontinuities of $\widehat{\varphi}$. Thus similar arguments to
those from the proof of (\ref{wartosci}) show that
$\widehat{\varphi}^{(h^{(n+1)}_{\alpha})}$ is constant  on each
$C^{(n)}_i$ and equals say $\bar{g}^{(n)}_i\in\R^\ell$.

Let $x\in [b_{i-1}/\rho_1^{n+1},a_i/\rho_1^{n+1})$ and $y\in
[b_{i}/\rho_1^{n+1},a_{i+1}/\rho_1^{n+1})$. By assumption,
$\gamma_i\in T^{j_0}[a_i/\rho_1^{n+1},b_i/\rho_1^{n+1})$ for some
$h^{(n)}_{\alpha_1}\leq j_0<h^{(n+1)}_{\alpha}$. It follow that
$\widehat{\varphi}(T^jx)=\widehat{\varphi}(T^jy)$ for all $0\leq
j<h^{(n+1)}_{\alpha}$, $j\neq j_0$ and
$\widehat{\varphi}(T^{j_0}y)-\widehat{\varphi}(T^{j_0}x)=\bar{d}_i$.
Consequently,
\[\bar{g}^{(n)}_i-\bar{g}^{(n)}_{i-1}=\widehat{\varphi}^{(h^{(n+1)}_{\alpha})}
(y)-\widehat{\varphi}^{(h^{(n+1)}_{\alpha})}(x)=\bar{d}_i.\]
It  follows that
\[\widehat{\varphi}^{(h^{(n+1)}_{\alpha})}(x)=\bar{g}^{(n)}_0+\sum_{l=1}^i\bar{d}_l
\text{ for all }x\in C^{(n)}_i,\;0\leq i\leq s.\]
Since  $\varphi\in
\bv^{\lozenge}(\sqcup_{\alpha\in\mathcal{A}}I_\alpha,\R^\ell)$, by
Theorem~\ref{thmcorrecgener} there exists $C>0$ such that
\[\|\widehat{\varphi}^{(h^{(n+1)}_{\alpha})}(x)\|=\|S(n+1)
\widehat{\varphi}(x)\|\leq C\text{ for all  }x\in
I^{(n+1)}_{\alpha},\] and hence $\|\bar{g}^{(n)}_0\|\leq C$.
Therefore for each $(\gamma_1,\ldots,\gamma_s)\in D$ there exists
a subsequence $(k_n)_{n\geq 1}$ such that
\[\widehat{\varphi}^{(h^{(k_n+1)}_{\alpha})}(x)=\bar{g}^{(k_n)}_0+\sum_{l=1}^i\bar{d}_l\text{ for all }
x\in C^{(k_n)}_i,\;0\leq i\leq s\] and $\bar{g}^{(k_n)}_0\to
\bar{g}_0$ in $\R^\ell$. Since $\liminf \mu(C^{(k_n)}_i)>0$ for each
$0\leq i\leq s$, Corollary~\ref{ergodic} implies
$\bar{g}_0+\sum_{l=1}^i\bar{d}_l\in E(\widehat{\varphi})$ for each
$0\leq i\leq s$. Therefore $\bar{d}_l\in E(\widehat{\varphi})$ for
each $1\leq l\leq s$. Since $\bar{d}_1,\ldots,\bar{d}_s$ generate a
dense subgroup of $\R^\ell$ and $E(\widehat{\varphi})$ is closed, it
follows that $E(\widehat{\varphi})=\R^\ell$.
\end{proof}

\begin{remark}
Notice that the condition (ii) implies $s>\ell$. On the other
hand, if $s>\ell$, in view of Remark~\ref{macierz}, we can easily
find a collection of vectors
$\bar{d}_1,\ldots,\bar{d}_s\in\R^\ell$ such that
$\overline{\Z(\bar{d}_1,\ldots,\bar{d}_s)}=\R^\ell$.
\end{remark}

In order to have a more specific condition on the discontinuities
$\gamma_i$, $i=1,\ldots,s$ guaranteeing ergodicity, we can use a
periodic type condition.

Let us consider a set $\{\gamma_1,\ldots,\gamma_s\}\subset
I\setminus\{l_\alpha:\alpha\in\mathcal{A}\}$. The points
$\gamma_1,\ldots,\gamma_s$ together with $l_\alpha$,
$\alpha\in\mathcal{A}$ give a new partition of $I$ into $d+s$
intervals. Therefore $T$ can be treated as a $d+s$-IET. Denote by
$(\pi',\lambda')$ the combinatorial data of this representation of
$T$.

\begin{definition}
We say that the set $\{\gamma_1,\ldots,\gamma_s\}$ is of {\em
periodic type with respect to $T_{(\pi,\lambda)}$} if the IET
$T_{(\pi',\lambda')}$ is of periodic type as an exchange of $d+s$
intervals.
\end{definition}

\begin{remark}\label{diamond}
By the definition of periodic type, $(\lambda',\pi')$ satisfies the
Keane condition. Therefore, each $\gamma_i$ does not belong to the
orbit of any $l_\alpha$, $\alpha\in\mathcal{A}$.
\end{remark}
In view of Theorem~23 in \cite{Ra}, each admissible interval
$I^{(p)}$ ($p$ is a period) for $T_{(\pi',\lambda')}$ is also
admissible for $T_{(\pi,\lambda)}$. Therefore $T_{(\pi,\lambda)}$ is
of periodic type as an exchange of $d$-intervals as well. It follows
that, for every $n\geq 0$ and $i=1,\ldots,s$ if $\gamma_i\in
I_\alpha$, then $\gamma_i=T_{(\pi,\lambda)}^j(\gamma_i/\rho^n)$ for
some $0\leq j<h^{(n)}_\alpha$. Therefore similar arguments to those
in the proof of Theorem~\ref{thmfullmeasure} give the following
result.

\begin{theorem}\label{ergpiconsper}
Suppose that $T=T_{(\pi,\lambda)}$ is an IET of periodic type and
it has non-degenerated spectrum. Let $\varphi:I\to\R^\ell$ be a
zero mean piecewise constant cocycle with additional discontinuity
at $\gamma_i\in I\setminus\{l_\alpha:\alpha\in\mathcal{A}\}$ with
the jump vectors $\bar{d}_i\in\R^\ell$ for $i=1,\ldots,s$. If
\begin{itemize}
\item[(i)] the set $\{\gamma_1,\ldots,\gamma_s\}$ is of periodic
type with respect to $T_{(\pi,\lambda)}$;
\item[(ii)] $\overline{\Z(\bar{d}_1,\ldots,\bar{d}_s)}=\R^\ell$,
\end{itemize}
then the cocycle $\widehat{\varphi}:I\to\R^\ell$ is ergodic. \bez
\end{theorem}

\section{Recurrence and ergodicity of  extensions of multivalued Hamiltonians} \label{multihami}
In this section we deal with a class of smooth flows on
non-compact manifolds which are extensions of so called
multivalued Hamiltonian flows on compact surfaces of higher genus.
Each such flow has a special representation over a skew product of
an IET and a BV cocycle. This allows us to apply abstract results
from previous sections to state some sufficient conditions for
recurrence and ergodicity whenever the IET is of periodic type.

\subsection{Special flows}\vskip 3mm

\hfill \break In this subsection we briefly recall some basic
properties of special flows. Let $T$ be an automorphism of a
$\sigma$-finite measure space $(X,\mathcal{B},\mu)$. Let
$f:X\to\R$ be a strictly positive function such that
\begin{equation}\label{infty}
\sum_{n\geq 1}f(T^nx)=+\infty\text{ for a.e. }x\in X.
\end{equation}
By $T^f=(T^f_t)_{t\in\R}$ we will mean the corresponding special
flow under $f$ (see e.g.\ \cite{Co-Fo-Si}, Chapter 11) acting on
$(X^f,\mathcal{ B}^f,\mu^f)$, where $X^f=\{(x,s)\in X\times
\R:\:0\leq s<f(x)\}$ and $\mathcal{ B}^f$ $(\mu^f)$ is the
restriction of $\mathcal{ B}\times\mathcal{ B}(\R)$ $(\mu\times
m_{\R})$ to $X^f$. Under the action of the flow ${T}^f$ each point
in $X^f$ moves vertically at unit speed, and we identify the point
$(x,f(x))$ with $(Tx,0)$. More precisely, for every $(x,s)\in X^f$
we have
\[T^f_t(x,s)=(T^nx,s+t-f^{(n)}(x)),\]
where $n\in\Z$ is a unique number such that $f^{(n)}(x)\leq
s+t<f^{(n+1)}(x)$.

\begin{remark}\label{special}
If $T$ is conservative then the condition (\ref{infty}) holds
automatically and the special flow $T^f$ is conservative as well.
Moreover, if $T$ is ergodic then $T^f$ is ergodic.
\end{remark}

\subsection{Basic properties of multivalued Hamiltonian flows}
\vskip 3mm

\hfill \break Now we will consider multivalued Hamiltonians and
their associated flows, a model which has been developed by S.P.
Novikov (see also \cite{Arn} for the toral case). Let $(M,\omega)$
be a compact symplectic smooth surface and $\beta$ be a Morse closed
$1$-form on $M$. Denote by $\pi:\widehat{M}\to M$ the universal
cover of $M$ and by $\widehat{\beta}$ the pullback of $\beta$ by
$\pi:\widehat{M}\to M$. Since $\widehat{M}$ is simply connected and
$\widehat{\beta}$ is also a closed form, there exists a smooth
function $\widehat{H}:\widehat{M}\to \R$, called a multivalued
Hamiltonian, such that $d\widehat{H}=\widehat{\beta}$. By
assumption, $\widehat{H}$ is a Morse function. Suppose additionally
that all critical values of $\widehat{H}$ are distinct.

Denote by $X:M\to TM$ the smooth vector field determined by
\[\beta=i_X\omega=\omega(X,\,\cdot\,).\]
Let $(\phi_t)_{t\in\R}$ stand for the smooth flow on $M$
associated to the vector field $X$. Since $d\beta=0$, the flow
$(\phi_t)_{t\in\R}$ preserves the symplectic form $\omega$, and
hence it preserves the smooth measure $\nu=\nu_\omega$ determined
by $\omega$. Since $\beta$ is a Morse form, the flow
$(\phi_t)_{t\in\R}$ has finitely many fixed points (equal to zeros
of $\beta$ and equal to images of critical points of $\widehat{H}$
by the map $\pi$). The set of fixed points will by denoted by
$\mathcal{F}(\beta)$. All of them are centers or non-degenerated
saddles. By assumption, any two different saddles are not
connected by a separatrix of the flow (called a saddle
connection). Nevertheless, the flow $(\phi_t)_{t\in\R}$ can have
saddle connections which are loops. Each such saddle connection
gives a decomposition of $M$ into two nontrivial invariant
subsets.

By Theorem~14.6.3 in \cite{Ka-Ha}, the surface $M$ can be
represented as the  finite union of disjoint
$(\phi_t)_{t\in\R}$--invariant sets as follows
\[M= \mathcal{P}\cup \mathcal{S} \cup \bigcup_{\mathcal{T}\in\mathfrak{T}}\mathcal{T},\]
where $\mathcal{P}$ is an open set consisting of periodic orbits,
$\mathcal{S}$ is a finite union of fixed points or saddle
connections, and each $\mathcal{T}\in \mathfrak{T}$ is open and
every positive semi-orbit in $\mathcal{T}$, that is not a separatrix
incoming to a fixed point, is dense in $\mathcal{T}$. It follows
that $\overline{\mathcal{T}}$ is a transitive component of
$(\phi_t)_{t\in\R}$.  Each transitive component
$\overline{\mathcal{T}}$ is a surface with boundary and the boundary
of $\overline{\mathcal{T}}$ is a finite union of fixed points and
loop saddle connections.

\begin{remark}\label{paraind}
Let $X$ be a smooth tangent vector field preserving a volume form
$\omega$ on a surface $M$. A parametrization $\gamma:[a,b]\to M$ of
a curve is called {\em induced} if
\[\int_{\gamma(s)}^{\gamma(s')}i_X\omega=s-s'\text{ for all }s,s'\in[a,b].\]
Let $\gamma:[a,b]\to M$ and
$\widetilde{\gamma}:[\widetilde{a},\widetilde{b}]\to M$ be induced
parameterizations of two curves. Suppose that for every $x\in[a,b]$
the positive semi-orbit of the flow through $\gamma(x)$ hits the
curve $\widetilde{\gamma}$. Denote by
$T_{\gamma\widetilde{\gamma}}(x)\in [\widetilde{a},\widetilde{b}]$
the parameter  and by $\tau_{\gamma\widetilde{\gamma}}(x)>0$ the
time of the first hit. Using Stokes' theorem, it is easy to check
that $T_{\gamma\widetilde{\gamma}}:[a,b]\to
[\widetilde{a},\widetilde{b}]$ is a translation and
$\tau_{\gamma\widetilde{\gamma}}:[a,b]\to\R_+$ is a smooth function.
\end{remark}

Fix $\mathcal{T}\in\mathfrak{T}$ and let $J\subset \mathcal{T}$ be
a transversal smooth curve for $(\phi_t)_{t\in\R}$ such that the
boundary of $J$ consists of two points lying  on an incoming and
an outgoing separatrix respectively, and the segment of each
separatrix between the corresponding boundary point of $J$ and the
fixed point has no intersection with $J$. Let $\gamma:[0,a]\to J$
stand for the induced parametrization such that the boundary
points $\gamma(0)$ and $\gamma(a)$ lie on the incoming and
outgoing separatrixes respectively (see Figure~\ref{rys1}). Set
$I=[0,a)$. We will identify the interval $I$ with the curve $J$.

Denote by $T:=T_{\gamma\gamma}$ the first-return map induced on $J$;
$T$ can be seen as a map $T:I\to I$. By Remark~\ref{paraind},
$T:I\to I$ is an exchange interval transformation. Then
$T=T_{(\pi,\lambda)}$, where $\pi\in\mathcal{S}^0_{\mathcal{A}}$ for
some finite set $\mathcal{A}$ and
$(\pi,\lambda)\in\mathcal{S}^0_{\mathcal{A}}\times\R_+^{\mathcal{A}}$
satisfies Keane's condition. Recall that $l_\alpha$,
$\alpha\in\mathcal{A}$ stand for the left end points of the
exchanged intervals. Let $\mathcal{Z}=\mathcal{F}(\beta)\cap
\overline{\mathcal{T}}$. Since $\overline{\mathcal{T}}$ is a
transitive component, each element of $\mathcal{Z}$ is a
non-degenerated saddle. Let us decompose the set of fixed points
$\mathcal{Z}$ into subsets $\mathcal{Z}_0$, $\mathcal{Z}_+$ and
$\mathcal{Z}_-$ of points $ z\in\mathcal{Z}$ such that $ z$ has no
loop connection, has a loop connection with positive orientation and
has a loop connection with negative orientation respectively. For
each $z\in\mathcal{Z}_+\cup\mathcal{Z}_-$ denote by
$\sigma_{loop}(z)$ the corresponding loop connection.

Denote by $ \underline{z}\in\mathcal{Z}$  the fixed point such
that $\gamma(0)$ belongs to its incoming separatrix
$\sigma^-(\underline{z})$.  Then $\gamma(0)$ is the first backward
intersection  with $J$ of $\sigma^-(\underline{z})$. Set
$\underline{\alpha}=\pi_1^{-1}(1)\in\mathcal{A}$. Then each point
$\gamma(l_\alpha)$ with $\alpha\neq\underline{\alpha}$ corresponds
to the first backward intersection with $J$ of an incoming
separatrix of a fixed point, denoted by $
z_{l_{\alpha}}\in\mathcal{Z}$ (see Figure~\ref{rys1}). The point
$\gamma(l_{\underline{\alpha}})$ corresponds to the second
backward intersection  with $J$ of $\sigma^-(\underline{z})$ and
$Tl_{\underline{\alpha}}=0$.

Denote by $\tau:I\to \R_+$ the first-return time map of the flow
$(\phi_t)_{t\in\R}$ to $J$. This map is well defined and smooth on
the interior of  each interval $I_\alpha$, $\alpha\in\mathcal{A}$,
and $\tau$ has a singularity of logarithmic type at each point
$l_\alpha$, $\alpha\in\mathcal{A}$ (see \cite{Ko}) except for the
right-side of $l_{\underline{\alpha}}$; here the right-sided limit
of $\tau$ exists. Moreover, the flow $(\phi_t)_{t\in\R}$ on
$(\mathcal{T},\nu|_{\mathcal{T}})$ is measure-theoretical
isomorphic to the special flow $T^\tau$. An isomorphism is
established by the map
$\Gamma:I^\tau\to\overline{\mathcal{T}},\;\;\Gamma(x,s)=\phi_s\gamma(x)$.
\begin{figure}[h]
\begin{center}
\resizebox{10cm}{!}{\includegraphics{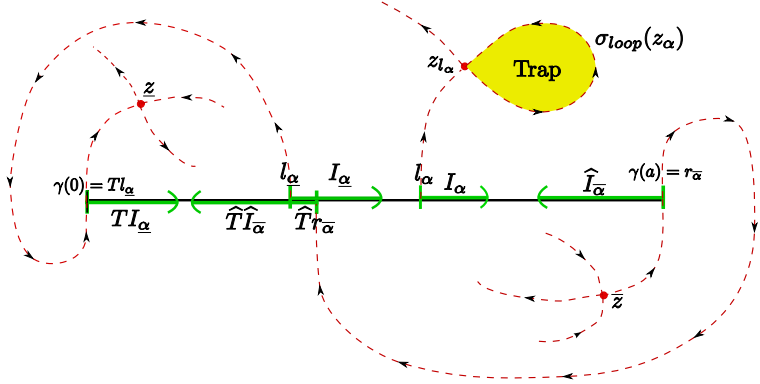}}
\caption{Separatrices of $(\phi_t)$\label{rys1}}
\end{center}
\end{figure}
\subsection{Extensions of multivalued Hamiltonian flows}
\hfill\break Let $f:M\to\R^\ell$ be a smooth function. Let us
consider a system of differential equations on $M\times\R^\ell$ of
the form
\[\left\{\begin{array}{ccc}
\frac{dx}{dt}&=&X(x),\\
\frac{dy}{dt}&=&f(x),
\end{array}\right.\]
for $(x,y)\in M\times\R^\ell$. Then the associated flow
$(\Phi^f_t)_{t\in\R}=(\Phi_t)_{t\in\R}$ on $M\times\R^\ell$ is
given by
\[\Phi_t(x,y)=\left(\phi_tx,y+\int_0^{t}f(\phi_sx)\,ds\right).\]
It follows that $(\Phi_t)_{t\in\R}$ is a skew product flow with
the base flow $(\phi_t)_{t\in\R}$ on $M$ and the cocycle
$F:\R\times M\to\R^\ell$ given by
\[F(t,x)=\int_0^{t}f(\phi_sx)\,ds.\]
Therefore $(\Phi_t)_{t\in\R}$ preserves the product measure
$\nu\times m_{\R^\ell}$. The deviation of the cocycle $F$ was
studied by Forni in \cite{Fo1}, \cite{Fo2} for typical
$(\phi_t)_{t\in\R}$ with no saddle connections.  Recall that the
ergodicity of $(\Phi^f_t)_{t\in\R}$ has been already studied in
\cite{Fa-Le} in the simplest case where $M=\T^2$ and $\ell=1$.

In this section we will study recurrence and ergodic properties of
the flow $(\Phi^f_t)_{t\in\R}$ for functions $f:M\to\R^\ell$ such
that $f(x)=0$ for all $x\in \mathcal{F}(\beta)$. By obvious reason
$(\Phi_t)_{t\in\R}$ will be restricted to the invariant set
$\overline{\mathcal{T}}\times\R^\ell$,
$\mathcal{T}\in\mathfrak{T}$. Let us consider its transversal
submanifold
$J\times\R^\ell\subset\overline{\mathcal{T}}\times\R^\ell$. Note
that every point $(\gamma(x),y)\in \gamma(\Int
I_\alpha)\times\R^\ell$ returns to $J\times\R^\ell$ and the return
time is $\widehat{\tau}(x,y)=\tau(x)$. Denote by
$\varphi:\bigcup_{\alpha\in\mathcal{A}}\Int I_\alpha\to\R^\ell$
the smooth function
\[\varphi(x)=F(\tau(x),\gamma(x))=\int_0^{\tau(x)}f(\phi_s\gamma(x))ds,\;\;\text{ for
}\;\;x\in\bigcup_{\alpha\in\mathcal{A}}\Int I_\alpha.\] Notice
that
\begin{equation}\label{mean}
\int_I\varphi(x)\,dx=\int_{\mathcal{T}}f\,d\nu.
\end{equation}

Let us consider the skew product
$T_\varphi:(I\times\R^\ell,\mu\times
 m_{\R^\ell})\to (I\times\R^\ell,\mu\times  m_{\R^\ell})$,
$T_\varphi(x,y)=(Tx,y+\varphi(x))$ and the special flow
$(T_\varphi)^{\widehat{\tau}}$ built over $T_\varphi$ and under
the roof function $\widehat{\tau}:I\times\R^\ell\to\R_+$ given by
$\widehat{\tau}(x,y)=\tau(x)$.

\begin{lemma}\label{specrep}
The special flow $(T_\varphi)^{\widehat{\tau}}$ is
measure-theoretical isomorphic to the flow $(\Phi_t)$ on
$(\mathcal{T}\times\R^\ell,\nu|_{\mathcal{T}}\times
m_{\R^\ell})$.\bez
\end{lemma}

\begin{remark}
If $\int_{\mathcal{T}}f\,d\nu\neq 0$ then, by (\ref{mean}), the
skew product $T_\varphi$ is dissipative. In view of
Lemma~\ref{specrep},  the flow $(\Phi_t)$ on
$(\mathcal{T}\times\R^\ell,\nu|_{\mathcal{T}}\times m_{\R^\ell})$
is dissipative, as well.

On the other hand, if $\ell=1$ and $(\phi_t)$ on
$(\mathcal{T},\nu|_{\mathcal{T}})$ is ergodic, then
$\int_{\mathcal{T}}f\,d\nu= 0$ implies the recurrence of
$(\Phi_t)$ on $(\mathcal{T}\times\R,\nu|_{\mathcal{T}}\times
m_{\R})$.
\end{remark}

The following lemma will help us to find out further properties of
$\varphi$. Since the proof is rather straightforward and the first
part follows very closely the proof of Proposition~2 in
\cite{Fr-Le3}, we leave it to the reader.
\begin{lemma}\label{lematprzej}
Let $g:[-1,1]\times[-1,1]\to\R$ be a $C^1$-function such that
$g(0,0)=0$. Then the function $\xi:[0,1]\to\R$,
\[\xi(s)=\left\{
\begin{array}{ccc}
\int_{s}^{1}g\left(u,\frac{s}{u}\right)\frac{1}{u}du,&\text{ if }&s>0,\\
\int_{0}^{1}\left(g(u,0)+g(0,u)\right)\frac{1}{u}du,&\text{ if
}&s=0,
\end{array}\right.\]
is absolutely continuous. If additionally $g$ is a $C^2$-function,
$g'(0,0)=0$, and $g''(0,0)=0$, then $\xi'$ is absolutely continuous.
\bez
\end{lemma}

\begin{remark}\label{istzal}
Note that the second conclusion of the lemma becomes false if the
requirement $g''(0,0)=0$ is omitted. Indeed, if $g(x,y)=x\cdot y$
then $\xi(s)=-\log s-1$, $s>0$, is not even bounded.
\end{remark}

For each $z\in\mathcal{Z}_+\cup\mathcal{Z}_-$ choose an element
$u_z$ of the saddle loop $\sigma_{loop}(z)$.

\begin{theorem}\label{twofunkcji}
If $f(x)=0$ for all $x\in \mathcal{F}(\beta)$, then $\varphi$ is
absolutely continuous on each interval $I_{\alpha}$, $\alpha\in
\mathcal{A}$, in particular $\varphi\in \bv(\sqcup_{\alpha\in
\mathcal{A}} I_{\alpha},\R^\ell)$. Moreover,
\[\int_I\varphi'(x)\,dx=\sum_{z\in\mathcal{Z}_+}\int_\R f(\phi_s
u_z) \,ds-\sum_{z\in\mathcal{Z}_-}\int_\R f(\phi_s u_z)\,ds.\]

If additionally $f'(x)=0$ and $f''(x)=0$ for all $x\in
\mathcal{F}(\beta)$, then $\varphi''\in L^1(I,\R^\ell)$, in
particular, $\varphi\in \bv^1(\sqcup_{\alpha\in \mathcal{A}}
I_{\alpha},\R^\ell)$.
\end{theorem}

\begin{proof}
First note that it suffices to consider the case $\ell=1$.  Since
$d\beta=0$, there exists a family of pairwise disjoint open sets
$U_ z\subset M$, $ z\in\mathcal{Z}$ such that $ z\in U_ z$ and
there exists a smooth function $H:\bigcup_{ z\in\mathcal{Z}}U_
z\to\R$ such that $dH=\beta$ on $U_ z$ for every $
z\in\mathcal{Z}$.  By the Morse Lemma, for every $
z\in\mathcal{Z}$ there exist a neighborhood $(0,0)\in V_
z\subset\R^2$ and a smooth diffeomorphism $\Upsilon_ z:V_ z\to U_
z$ such that $\Upsilon_ z(0,0)= z$ and  \[{H}_
z(x,y):=H\circ\Upsilon_ z(x,y)=x\cdot y\text{ for all }(x,y)\in V_
z.\]

Denote by $\omega^ z\in\Omega^2(V_ z)$ the pullback of the form
$\omega$ by $\Upsilon_ z:V_ z\to U_ z$. Since $\omega^ z$ is
non-zero at each point, there exists a smooth non-zero function
$p=p_ z:V_ z\to\R$ such that
\[\omega^ z_{(x,y)}=p(x,y){dx\wedge dy}{}.\]
Let $(\phi^ z_t)$ stand for the pullback of the flow $(\phi_t)$ by
$\Upsilon_ z:V_ z\to U_ z$, i.e. the local flow on $V_z$ given by
$\phi^z_t=\Upsilon_z^{-1}\circ\phi_t\circ\Upsilon_z$. Denote by
$X^z:V_z\to\R^2$ the vector field corresponding to $(\phi^z_t)$.
Then $ dH_ z=\omega^ z(X^ z,\,\cdot\,)$, and hence
\[X^ z(x,y)=
\frac{\left(\frac{\partial H_ z}{\partial y}(x,y),-\frac{\partial
H_ z}{\partial x}(x,y)\right)}{p(x,y)} =\frac{(x,-y)}{p(x,y)}.\]

Let $\delta$ be a positive number such that
$[-\delta,\delta]\times[-\delta,\delta]\subset V_ z$ for every $
z\in\mathcal{Z}$. Let us consider the $C^{\infty}$--curves
$\gamma^{\pm,0}_ z,\gamma^{\pm,1}_ z:[-\delta^2,\delta^2]\to M$
given by
\[\gamma^{\pm,0}_ z(s)=\Upsilon_ z(\pm
s/\delta,\pm\delta),\;\;\gamma^{\pm,1}_ z(s)=\Upsilon_
z(\pm\delta,\pm s/\delta).\] Notice that $\gamma^{\pm,i}_ z$
establishes an induced parametrization for the form $\omega(x,y)$
and the vector field $X$. Indeed, we have for every
$s\in[-\delta^2,\delta^2]$ and $i=0,1$,
\begin{eqnarray*}\int^{\gamma^{\pm,i}_ z(s)}_{\gamma^{\pm,i}_ z(0)}\beta=
\int^{\gamma^{\pm,i}_ z(s)}_{\gamma^{\pm,i}_ z(0)}dH=
H(\gamma^{\pm,i}_ z(s))-H(\gamma^{\pm,i}_ z(0))=\pm s/\delta \cdot
\pm\delta=s.
 \end{eqnarray*}

We consider the functions $\tau^{\pm}_ z$ and $\varphi^{\pm}_ z$
from $[-\delta^2,0)\cup(0,\delta^2]$ to $\R$, where $\tau^{\pm}_
z(x)$ is the  exit time  of the point $\gamma^{\pm,0}_ z(x)$ for
the flow $({\phi}_{t})$ from the set $\Upsilon_
z([-\delta,\delta]\times[-\delta,\delta])$ and
\[\varphi^{\pm}_ z(x)=\int_{0}^{\tau^{\pm}_ z(x)}f(\phi_s\gamma^{\pm,0}_ z x)ds.\]
Note that $\tau^{\pm}_ z(x)$ is the passage time from $(\pm
x/\delta,\pm\delta)$ to $(\pm\sgn(x)\delta,\pm\sgn(x)x/\delta)$
for the local flow $(\phi^ z_{t})$. Let $f_ z:V_ z\to\R$ be given
by $f_ z=f\circ \Upsilon_ z$. By assumption, $f_ z$ is a smooth
function such that $f_ z(0,0)=0$. Furthermore,
\[\varphi^{\pm}_ z(x)=\int_{0}^{\tau^{\pm}_ z(x)}f_ z(\phi_s^ z(\pm x/\delta,\pm\delta))ds.\]
Let $(x_s,y_s)=\phi_s^ z(\pm x/\delta,\pm\delta)$. Then
\begin{equation}\label{rownanie}
\left(\frac{d}{ds}x_s,\frac{d}{ds}y_s\right)=X^z(x_s,-y_s)=\frac{(x_s,-y_s)}{p(x_s,y_s)},
\end{equation}
and hence
\[x_s\cdot y_s=H_ z(x_s,y_s)=H_ z(x_0,y_0)=H_ z(\pm x/\delta,\pm\delta)=x.\]
Since $x\neq 0$, it follows that $x_s\neq 0$ for all $s\in\R$. By
using the substitution $u = x_s$, we obtain $du = \frac{d}{ds}x_s
ds=\frac{x_s}{p(x_s,y_s)} ds$ and
\begin{eqnarray*}\varphi^{\pm}_ z(x)&=&\int_{0}^{\tau^{\pm}_ z(x)}f_
z(x_s,y_s)ds=\int_{0}^{\tau^{\pm}_ z(x)}f_
z\left(x_s,\frac{x}{x_s}\right)ds\\&=&\int_{\pm
x/\delta}^{\pm\sgn(x)\delta}f_
z\left(u,\frac{x}{u}\right)p\left(u,\frac{x}{u}\right)\frac{du}{u}.
\end{eqnarray*}
By Lemma~\ref{lematprzej}, the functions $\varphi^{\pm}_
z:[-\delta^2,0)\cup(0,\delta^2]\to\R$ is absolutely continuous and
\[\lim_{x\to 0^+}\varphi^{\pm}_
z(x)=\int_{0}^{\pm\delta}f_ z(u,0)p(u,0)\frac{du}{u}+
\int_{0}^{\pm\delta}f_ z(0,u)p(0,u)\frac{du}{u}.\]  It follows
that
\begin{equation}\label{granpr}
\lim_{x\to 0^+}\varphi^{\pm}_
z(x)=\int_{0}^{+\infty}f(\phi_s\gamma^{\pm,0}_ z
0)ds+\int_{-\infty}^{0}f(\phi_s\gamma^{\pm,1}_ z 0)ds.
\end{equation}
Similar arguments to those above show that
\begin{equation}\label{granpr1}
\lim_{x\to 0^-}\varphi^{\pm}_
z(x)=\int_{0}^{+\infty}f(\phi_s\gamma^{\pm,0}_ z
0)ds+\int_{-\infty}^{0}f(\phi_s\gamma^{\mp,1}_ z 0)ds.
\end{equation}
In view of Remark~\ref{paraind}, we conclude that $\varphi:I\to\R$
is absolutely continuous  on each interval $I_\alpha$,
$\alpha\in\mathcal{A}$ and
\begin{equation}\label{fina1}
\varphi_+(l_\alpha)=\int_0^{+\infty}f(\phi_s\gamma(l_\alpha))ds+\int_{-\infty}^{0}f(\phi_s\gamma(T
l_\alpha))ds,
\end{equation}
whenever $\alpha\neq\underline{\alpha}$ and
$z_{l_\alpha}\in\mathcal{Z}_-\cup\mathcal{Z}_0$. If
$\alpha\neq\underline{\alpha}$ and $z_{l_\alpha}\in\mathcal{Z}_+$,
then computing $\varphi_+(l_\alpha)$ we have cover also a distance
along the loop separatrix $\sigma_{loop}(z_{l_\alpha})$, so that
\begin{equation}\label{fina2}
\varphi_+(l_\alpha)=\int_0^{+\infty}f(\phi_s\gamma(l_\alpha))ds+\int_{-\infty}^{0}f(\phi_s\gamma(T
l_\alpha))ds+\int_{-\infty}^{+\infty}f(\phi_s u_{z_{l_\alpha}})ds.
\end{equation}

Moreover, if $f'(z_{l_\alpha})=0$ and $f''(z_{l_\alpha})=0$ then
the derivative $\varphi''$ is integrable on a neighborhood of
$l_\alpha$. It follows that if $f'(z)=0$ and $f''(z)=0$ for each
$z\in\mathcal{F}(\beta)$ then $\varphi''$ is integrable.

 If $\alpha=\underline{\alpha}$, then, by definition, the
positive semi-orbit through $\gamma(l_\alpha)$ returns to $\gamma$
before approaching the fixed point $\underline{z}$. It follows
that
\begin{equation}\label{fina3}
\varphi_+(l_\alpha)=\int_0^{\tau_{\gamma\gamma}(l_\alpha)}f(\phi_s\gamma(l_\alpha))ds.
\end{equation}
 Let
$\overline{\alpha}=\pi^{-1}_0(d)$, i.e.\
$r_{\overline{\alpha}}=|I|$. Similar arguments to those used for
the right-sided limits show that for every
$\alpha\neq\overline{\alpha}$ we have
\begin{align}\label{fina4}\varphi_-(r_\alpha)&=\int_0^{+\infty}f(\phi_s\gamma(r_\alpha))ds+\int_{-\infty}^{0}f(\phi_s
\gamma(\widehat{T}r_\alpha))ds\text{ if
}z_{r_\alpha}\in\mathcal{Z}_0\cup\mathcal{Z}_+,\\
\label{fina5}\varphi_-(r_\alpha)&=\int_0^{+\infty}f(\phi_s\gamma(r_\alpha))ds+\int_{-\infty}^{0}f(\phi_s
\gamma(\widehat{T}r_\alpha))ds+\int_{-\infty}^{+\infty}f(\phi_s
u_{z_{r_\alpha}})ds,
\end{align}
if $z_{r_\alpha}\in\mathcal{Z}_-$. Moreover, since
$\gamma(r_{\overline{\alpha}})$ lies on an outgoing separatrix,
the positive semi-orbit through $\gamma(r_\alpha)$ returns to the
curve $\gamma$, so that
\begin{equation}\label{fina6}
\varphi_-(r_{\overline{\alpha}})=\varphi_-(|I|)=
\int_0^{\tau_{\gamma\gamma}(r_{\overline{\alpha}})}f(\phi_s\gamma(r_{\overline{\alpha}}))ds.
\end{equation}
In view of (\ref{fina1}) - (\ref{fina6}), we have
\begin{eqnarray}\label{dlugi}
\begin{split}
\int_I\varphi'(x)\,dx=&\sum_{\alpha\in\mathcal{A}}\int_{I_\alpha}\varphi'(x)\,dx
=\sum_{\alpha\in\mathcal{A}}\left(\varphi_-(r_\alpha)-\varphi_+(l_\alpha)\right)\\
=&\sum_{z\in\mathcal{Z}_-}\int_{-\infty}^{+\infty}f(\phi_s
u_z)ds-\sum_{z\in\mathcal{Z}_+}\int_{-\infty}^{+\infty}f(\phi_s
u_z)ds\\&+ \sum_{\alpha\in\mathcal{A},\alpha\neq
\overline{\alpha}}\int_{-\infty}^{0}f(\phi_s
\gamma(\widehat{T}r_\alpha))ds-\sum_{\alpha\in\mathcal{A},\alpha\neq
\underline{\alpha}}\int_{-\infty}^{0}f(\phi_s\gamma(T
l_\alpha))ds\\&+ \sum_{\alpha\in\mathcal{A},\alpha\neq
\overline{\alpha}}\int^{+\infty}_{0}f(\phi_s
\gamma(r_\alpha))ds-\sum_{\alpha\in\mathcal{A},\alpha\neq
\underline{\alpha}}\int^{+\infty}_{0}f(\phi_s\gamma(l_\alpha))ds\\&
+\int_0^{\tau_{\gamma\gamma}(r_{\overline{\alpha}})}f(\phi_s\gamma(r_{\overline{\alpha}}))ds-
\int_0^{\tau_{\gamma\gamma}(l_{\underline{\alpha}})}f(\phi_s\gamma(l_{\underline{\alpha}}))ds.
\end{split}
\end{eqnarray}
Since $\underline{\alpha}=\pi^{-1}_1(1)$ and
$\overline{\alpha}=\pi^{-1}_0(d)$, in view of (\ref{zbzero}),
(\ref{zbjeden}),  we have
\begin{align*}\{r_\alpha:\alpha\in\mathcal{A},\;\alpha\neq
\overline{\alpha}\}&=
\{r_\alpha:\alpha\in\mathcal{A},\;\pi_0(\alpha)\neq d\}
=\{l_\alpha:\alpha\in\mathcal{A},\;\pi_0(\alpha)\neq 1\},\\
\{Tl_\alpha:\alpha\in\mathcal{A},\;\alpha\neq\underline{\alpha}\}&=
\{Tl_\alpha:\alpha\in\mathcal{A},\;\pi_1(\alpha)\neq 1\}=
\{\widehat{T}r_\alpha:\alpha\in\mathcal{A},\;\pi_1(\alpha)\neq
d\}.
\end{align*}
Moreover, $l_{\pi_0^{-1}(1)}=0=Tl_{\underline{\alpha}}$ and
$\widehat{T}r_{\pi_1^{-1}(d)}=|I|=r_{\overline{\alpha}}$. It
follows that
\begin{align*}
\sum_{\alpha\in\mathcal{A},\alpha\neq
\overline{\alpha}}&\int_{-\infty}^{0}f(\phi_s
\gamma(\widehat{T}r_\alpha))ds-\sum_{\alpha\in\mathcal{A},\alpha\neq
\underline{\alpha}}\int_{-\infty}^{0}f(\phi_s\gamma(T
l_\alpha))ds\\
=&\int_{-\infty}^{0}f(\phi_s
\gamma(r_{\overline{\alpha}}))ds-\int_{-\infty}^{0}f(\phi_s
\gamma(\widehat{T}r_{\overline{\alpha}}))ds,\\
\sum_{\alpha\in\mathcal{A},\alpha\neq
\overline{\alpha}}&\int^{+\infty}_{0}f(\phi_s
\gamma(r_\alpha))ds-\sum_{\alpha\in\mathcal{A},\alpha\neq
\underline{\alpha}}\int^{+\infty}_{0}f(\phi_s\gamma(l_\alpha))ds\\
=&\int^{+\infty}_{0}f(\phi_s\gamma(l_{\underline{\alpha}}))ds-
\int^{+\infty}_{0}f(\phi_s\gamma(Tl_{\underline{\alpha}}))ds.
\end{align*}
Since the negative semi-orbit of
$\widehat{T}r_{\overline{\alpha}}$ visits $r_{\overline{\alpha}}$
before approaching the fixed point $\overline{z}$ and the positive
semi-orbit of $l_{\underline{\alpha}}$ visits
$Tl_{\underline{\alpha}}$ before approaching the fixed point
$\underline{z}$ (see Figure~\ref{rys1}), we have
\begin{align*}
\int_{-\infty}^{0}f(\phi_s
\gamma(\widehat{T}r_{\overline{\alpha}}))ds-\int_{-\infty}^{0}f(\phi_s
\gamma(r_{\overline{\alpha}}))ds&=\int_{0}^{\tau_{\gamma\gamma}(r_{\overline{\alpha}})}f(\phi_s
\gamma(r_{\overline{\alpha}}))ds,\\
\int^{+\infty}_{0}f(\phi_s\gamma(l_{\underline{\alpha}}))ds-\int^{+\infty}_{0}f(\phi_s\gamma(Tl_{\underline{\alpha}}))ds
&=\int^{\tau_{\gamma\gamma}(l_{\underline{\alpha}})}_{0}f(\phi_s\gamma(l_{\underline{\alpha}}))ds.
\end{align*}
In view  of (\ref{dlugi}), it follows that
\[\int_I\varphi'(x)\,dx=\sum_{z\in\mathcal{Z}_-}\int_{-\infty}^{+\infty}f(\phi_s
u_z)ds-\sum_{z\in\mathcal{Z}_+}\int_{-\infty}^{+\infty}f(\phi_s
u_z)ds.\]
\end{proof}

\begin{remark}
Notice that, in view of Remark~\ref{istzal}, the assumption on the
vanishing of derivatives of $f$ at fixed points is necessary to
control the smoothness of $\varphi_f$.
\end{remark}

\begin{theorem}
Suppose that the IET $T$ is of periodic type. Let $f:M\to\R^\ell$
be a smooth function such that $f(x)=0$ for all
$x\in\mathcal{F}(\beta)$ and $\int_{\mathcal{T}}f\,d\nu=0$. If
$\theta_2(T)/\theta_1(T)<1/\ell$ then the flow $(\Phi_t)$ on
$\mathcal{T}\times\R^\ell$ is conservative.
\end{theorem}

\begin{proof}
By Lemma~\ref{specrep}, Theorem~\ref{twofunkcji} and (\ref{mean}),
the flow $(\Phi_t)$ on $\mathcal{T}\times\R$ is isomorphic to a
special flow built over the skew product $T_\varphi$, where
$\varphi:I\to\R^\ell$ is a function of bounded variation with zero
mean. In view of Corollary~\ref{correc}, the skew product is
conservative. Now the conservativity of $(\Phi_t)$ follows from
Remark~\ref{special}.
\end{proof}

Let $g$ be a Riemann metric on $M$. Let us consider $1$-form
$\vartheta^\beta\in\Omega^1(M\setminus \mathcal{F}(\beta))$ on
$M\setminus \mathcal{F}(\beta)$ defined by
\[\vartheta^\beta_xY=\frac{g_x(Y,X(x))}{g_x(X(x),X(x))}.\]
Then $\vartheta^\beta_xX(x)=1$, and hence
\[\int_{\{\phi_sx:s\in[a,b]\}}f\cdot\vartheta^\beta=\int_a^bf(\phi_sx)\cdot\vartheta^\beta_{\phi_sx}(X(\phi_sx))ds
=\int_a^bf(\phi_sx)ds.\] It follows that
\[\int_{\partial \mathcal{T}}f\cdot\vartheta^\beta=\sum_{z\in\mathcal{Z}_-}\int_{-\infty}^{+\infty}f(\phi_s
u_z)ds-\sum_{z\in\mathcal{Z}_+}\int_{-\infty}^{+\infty}f(\phi_s
u_z)ds=\int_I\varphi'(x)\,dx.\]

\begin{theorem}\label{thmmain1}
Suppose that the IET $T$ is of periodic type. Let $f:M\to\R$ be a
smooth function such that $f(x)=0$, $f'(x)=0$ and $f''(x)=0$ for
all $x\in\mathcal{F}(\beta)$,
\[\int_{\mathcal{T}}f\,d\nu=0\text{ and }\int_{\partial \mathcal{T}}f\cdot\vartheta^\beta\neq 0.\]
Then the corresponding flow $(\Phi_t)$ on $\mathcal{T}\times\R$ is
ergodic.
\end{theorem}

\begin{proof}
By Lemma~\ref{specrep},  Theorem~\ref{twofunkcji}  and
(\ref{mean}), the flow $(\Phi_t)$ on $\mathcal{T}\times\R$ is
isomorphic to a special flow built over the skew product
$T_\varphi$,  where $\varphi\in
\bv^1(\sqcup_{\alpha\in\mathcal{A}}I_\alpha)$ has zero mean and
\[s(\varphi)=\int_I\varphi'(x)dx=\int_{\partial
\mathcal{T}}f\cdot\vartheta^\beta\neq 0.\]  By Lemma~\ref{lempl},
the cocycle $\varphi$ is cohomologous to a cocycle $\varphi_{pl}\in
\pl(\sqcup_{\alpha\in\mathcal{A}}I_\alpha)$ with
$\int\varphi_{pl}(x)\,dx=0$ and $s(\varphi_{pl})=s(\varphi)\neq 0$.
In view of Theorem~\ref{kawallin}, the skew product
$T_{\varphi_{pl}}$ is ergodic. Consequently, the skew product
$T_\varphi$ and hence the flow $(\Phi_t)$ on $\mathcal{T}\times\R$
are ergodic, by Remark~\ref{special}.
\end{proof}

Suppose that the IET $T$ is of periodic type and
$\theta_2(T)/\theta_1(T)<1/\ell$ ($\ell\geq 2$). Let
$f:M\to\R^\ell$ be a smooth function such that $f(x)=0$, $f'(x)=0$
and $f''(x)=0$ for all $x\in\mathcal{F}(\beta)$,
\[\int_{\mathcal{T}}f\,d\nu=0\text{ and }\R^\ell\ni v=\int_{\partial \mathcal{T}}f\cdot\vartheta^\beta\neq 0.\]
Let $a_2,\ldots,a_\ell$ be a basis of the subspace $\{v\}^\perp$
and let $f_a:M\to\R^{\ell-1}$ be given by $f_a=(\langle a_2,
f\rangle,\ldots,\langle a_\ell, f\rangle)$.

\begin{theorem}\label{thmmain2}
If the flow $(\Phi^{f_a}_t)$ on $\mathcal{T}\times\R^{\ell-1}$ is
ergodic then $(\Phi^{f}_t)$ on $\mathcal{T}\times\R^\ell$ is
ergodic.
\end{theorem}

\begin{proof}
Without loss of generality we can assume that $v=(1,0,\ldots,0)$,
$a_2=(0,1,0,\ldots,0),\ldots,a_\ell=(0,\ldots,0,1)$. Then
$\varphi=(\varphi_1,\varphi_2)$, where $\varphi_1:I\to\R$,
$\varphi_2:I\to\R^{\ell}$ are functions with
$\int_I\varphi_1'(x)\,dx\neq 0$ and $\int_I\varphi_2'(x)\,dx=0$.
Applying Proposition~\ref{cohcon} we can pass to cohomological
cocycles which are piecewise linear with constant slope. Now we
can apply Theorem~\ref{kawallin2} to prove the ergodicity of
$T_\varphi$ which implies the ergodicity of the flow $(\Phi_t)$ on
$\mathcal{T}\times\R^\ell$.
\end{proof}

\vskip 3mm
\section{Examples of ergodic extensions of multivalued
Hamiltonian flows} \label{constmulti} In this section we will
apply Theorems~\ref{ergpiconsper}, \ref{thmmain1} and
\ref{thmmain2} to construct explicit examples of recurrent ergodic
extensions of multivalued Hamiltonian flows.

\subsection{Construction of multivalued Hamiltonians}

\hfill \break Let $T=T_{(\pi,\lambda)}:I\to I$ be an arbitrary IET
satisfying the Keane condition. We begin this section by recalling
a recipe for constructing multivalued Hamiltonians such that the
corresponding flows have special representation over $T$. Let us
start from any translation surface $(M,\alpha)$ built over $T$ by
applying the zipped rectangles procedure (see \cite{Ve1} or
\cite{ViB}). Denote by $\Sigma=\{p_1,\ldots,p_{\kappa}\}$ the set
of singular point of $(M,\alpha)$. Let $J\subset M\setminus
\Sigma$ be a curve transversal to the vertical flow and such that
the first return map to $J$ is $T$. We will constantly identify
$J$ with the interval $I$. Denote by $S\subset M$ the union of
segments of all separatrices connecting singular points with $J$.

We will consider so called  regular adapted coordinates on
$M\setminus\Sigma$, this is coordinates $\zeta$ relatively to which
$\alpha_\zeta=d\zeta$. If $p\in \Sigma$ is a singular point with
multiplicity $m\geq 1$ then we consider singular adapted coordinates
around $p$, this is coordinates $\zeta$ relatively to which
$\alpha_\zeta=d\frac{\zeta^{m+1}}{m+1}=\zeta^m\,d\zeta$. Then all
changes of regular coordinates are given by translations. If
$\zeta'$ is a regular adapted coordinate and $\zeta$ is a singular
adapted coordinate, then $\zeta'=\zeta^{m+1}/(m+1)+c$. Let us
consider the vertical vector field $Y$ and the associated vertical
flow $(\psi_t)_{t\in\R}$ on $(M,\alpha)$, this is $\alpha_{x}Y(x)=i$
and  $\frac{d}{dt}\psi_tx=Y(\psi_tx)$ for $x\in M\setminus\Sigma$.
Then for a regular adapted coordinate $\zeta$ we have $Y(\zeta)=i$
and  $\psi_t\zeta=\zeta+it$. Moreover, for a singular adapted
coordinate $\zeta$ we have $\zeta^mY(\zeta)=i$, and hence
$Y(\zeta)=\frac{i\overline{\zeta}^m}{|\zeta|^{2m}}$.

For each $\vep>0$ and $p\in\Sigma$ denote by $B_\vep(p)$ the
$\vep$ open ball of center $p$ and let
$g=g_{\vep}:[0,+\infty)\to[0,1]$ be a monotonic
$C^\infty$-function such that $g(x)=x$ for $x\in[0,\vep]$ and
$g(x)=1$ for $x\geq 2\vep$. Fix $\vep>0$ small enough. In what
follows, we will deal with regular adapted coordinates on
$M\setminus\bigcup_{p\in\Sigma}B_{2\vep}(p)$ and singular adapted
coordinates on $B_{3\vep}(p)$ for $p\in\Sigma$. Let us consider a
 tangent $C^\infty$-vector field $\tilde{Y}$ on $M$ such that in adapted
coordinates $\zeta$ we have
\[\tilde{Y}(\zeta)=\left\{
\begin{matrix} Y(\zeta)=i,&\text{ on }& M\setminus\bigcup_{p\in\Sigma}B_{2\vep}(p),\\
\frac{g(|\zeta|)^{2m}i\overline{\zeta}^m}{|\zeta|^{2m}},&\text{ on
}&B_{3\vep}(p),\;p\in\Sigma.
\end{matrix}\right.\]

Denote by $(\tilde{\psi}_t)_{t\in\R}$ the associated $C^\infty$-flow
on $M$. Then $(\tilde{\psi}_t)_{t\in\R}$ on $M\setminus\Sigma$ is
obtained by a $C^{\infty}$ time change in the vertical flow
$(\psi_t)_{t\in\R}$, and $(\tilde{\psi}_t)_{t\in\R}$ coincides with
$(\psi_t)_{t\in\R}$ on $M\setminus\bigcup_{p\in\Sigma}B_{2\vep}(p)$.

Denote by $\tilde{\omega}$ the symplectic $C^\infty$-form  on $M$
such that in adapted coordinates $\zeta=x+iy$ we have
\[\tilde{\omega}_\zeta=\left\{
\begin{matrix} dx\wedge dy,&\text{ on }& M\setminus\bigcup_{p\in\Sigma}B_{2\vep}(p),\\
\frac{|\zeta|^{2m}}{g(|\zeta|)^{2m}}dx\wedge dy,&\text{ on
}&B_{3\vep}(p),\;p\in\Sigma.
\end{matrix}\right.\]

Let us consider the $C^\infty$ $1$-form on $M$ given by
$\tilde{\beta}=i_{\tilde{Y}}\tilde{\omega}$. Then in adapted
coordinates $\zeta=x+iy$ we have
\[\tilde{\beta}_{\zeta}=\left\{
\begin{matrix} -dx,&\text{ on }& M\setminus\bigcup_{p\in\Sigma}B_{2\vep}(p),\\
-\Re{\zeta}^mdx+\Im{\zeta}^mdy,&\text{ on
}&B_{3\vep}(p),\;p\in\Sigma.
\end{matrix}\right.\]

By Cauchy-Riemann equations, $\frac{\partial}{\partial
y}\Re{\zeta}^m+\frac{\partial}{\partial x}\Im{\zeta}^m=0$, and
hence $d\tilde{\beta}=0$. Therefore $(\tilde{\psi}_t)_{t\in\R}$ is
a multivalued Hamiltonian $C^\infty$-flow whose orbits on
$M\setminus\Sigma$ coincide with orbits of the vertical flow. It
follows that $(\tilde{\psi}_t)_{t\in\R}$ has a special
representation over the IET $T_{(\pi,\lambda)}$. If the
multiplicity of a singularity $p\in\Sigma$ is equal to $m=1$ then
in singular adapted coordinates $\zeta=x+iy$ on $B_{\vep}(p)$ we
have $\tilde{\beta}=-xdx+ydy$, and hence the multivalued
Hamiltonian $\widehat{H}$ is equal to
$\widehat{H}(x,y)=(y^2-x^2)/2+const$, so $p$ is a non-degenerated
critical point of $\widehat{H}$.

Let us consider the symplectic form $\nu=ce^{2x}dx\wedge dy$, $c\neq
0$ on the disk $D=\{(x,y)\in\R^2:(x-1/2)^2+y^2\leq (3/2)^2\}$ and
the Hamilton differential equation
\[\frac{dx}{dt}=-y,\;\;\frac{dy}{dt}=x(x-1)+y^2.\]
Then the function $-ce^{2x}((x-1)^2+y^2)/2$ is the corresponding
Hamiltonian. Denote by $(h_t)$ the associated local Hamiltonian
flow. It has two critical points: $z_0=(0,0)$ is a non-degenerated
saddle and $(1,0)$ is a center. The point $(0,0)$ has a loop
saddle connection which coincides with the curve
$e^{2x}((x-1)^2+y^2)=1$, $x\geq 0$. Inside this loop connection
all trajectories of $(h_t)$ are periodic (see Figure~\ref{rys2}).
Such domains  are called traps.
\begin{figure}[h]
\begin{center}
\resizebox{6cm}{!}{\includegraphics{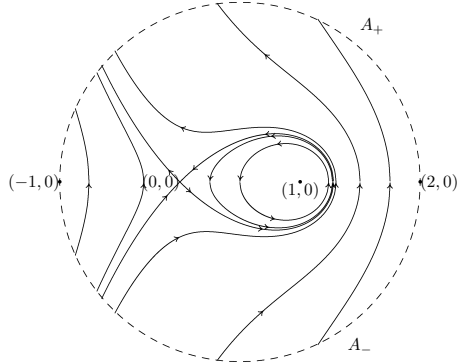}} \caption{The
phase portrait of the Hamiltonian flow for $c>0$\label{rys2}}
\end{center}
\end{figure}
\noindent It is easy to show that the corresponding Hamiltonian
vector field $Z$ does not vanish on $\partial D$, and that it has
two contact points $(2,0)$ and $(-1,0)$ and two arcs $A_+$ and
$A_-$ connecting them with the same length (with respect to
$\nu$). Let us cut out from $M\setminus S$ a disk $B_{\delta}(q)$,
$\delta>0$ such that $B_{2\delta}(q)$ is disjoint from the
transversal curve $J$ and from each $B_{3\vep}(p)$, $p\in\Sigma$.
The vector field $\tilde{Y}$ does not vanish on $\partial
B_{\delta}(q)$, has two contact points and two arcs $\tilde{A}_+$
and $\tilde{A}_-$ connecting them with the same length (with
respect to $\tilde{\omega}$). Choose $c\neq 0$ such that all four
arcs ${A}_+$, ${A}_-$, $\tilde{A}_+$ and $\tilde{A}_-$ have the
same length. Note that $c$ is unique up to sign. Therefore, by
Lemma~1 in \cite{Bl}, there exists a $C^{\infty}$-diffeomorphism
$f:\partial D\to\partial B_{\delta}(q)$, a symplectic
$C^{\infty}$-form $\omega$ on $(M\setminus B_{\delta}(q))\cup_f D$
and a tangent $C^{\infty}$ vector field $X$ such that
\begin{itemize}
\item $\mathcal{L}_{X}\omega=0$;
\item $\omega=\tilde{\omega}$ and $X=\tilde{Y}$ on $M\setminus B_{2\delta}(q)$;
\item $\omega=\nu$ and $X=Z$ on $D$;
\item the orbits of $X$ on $M\setminus B_{2\delta}(q)$ are pieces
of orbits of the flow $(\tilde{\psi}_t)$.
\end{itemize}
Of course, $(M\setminus B_{\delta}(q))\cup_f D$ is diffeomorphic
to $M$, and so the vector field $X$ and the symplectic form
$\omega$ can be considered on $M$. Since
$d(i_{X}\omega)=\mathcal{L}_{X}\omega=0$, $X$ is a Hamiltonian
vector field with respect to $\omega$. Denote by
$(\phi_t)_{t\in\R}$ the Hamilton flow associated to $X$. Since the
dynamics of $(\phi_t)_{t\in\R}$ and $(\psi_t)_{t\in\R}$ coincide
on $M\setminus(\bigcup_{p\in\Sigma}B_{2\vep}(p)\cup
B_{2\delta}(q))$ and $J\subset
M\setminus(\bigcup_{p\in\Sigma}B_{2\vep}(p)\cup B_{2\delta}(q))$,
the first return map to $J$ for $(\phi_t)_{t\in\R}$ is $T$. Denote
by $\gamma\in I$ the first backward intersection  with $J$ of the
separatrix incoming to $z_0$. Note that $\gamma$ may be an
arbitrary point of $I$ different from the ends of the exchanged
intervals. It suffices to choose the point $q\in M\setminus S$ and
$\delta>0$ carefully enough. Recall that the saddle point $z_0$
has a loop connection which will be denoted by
$\sigma_{loop}(z_0)$. Then the orientation of $\sigma_{loop}(z_0)$
is positive if $c>0$ and negative if $c<0$.

\begin{remark}\label{istpotoku}
We can repeat  the procedure of producing new loop connections
(positively or negatively oriented) as many times as we want.
Therefore for any collection of distinct points
$\{\gamma_1,\ldots,\gamma_s\}\subset\bigcup_{\alpha\in\mathcal{A}}\Int
I_\alpha$ and $\delta>0$ small enough we can construct a
multivalued Hamiltonian flow  $(\phi_t)_{t\in\R}$ on $M$ which has
$s$ non-degenerated saddle critical points $z_1,\ldots,z_s$ such
that each $z_i$ has a loop connection $\sigma_{loop}(z_i)$
included in $B_{\delta}(z_i)$ for $i=1,\ldots,s$. Moreover,
$(\psi_t)_{t\in\R}$ and $(\phi_t)_{t\in\R}$ coincide on
$M\setminus(\bigcup_{p\in\Sigma}B_{2\vep}(p)\cup \bigcup_{i=1}^s
B_{2\delta}(z_i)))$ and $\gamma_i\in I$ corresponds to the first
backward intersection with $J$ of the separatrix incoming to $z_i$
for $i=1,\ldots,s$.

We denote by $Trap_i$ the trap corresponding to $z_i$, by
$\epsilon(z_i)\in\{-,+\}$ the sign of the orientation of
$\sigma_{loop}(z_i)$ for $i=1,\ldots,s$, and by $\mathcal{T}$ the
surface $M$ with the interior of the traps $Trap_i$, $i=1,\ldots,s$
removed.
\end{remark}

\begin{remark}\label{geom}
Choose $0<\delta'<\delta$  such that
$\sigma_{loop}(z_i)\cap(M\setminus B_{\delta'}(z_i))\neq
\emptyset$ for $i=1,\ldots,s$. Let $f:M\to\R$ be a
$C^{\infty}$-function with $\int_{\mathcal{T}}f\omega=0$ and such
that $f$ vanishes on each $B_{2\vep}(p)$, $p\in\Sigma$ and
$B_{\delta}(z_i)$, $i=1,\ldots,s$. Then the corresponding function
\[\varphi_f : I\to\R,\;\;\;
\varphi_f(x)=\varphi(x)=\int_{0}^{\tau(x)}f(\phi_tx)\,dt\]
($\tau:I\to \R_+$ is the first-return time map of the flow
$(\phi_t)_{t\in\R}$ to $J$) can be extended to a
$C^{\infty}$-function on the closure of any interval of the
partition
$\mathcal{P}(\{l_{\alpha}:\alpha\in\mathcal{A}\}\cup\{\gamma_i:i=1,\ldots,s\})$.
Moreover,
\begin{equation}\label{relfphi}
d_i(f):=\varphi_+(\gamma_i)-\varphi_-(\gamma_i)=\epsilon(z_i)
\int_{-\infty}^{+\infty}f(\phi_tu_{z_i})\,dt \text{ and
}s(\varphi)=\sum_{i=1}^sd_i(f),
\end{equation}
where $u_{z_i}$ is an arbitrary point of $\sigma_{loop}(z_i)$ for
$i=1,\ldots,s$.
\end{remark}

\begin{lemma}\label{sskokow}
 For every $(d_1,\ldots,d_s)\in\R^s$ there exists a $C^{\infty}$-function
$f:M\to\R$ which vanishes on a neighborhood of each fixed point of
$(\phi_t)$ such that $\int_{\mathcal{T}}f\omega=0$ and
$(d_1(f),\ldots,d_s(f))=(d_1,\ldots,d_s)$.
\end{lemma}

\begin{proof}Let us start from  $f\equiv 0$.  Since $\sigma_{loop}(z_i)\cap
(B_{2\delta'}(z_i)\setminus B_{\delta'}(z_i))\neq \emptyset$, we
can modify $f$ smoothly on  $ B_{2\delta'}(z_i)\setminus
B_{\delta'}(z_i)$ such that
\[\epsilon(z_i)\int_{-\infty}^{+\infty}f(\phi_tu_{z_i})\,dt=d_i\text{ and }
\int_{(B_{2\delta'}(z_i)\setminus B_{\delta'}(z_i))\setminus
Trap_i}f\omega=0\] for $i=1,\ldots,s$.  In view of
(\ref{relfphi}), it follows that $d_i(f)=d_i$ for $i=1,\ldots,s$.
Moreover,
\[\int_{\mathcal{T}}d\omega=\sum_{i=1}^s\int_{(B_{2\delta'}(z_i)\setminus B_{\delta'}(z_i))\setminus
Trap_i}f\omega=0.\]
\end{proof}

\begin{lemma}\label{lematkoh}
For every $h\in H_\pi$ there exists a $C^{\infty}$-function
$f:M\to\R$ such that
$\varphi_f=\sum_{\alpha\in\mathcal{A}}h_\alpha\chi_{I_\alpha}$
(cf.\ Remark~\ref{geom}). If $h\in H_\pi\cap\Gamma_0$ then
$\int_{\mathcal{T}}f\omega=0$.
\end{lemma}

\begin{proof} Following \cite{ViB}, for every $\alpha\in\mathcal{A}$ denote by $[v_\alpha]\in
H_1(M,\R)$ the homology class of any closed curve $v_\alpha$
formed by a segment of the orbit for $(\psi_t)_{t\in\R}$ starting
at any point $x\in\Int I_{\alpha}$ and ending at $Tx$ together
with the segment of $J$ that joins $Tx$ and $x$. Let
$\Psi:H^1(M,\R)\to\R^{\mathcal{A}}$ be given by
$\Psi([\varrho])=(\int_{v_\alpha}\varrho)_{\alpha\in\mathcal{A}}$.
By Lemma~2.19 in \cite{ViB}, the map $\Psi:H^1(M,\R)\to H_\pi$
establishes the isomorphism of linear spaces. Therefore for every
$h\in H_\pi$ there exists a closed $1$-form $\varrho$ such that
$\Psi([\varrho])=h$ and $\varrho$ vanishes on an open neighborhood
of $J$. Let $f:M\to\R$ be given by $f(x)=\varrho_xX_x$ for $x\in
M$.

For every $x\in\Int I_\alpha$ let $v_x$ be the closed curve formed
by the segment of orbit for $(\phi_t)_{t\in\R}$ starting at $x$ and
ending at $Tx$ together with the segment of $J$ that joins $Tx$ and
$x$. Then $[v_x]=[v_\alpha]$. Therefore,
$h_\alpha=\int_{v_\alpha}\varrho=\int_{v_x}\varrho$.

Since the form $\rho$ vanishes on $J$, we have
\[\int_{v_x}\varrho=\int_{0}^{\tau(x)}\varrho_{\phi_tx}X(\phi_tx)\,dt
=\int_{0}^{\tau(x)}f(\phi_tx)\,dt=\varphi_f(x).\] Consequently,
$\varphi_f(x)=h_\alpha$ for all $x\in \Int I_\alpha$ and
$\alpha\in\mathcal{A}$. If we assume that $h\in H_\pi\cap\Gamma_0$,
then
\[0=\langle\lambda,h\rangle=\int_I\varphi_f(x)dx=\int_{\mathcal{T}}f\omega.\]
\end{proof}

\subsection{Examples} \label{example7permut}\vskip 3mm

\hfill \break Let us consider an IET $T=T_{(\pi,\lambda)}$ and a
set $\{\gamma_1,\ldots,\gamma_s\}\subset
I\setminus\{l_\alpha:\alpha\in\mathcal{A}\}$, $s\geq 3$. Set
$\ell=s-1$. Suppose that
\begin{equation}\label{zalnagamma}
\{\gamma_1,\ldots,\gamma_s\}\text{ is of  periodic type with
respect to $T$ and $\theta_2(T)/\theta_1(T)<1/\ell$}.
\end{equation}
Recall that $T$ has to be of periodic type as well. An explicit
example of such data  for $s=3$ is given at the end of this section.

\vskip 3mm By Remark~\ref{istpotoku}, there exists  a multivalued
Hamiltonian flow $(\phi_t)_{t\in\R}$ with $s$ traps (determined by
saddle points $z_i$, $i=1,\ldots,s$) on a symplectic surface
$(M,\omega)$ such that $(\phi_t)_{t\in\R}$ on $\mathcal{T}$ has a
special representation over $T_{(\pi,\lambda)}$ and $\gamma_i$
corresponds to the first backward intersection with the
transversal curve of the separatrix incoming to $z_i$ for
$i=1,\ldots,s$.

By Lemma~\ref{sskokow} and (\ref{relfphi}), there exists  a
$C^{\infty}$-function $f_1:M\to\R$ such that
$\int_{\mathcal{T}}f_1\omega=0$ and $s(\varphi_{f_1})\neq 0$. In
view of Theorem~\ref{thmmain1}, the flow $(\Phi^{f_1}_t)_{t\in\R}$
on $\mathcal{T}\times\R$ is ergodic.

Let $\bar{d}_1,\ldots,\bar{d}_s$ be  vectors in $\R^{\ell-1}$ such
that $\overline{\Z(\bar{d}_1,\ldots,\bar{d}_s)}=\R^{\ell-1}$ and
$\sum_{i=1}^s\bar{d}_i=\bar{0}$. Since $s=(\ell-1)+2$, the
existence of such collection follows directly from
Remark~\ref{macierz}. By Lemma~\ref{sskokow}, there exists a
$C^{\infty}$-function $f'_2:M\to\R^{\ell-1}$ such that $f_2'$
vanishes on a neighborhood of each fixed point of $(\phi_t)$,
$\int_{\mathcal{T}}f'_2\omega=\bar{0}$ and
$(\varphi_{f'_2})_+(\gamma_i)-(\varphi_{f'_2})_-(\gamma_i)=\bar{d}_i$
for $i=1,\ldots,s$. Then $\varphi_{f'_2}$ has zero mean and, by
(\ref{relfphi}),
$s(\varphi_{f'_2})=\sum_{i=1}^s\bar{d}_i=\bar{0}$.

Denote by $\bar{\varphi}:I\to\R^{\ell}$ the piecewise constant
function with zero mean whose discontinuities are $\gamma_i$,
$i=1,\ldots,s$ and
$\bar{\varphi}_+(\gamma_i)-\bar{\varphi}_-(\gamma_i)=\bar{d}_i$ for
$i=1,\ldots,s$. In view of (\ref{zalnagamma}) and
Remark~\ref{diamond},
$\bar{\varphi}\in\bv_0^\lozenge(\sqcup_{\alpha\in \mathcal{A}}
I_{\alpha},\R^{\ell-1})$. By Remark~\ref{geom}, $\varphi_{f'_2}$ can
be extended to a $C^{\infty}$-function on the closure of any
interval of the partition
$\mathcal{P}(\{l_{\alpha}:\alpha\in\mathcal{A}\}\cup\{\gamma_i:i=1,\ldots,s\})$.
It follows that
$\varphi_{f'_2}-\bar{\varphi}\in\bv^1(\sqcup_{\alpha\in
\mathcal{A}},\R^{\ell-1})$. Moreover, $\varphi_{f'_2}-\bar{\varphi}$
has zero mean and
$s(\varphi_{f'_2}-\bar{\varphi})=s(\varphi_{f'_2})-s(\bar{\varphi})=0$.
Therefore, by Proposition~\ref{cohcon},
$\varphi_{f'_2}-\bar{\varphi}$ is cohomologous to
$\bar{h}^1=(h_1^1,\ldots,h_{\ell-1}^1)$, where $h_i^1\in\Gamma_0$
for $i=1,\ldots,\ell-1$.

In view of Theorem~\ref{thmcorrecgener} applied to the coordinate
functions of the function
$\bar{\varphi}+\bar{h}^1\in\bv_0^\lozenge(\sqcup_{\alpha\in
\mathcal{A}} I_{\alpha},\R^{\ell-1})$, there exists
$\bar{h}^2=(h_2^1,\ldots,h_{\ell-1}^2)$ with
$h_i^2\in\Gamma_u\cap\Gamma_0$ for $i=1,\ldots,\ell-1$ such that
$\bar{\varphi}+\bar{h}^1+\bar{h}^2=\widehat{\bar{\varphi}+\bar{h}^1}$.
Moreover, by Theorem~\ref{ergpiconsper}, the cocycle
$\bar{\varphi}+\bar{h}^1+\bar{h}^2=\widehat{\bar{\varphi}+\bar{h}^1}$
is ergodic. As $\varphi_{f'_2}+\bar{h}^2$ is cohomologous to
$\bar{\varphi}+\bar{h}^1+\bar{h}^2$, it is ergodic as well. By
Lemma~\ref{lematkoh}, there exists a $C^{\infty}$-function
$f''_2:M\to\R^{\ell-1}$ with
$\int_{\mathcal{T}}f''_2\omega=\bar{0}$ such that
 $\varphi_{f''_2}=\bar{h}^2$. Setting $f_2=f'_2+f''_2$, we have
$\int_{\mathcal{T}}f_2\omega=\bar{0}$,
$\varphi_{f_2}=\varphi_{f'_2}+\bar{h}^2$, and
$s(\varphi_{f_2})=s(\varphi_{f'_2})=\sum_{i=1}^s\bar{d}_i=\bar{0}$.
It follows that the flow $(\Phi^{f_2}_t)_{t\in\R}$ on
$\mathcal{T}\times\R^{\ell-1}$ is ergodic. Finally applying
Theorem~\ref{thmmain2} to $f=(f_1,f_2):I\to\R^\ell$ we have the
ergodicity of the flow $(\Phi^{f}_t)_{t\in\R}$ on
$\mathcal{T}\times\R^{\ell}$.

\begin{example}
Let us consider  the permutation
$\begin{pmatrix}1&2&3&4&5&6&7\\6&7&4&5&3&1&2
\end{pmatrix}$ and a  corresponding pair $\pi'$.  On
the Rauzy graph $\mathcal{R}(\pi')$ let us consider the loop
starting from $\pi'$ and passing through the edges labeled
consecutively by
$$1,0,1,1,1,1,1,1,0,1,1,0,1,1,1,0,0,1,1,1,1,0,1,0,0,0,0,1,1,1.$$
Then the resulting matrix is
\[A':=\begin{pmatrix}9& 8& 20& 20& 15& 5& 5\\ 1& 2& 4& 4& 3& 2& 2\\
2& 2& 6& 5& 4& 1& 1\\ 2& 2& 5& 6& 4& 1& 1\\ 1& 1& 2& 2& 2& 0& 0\\
2& 2& 4& 4& 3& 2& 1\\ 1& 1& 3& 3& 2& 1& 2
\end{pmatrix}\]
and $(A')^2$ has positive entries. Let $\lambda'\in\R^7_+$ be  a
Perron-Frobenius eigenvector of $A'$. Then $T_{(\pi',\lambda')}$
is of periodic type and $A'$ is its periodic matrix. Of course,
$T_{(\pi',\lambda')}$ is an exchange of $4$ intervals, more
precisely, $T_{(\pi',\lambda')}=T_{(\pi^{sym}_4,\lambda)}$, where
$\lambda_1=\lambda'_1+\lambda'_2$,
$\lambda_2=\lambda'_3+\lambda'_4$, $\lambda_3=\lambda'_5$ and
$\lambda_4=\lambda'_6+\lambda'_7$. As we already noticed in
Section~\ref{subsechyper}, $T_{(\pi^{sym}_4,\lambda)}$ has also
periodic type and the family $\gamma_1=\lambda'_1$,
$\gamma_2=\lambda'_1+\lambda'_2+\lambda'_3$,
$\gamma_3=\lambda'_1+\lambda'_2+\lambda'_3+\lambda'_4+\lambda'_5+\lambda'_6$
is of periodic type with respect to $T_{(\pi^{sym}_4,\lambda)}$.
Moreover,
\[A=\begin{pmatrix}10& 24& 18& 7\\ 4& 11& 8& 2\\ 1& 2& 2& 0\\ 3& 7& 5& 3
\end{pmatrix}\]
is the periodic matrix of $T_{(\pi^{sym}_4,\lambda)}$, so
\[\rho_1=\frac{13}{2}+\frac12\sqrt{115}+\frac12\sqrt{280+26\sqrt{115}},\;\;\rho_2
=\frac{13}{2}-\frac12\sqrt{115}+\frac12\sqrt{280-26\sqrt{115}}.\]
Hence $\theta_2/\theta_1\approx 0.164<1/2$, so
$T_{(\pi^{sym}_4,\lambda)}$ and $\{\gamma_1,\gamma_2,\gamma_3\}$
satisfy (\ref{zalnagamma}) with $s=3$.
\end{example}

\begin{remark}
Similar examples can be constructed by matching the set
$\{\gamma_1,\ldots,\gamma_s\}$ for a fixed IET $T=T_{\pi,\lambda}$
of periodic type. Let $p\geq 1$ be the period  of $T$ and let
$\rho>1$ be the Perron-Frobenius eigenvalue of the periodic matrix
$A$ of $T$. For every $x\in I$ let $k(x)=\inf\{k\geq 0:T^{-k}x\in
I^{(p)}\}$. Let us consider the map $S:I\to I$, $S(x)=\rho\cdot
T^{-k(x)}x$. Note that for every $\alpha\in\mathcal{A}$ the map
$S$ has at least $A_{\alpha\alpha}-2$ fixed points in the interior
of $I_\alpha$. Therefore, multiplying the period of $T$, if
necessary, for every $s\geq 1$ we can find $s$ distinct fixed
points $\gamma_1,\ldots,\gamma_s$ different from $l_{\alpha}$,
$\alpha\in\mathcal{A}$. In view of Theorem~23 in \cite{Ra}, the
set $\{\gamma_1,\ldots,\gamma_s\}$ is of periodic type with
respect to $T$.
\end{remark}
Denote by $M_2$ a compact $C^\infty$-surface of genus $2$. We can
apply the above constructions to the sequence of IETs $T$ with
arbitrary small values of the ratio $\theta_2(T)/\theta_1(T)$ from
Appendix~\ref{theta1theta2} to obtain the following result:

\begin{corollary}
For every $\ell\geq 1$ there exists a multivalued Hamiltonian flow
$(\phi_t)_{t\in\R}$ on $M_2$  and  a $C^\infty$-function
$f:M_2\to\R^\ell$ for which the flow $(\Phi^{f}_t)_{t\in\R}$ on
$\mathcal{T}\times\R^\ell$ is ergodic.\bez
\end{corollary}

\appendix

\section{Deviation of cocycles: proofs \label{proofs-dev}}

Let $T:I\to I$ be an arbitrary IET satisfying Keane's condition.
For every $x\in I$ and $n\geq 0$ set
\[m(x,n,T)=\max\{l\geq 0:\#\{0\leq k\leq n:T^kx\in I^{(l)}\}\geq 2\}.\]
\begin{proposition}[see \cite{Zo} or \cite{ViB}] \label{relmn}
For every $x\in I$ and $n>0$ we have
\[\min_{\alpha\in\mathcal{A}}Q_{\alpha}(m)\leq n\leq d\max_{\alpha\in\mathcal{A}}
Q_{\alpha}(m+1)=d\|Q(m+1)\|, \text{ where }m=m(x,n,T).\text{  \bez}\]
\end{proposition}

\begin{remark}
Assume that $T=T_{(\pi,\lambda)}$ is of periodic type and $A$ is
its periodic matrix.  Then there exists $C>0$ such that $
e^{\theta_1 k}/C\leq\|A^k\|\leq C e^{\theta_1 k}$ for every $k\geq
1$, where $\theta_1$ is the greatest Lyapunov exponent of $A$. Let
$m=m(x,n,T)$. Since $\|A^n\|=\max_{\alpha\in
\mathcal{A}}A^n_\alpha$, by Proposition~\ref{relmn} and
(\ref{maxmin1}), it follows that
\begin{equation*}
n\geq
\min_{\alpha\in\mathcal{A}}Q_{\alpha}(m)=\min_{\alpha\in\mathcal{A}}A^m_{\alpha}\geq
\frac{1}{\nu(A)}
\max_{\alpha\in\mathcal{A}}A^m_{\alpha}=\frac{\|A^m\|}{\nu(A)}\geq\frac{e^{\theta_1
m}}{C\nu(A)}.
\end{equation*}
Thus
\begin{equation}\label{mperio}
m\leq \frac{1}{\theta_1}\log(C\nu(A)n).
\end{equation}
\end{remark}

\begin{proposition}[see \cite{Ma-Mo-Yo}]\label{propsn}
For each bounded function $\varphi:I\to\R$, $x\in I$ and $n>0$ we
have
\begin{equation}\label{szacsn}
|\varphi^{(n)}(x)|\leq
2\sum_{l=0}^m\|Z(l+1)\|\|S(l)\varphi\|_{\sup}, \text{ where
}m=m(x,n,T).
\end{equation}
If additionally $\varphi\in\bv_0(\sqcup_{\alpha\in \mathcal{A}}
I_{\alpha})$ then
\begin{equation}\label{szacsl}
\|S(l)\varphi\|_{\sup}\leq\sum_{1\leq j\leq
l}\|Z(j)\|\|S(j,l)|_{\Gamma^{(j)}_0}\|\var\varphi.\text{  \bez}
\end{equation}
\end{proposition}

\begin{proof}[Proof of Theorem~\ref{thmthetas}]
Since $\lambda$ is a positive Perron-Frobenius eigenvector of
$A$, by Proposition 5 in \cite{Zo}, the restriction of $A^t$ to
the invariant space $Ann(\lambda)=\{h\in\R^{\mathcal{A}}:\langle
h,\lambda\rangle=0\}$ has the following Lyapunov exponents:
\[\theta_2\geq\theta_3\geq\ldots\geq\theta_g\geq0=\ldots
=0\geq-\theta_g\geq\ldots\geq-\theta_3\geq-\theta_2>-\theta_1.\]
Thus there exists $C>0$ such that for every $k\in\N$ we have
\[\|(A^t)^kh\|\leq Ck^{M-1}\exp(k\theta_2)\|h\|\text{ for all }h\in Ann(\lambda).\]
Since $\Gamma_0^{(j)}=Ann(\lambda)$  and
$S(j,l)=Q^t(j,l)=(A^t)^{l-j}$ on $\Gamma_0^{(j)}$, by
(\ref{szacsl}),
\begin{eqnarray*}
\|S(l)\varphi\|_{\sup}&\leq&\sum_{1\leq j\leq
l}\|A\|\|(A^t)^{l-j}|_{Ann(\lambda)}\|\var\varphi
\\&\leq&\sum_{0\leq k< l}\|A\|Ck^{M-1}\exp(k\theta_2)\var\varphi\leq\|A\|Cl^{M}\exp(l\theta_2)\var\varphi.
\end{eqnarray*}
In view of (\ref{szacsn}), it follows that
\begin{eqnarray*}
|\varphi^{(n)}(x)|&\leq&2\sum_{l=0}^m\|A\|\|S(l)\varphi\|_{\sup}\leq
2\sum_{l=0}^m\|A\|^2Cl^{M}\exp(l\theta_2)\var\varphi
\\&\leq&
2\|A\|^2Cm^{M+1}\exp(m\theta_2)\var\varphi,\end{eqnarray*} where
$m=m(x,n,T)$. Consequently, by (\ref{mperio}),
\[|\varphi^{(n)}(x)| \leq 2
\frac{\|A\|^2C^2\nu(A)}{\theta^{M+1}_1}
\log^{M+1}(C\nu(A)n)n^{\theta_2/\theta_1}\var\varphi.\]
\end{proof}

\section{Possible values of $\theta_2/\theta_1$}\label{theta1theta2}

In this section we will show that for each symmetric pair
$\pi^{sym}_4$ there are IETs of periodic type such that
$\theta_2/\theta_1$ is arbitrary small and the spectrum of the
periodic matrix is non-degenerated. As it was shown in
\cite{Ma-Mo-Yo} for every natural $n$ the matrix
\[M(n)=\begin{pmatrix}1&1&1&1\\n&n+1&0&0\\0&0&2&1\\n+1&n+2&2&2\end{pmatrix}\]
is a resulting matrix corresponding to a loop in the Rauzy class
of $\pi^{sym}_4$ and starting from $\pi^{sym}_4$. Since $M(n)$ is
primitive, there exists an IET of periodic type for which $M(n)$
is its periodic matrix. The eigenvalues
$\rho_1(n)>\rho_2(n)>1>\rho_3(n)>\rho_4(n)>0$  of $M(n)$ are of
the form
\begin{align*}&\rho_1(n)=\frac{1}{2}\left(a^+_n+\sqrt{(a^+_n)^2-4}\right),
\;\;\rho_2(n)=\frac{1}{2}\left(a^-_n+\sqrt{(a^-_n)^2-4}\right),\\
&\rho_3(n)=\frac{1}{2}\left(a^-_n-\sqrt{(a^-_n)^2-4}\right),
\;\;\rho_4(n)=\frac{1}{2}\left(a^+_n-\sqrt{(a^+_n)^2-4}\right),
\end{align*}
where
\[a_n^\pm=\frac{1}{2}(n+6\pm\sqrt{n^2+4}).\]
Since $a_n^+\to+\infty$ and $a_n^-\to 3$ as $n\to+\infty$, it
follows that
\[\frac{\theta_2(n)}{\theta_1(n)}=\frac{\log \rho_2(n)}{\log \rho_1(n)}\to 0\text{ as }n\to+\infty.\]

\section{Deviation of corrected functions}\label{korekcja}
\begin{proof}[Proof of Theorem~\ref{thmcorre}]
First note that for every natural $k$ the subspace
$\Gamma^{(k)}_{cs}\subset\R^{\mathcal{A}}$ is the direct sum of
invariant subspaces associated to Jordan blocks of $A^t$ with
non-positive Lyapunov exponents. It follows that there exists
$C>0$ such that
\begin{equation}\label{censtab}
\|(A^t)^{n}h\|\leq Cn^{M-1}\|h\|\text{ for all
}h\in\Gamma^{(k)}_{cs}\text{ and }n\geq 0.
\end{equation}
It is easy to show that $\Gamma^{(k)}_{cs}\subset \Gamma^{(k)}_0$.

Next note that  $S(k,l)\Gamma^{(k)}_{cs}=\Gamma^{(l)}_{cs}$ and
the quotient linear transformation
\[S_u(k,l):\bv(\sqcup_{\alpha\in \mathcal{A}} I^{(k)}_{\alpha})/\Gamma^{(k)}_{cs}
\to\bv(\sqcup_{\alpha\in \mathcal{A}} I^{(l)}_{\alpha})/\Gamma^{(l)}_{cs}\]
is invertible. Moreover,
\begin{equation}\label{splatanies}
S_u(k,l)\circ U^{(k)}\varphi=U^{(l)}\circ S(k,l)\varphi\text{ for
}\varphi\in\bv(\sqcup_{\alpha\in \mathcal{A}} I^{(k)}_{\alpha}).
\end{equation}

Since $\Gamma^{(k)}_{u}\subset\R^{\mathcal{A}}$  the direct sum of
invariant subspaces associated to Jordan blocks of $A^t$ with
positive Lyapunov exponents,
$\R^{\mathcal{A}}=\Gamma^{(k)}=\Gamma^{(k)}_{cs}\oplus\Gamma^{(k)}_{u}$
is an invariant decomposition. Moreover, there exist $\theta_+>0$
and $C> 0$ such that
\[\|(A^t)^{-n}h\|\leq C\exp(-n\theta_+)\|h\|\text{ for all }h\in\Gamma^{(k)}_{u}\text{ and }n\geq 0.\]
Since the linear operators
$A^t:\Gamma^{(k)}_{u}\to\Gamma^{(k)}_{u}$ and
$A^t:\Gamma^{(k)}/\Gamma^{(k)}_{cs}\to\Gamma^{(k)}/\Gamma^{(k)}_{cs}$
are isomorphic, there exists $C'>0$ such that
\[\|(A^t)^{-n}(h+\Gamma^{(k)}_{cs})\|\leq C'\exp(-n\theta_+)\|h+\Gamma^{(k)}_{cs}\|\]
for all $h+\Gamma^{(k)}_{cs}\in\Gamma^{(k)}/\Gamma^{(k)}_{cs}$ and
$n\geq 0$. Consequently,
\begin{equation}\label{szacowanieniest}
\|(S_u(k,l))^{-1}(h+\Gamma^{(k)}_{cs})\|\leq
C'\exp(-(l-k)\theta_+)\|h+\Gamma^{(k)}_{cs}\|
\end{equation}
for all $h+\Gamma^{(k)}_{cs}\in\Gamma^{(k)}/\Gamma^{(k)}_{cs}$ and
$0\leq k< l$.

 Let us consider the linear operator
$C^{(k)}:\bv_0(\sqcup_{\alpha\in \mathcal{A}}
I^{(k)}_{\alpha})\to\Gamma^{(k)}_0$ given by
\[C^{(k)}\varphi(x)=\frac{1}{|I^{(k)}_{\alpha}|}\int_{I^{(k)}_{\alpha}}
\varphi(t)dt\text{ if }x\in I^{(k)}_{\alpha}.\]
Then $P^{(k)}_0\varphi=\varphi-C^{(k)}\varphi$ and
\begin{equation}\label{nace}
\|C^{(k)}\varphi\|\leq\|\varphi\|_{\sup},
\end{equation}
\begin{equation}\label{nape}
\|P^{(k)}_0\varphi\|_{\sup}\leq \var P^{(k)}_0\varphi=\var
\varphi.
\end{equation}
Let $\varphi\in\bv_0(\sqcup_{\alpha\in \mathcal{A}}I_\alpha)$.
Note that for $0\leq k\leq l$ we have
\begin{eqnarray*}
\lefteqn{ P_0^{(k)}\varphi-S(k,l)^{-1}\circ P_0^{(l)}\circ
S(k,l)\varphi}\\&=&\sum_{k\leq r<l}(S(k,r)^{-1}\circ P_0^{(r)} \circ
S(k,r)-S(k,r+1)^{-1}\circ P_0^{(r+1)}\circ S(k,r+1))\varphi
\\&=&\sum_{k\leq r<l}S(k,r+1)^{-1}\circ(S(r,r+1)
\circ  P_0^{(r)}- P_0^{(r+1)}\circ S(r,r+1))\circ S(k,r)\varphi.
\end{eqnarray*}
Next observe that
\[(S(r,r+1)\circ P_0^{(r)}-P_0^{(r+1)}\circ S(r,r+1))\psi=C^{(r+1)}
\circ S(r,r+1)\circ P_0^{(r)}\psi\in\Gamma^{r+1}_0\]
for $\psi\in\bv_0(\sqcup_{\alpha\in \mathcal{A}}
I^{(r)}_{\alpha})$. Indeed, if $\psi\in\bv_0(\sqcup_{\alpha\in
\mathcal{A}} I^{(r)}_{\alpha})$ then
$\psi=P_0^{(r)}\psi+C^{(r)}\psi$ and
\[P_0^{(r+1)}\circ S(r,r+1)\psi=P_0^{(r+1)}\circ S(r,r+1)
\circ P_0^{(r)}\psi+P_0^{(r+1)}\circ S(r,r+1)\circ C^{(r)}\psi.\]
Since $S(r,r+1)\circ C^{(r)}\psi\in\Gamma^{(r+1)}_0$, we obtain
$P_0^{(r+1)}\circ S(r,r+1)\circ C^{(r)}\psi=0$; hence
\begin{eqnarray*}
\lefteqn{S(r,r+1)\circ P_0^{(r)}\psi-P_0^{(r+1)}\circ S(r,r+1)\psi}\\
&=&S(r,r+1)\circ P_0^{(r)}\psi-P_0^{(r+1)}\circ S(r,r+1)\circ P_0^{(r)}\psi\\
&=&C^{(r+1)}\circ S(r,r+1)\circ P_0^{(r)}\psi.
\end{eqnarray*}
Therefore
\begin{eqnarray*}
\lefteqn{ P_0^{(k)}\varphi-S(k,l)^{-1}\circ P_0^{(r)}\circ
 S(k,l)\varphi}\\&=&\sum_{k\leq r<l}S(k,r+1)^{-1}\circ C^{(r+1)}\circ S(r,r+1)\circ P_0^{(r)}\circ
 S(k,r)\varphi\in\Gamma^{(k)}_0.
\end{eqnarray*}
In view of (\ref{splatanies}),
\begin{eqnarray*}
\lefteqn{ (U^{(k)}\circ P_0^{(k)}-U^{(k)}\circ S(k,l)^{-1}\circ
P_0^{(r)}\circ
 S(k,l))\varphi}\\&=&\sum_{k\leq r<l}S_u(k,r+1)^{-1} \circ U^{(r+1)}
 \circ C^{(r+1)}\circ S(r,r+1)\circ P_0^{(r)}\circ
 S(k,r)\varphi.
\end{eqnarray*}
Moreover, using (\ref{nace}), (\ref{nase}), (\ref{nape}) and
(\ref{nave}) successively we obtain
\begin{eqnarray*}
\lefteqn{\|C^{(r+1)}\circ S(r,r+1)\circ P_0^{(r)}\circ S(k,r)\varphi\|
\leq \|S(r,r+1)\circ P_0^{(r)}\circ S(k,r)\varphi\|_{\sup}}\\
&\leq& \|Z(r+1)\|\|P_0^{(r)}\circ S(k,r)\varphi\|_{\sup}\leq
\|A\|\var S(k,r)\varphi\leq \|A\|\var \varphi.
\end{eqnarray*}
Next let consider the series in $\Gamma^{(k)}_0/\Gamma^{(k)}_{cs}$
\begin{equation}\label{defoperdel}
\sum_{r\geq k}(S_u(k,r+1))^{-1}\circ U^{(r+1)}\circ C^{(r+1)}\circ
S(r,r+1)\circ P_0^{(r)}\circ S(k,r)\varphi.
\end{equation}
Since $\|U^{(r+1)}\|=1$ and $U^{(r+1)}\circ C^{(r+1)}\circ
S(r,r+1)\circ P_0^{(r)}\circ S(k,r)\varphi\in
\Gamma_0^{(r+1)}/\Gamma^{(r+1)}_{cs}$, by (\ref{szacowanieniest}),
the norm of the $r$-th element of the series  (\ref{defoperdel})
is bounded from above by
$C'\|A\|\exp(-(r-k+1)\theta_+)\var\varphi$. As
\[\sum_{r\geq k}C'\|A\|\exp(-(r-k+1)\theta_+)\var\varphi<+\infty,\]
the series (\ref{defoperdel})  converges in
$\Gamma_0^{(k)}/\Gamma^{(k)}_{cs}$. Denote by $\Delta
P^{(k)}\varphi\in\Gamma_0^{(k)}/\Gamma^{(k)}_{cs}$ the sum of
(\ref{defoperdel}). Then there exists $K>0$ such that
\begin{equation}\label{szacdelty}
\|\Delta P^{(k)}\varphi\|\leq K\var\varphi, \text{ for every
}\varphi\in\bv_0(\sqcup_{\alpha\in \mathcal{A}}
I^{(k)}_{\alpha})\text{ and }k\geq 0.
\end{equation}
It follows that the sequence (\ref{ciagdop}) converges in
$\bv_0(\sqcup_{\alpha\in \mathcal{A}}
I^{(k)}_{\alpha})/\Gamma^{(k)}_{cs}$ and
\begin{equation}\label{wzorp}
P^{(k)}=U^{(k)}\circ P_0^{(k)}-\Delta P^{(k)}.
\end{equation}
\end{proof}
\begin{lemma}
For all $0\leq k\leq l$ and $\varphi\in\bv_0(\sqcup_{\alpha\in
\mathcal{A}} I^{(k)}_{\alpha})$ we have
\begin{eqnarray}
S_u(k,l)\circ P^{(k)}\varphi&=&P^{(l)}\circ
S(k,l)\varphi, \label{przem} \\
\|P^{(k)}\varphi\|_{\sup/\Gamma^{(k)}_{cs}} &\leq& (1+K)\var\varphi.
\label{szapk}
\end{eqnarray}
\end{lemma}

\begin{proof} By definition and by (\ref{splatanies}),
\begin{eqnarray*}\lefteqn{S_u(k,l)\circ
P^{(k)}\varphi=S_u(k,l)\lim_{r\to\infty}U^{(k)}\circ
S(k,r)^{-1}\circ P_0^{(r)}\circ S(k,r)\varphi}\\
&=&\lim_{r\to\infty}U^{(l)}\circ S(l,r)\circ S(k,r)^{-1}\circ
P_0^{(r)}\circ
S(k,r)\varphi\\
&=&\lim_{r\to\infty}U^{(l)}\circ S(l,r)^{-1}\circ P_0^{(r)}\circ
S(l,r)\circ S(k,l)\varphi=P^{(l)}\circ S(k,l)\varphi.
\end{eqnarray*}
Moreover, by (\ref{wzorp}), (\ref{nape}) and (\ref{szacdelty}),
\[\|P^{(k)}\varphi\|_{\sup/\Gamma^{(k)}_{cs}}\leq\|P_0^{(k)}
\varphi\|_{\sup}+\|\Delta P^{(k)}\varphi\|\leq (1+K)\var\varphi.\]
\end{proof}

Let $p:\{0,1,\ldots,d,d+1\}\to\{0,1,\ldots,d,d+1\}$ stand for  the
permutation
\[p(j)=\left\{
\begin{matrix}
\pi_1\circ\pi^{-1}_0(j)&\text{ if }&1\leq j\leq d\\
j&\text{ if }&j=0,d+1.
\end{matrix}
\right.
\]
Following \cite{Ve1,Ve2}, denote by $\sigma=\sigma_\pi$ the
corresponding permutation on $\{0,1,\ldots,d\}$,
\[\sigma(j)=p^{-1}(p(j)+1)-1\text{ for }0\leq j\leq d.\]
Then
$\widehat{T}_{(\pi,\lambda)}r_{\pi_0^{-1}(j)}={T}_{(\pi,\lambda)}r_{\pi_0^{-1}(\sigma
j)}$ for all $j\neq 0,p^{-1}(d)$. Denote by $\Sigma(\pi)$ the set
of orbits for the permutation $\sigma$. Let $\Sigma_0(\pi)$ stand
for the subset of orbits that do not contain zero. Then
$\Sigma(\pi)$ corresponds to the set of singular points of any
translation surface associated to $\pi$ and hence
$\#\Sigma(\pi)=\kappa(\pi)$. For every $\mathcal{O}\in\Sigma(\pi)$
denote by $b(\mathcal{O})\in\R^{\mathcal{A}}$ the vector given by\
\[b(\mathcal{O})_{\alpha}=\chi_{\mathcal{O}}(\pi_0(\alpha))-\chi_{\mathcal{O}}(\pi_0(\alpha)-1)
\text{ for }\alpha\in\mathcal{A}.\]
\begin{lemma}[see \cite{Ve2}]\label{bcharh}
For every irreducible pair $\pi$ we have
$\sum_{\mathcal{O}\in\Sigma(\pi)}b(\mathcal{O})=0$, the vectors
$b(\mathcal{O})$, $\mathcal{O}\in\Sigma_0(\pi)$ are linearly
independent and the linear subspace generated by them is equal to
$\ker\Omega_\pi$. Moreover, $h\in H_{\pi}$ if and only if $\langle
h,b(\mathcal{O)}\rangle=0$ for every $\mathcal{O}\in\Sigma(\pi)$.
  \bez
\end{lemma}

\begin{remark}\label{remiso} Let $\Lambda^\pi:\R^{\mathcal{A}}\to\R^{\Sigma_0(\pi)}$ stand for the
linear transformation given by $(\Lambda^\pi
h)_\mathcal{O}=\langle h,b(\mathcal{O)}\rangle$ for
$\mathcal{O}\in\Sigma_0(\pi)$. By Lemma~\ref{bcharh},
$H_{\pi}=\ker\Lambda^\pi$ and if $\R^{\mathcal{A}}=F\oplus H_\pi$
is a direct sum decomposition then
$\Lambda^\pi:F\to\R^{\Sigma_0(\pi)}$ establishes an isomorphism of
linear spaces. It follows that there exists $K_F>0$ such that
\[\|h\|\leq K_F\|\Lambda^\pi h\| \;\;\text{ for all }\;\;h\in F.\]
\end{remark}

\begin{lemma}[see \cite{Ve2}]\label{invario}
Suppose that
$T_{(\tilde{\pi},\tilde{\lambda})}=\mathcal{R}(T_{(\pi,\lambda)})$.
Then there exists a bijection
$\xi:\Sigma(\pi)\to\Sigma(\tilde{\pi})$ such that
$\Theta(\pi,\lambda)^{-1}b(\mathcal{O})=b(\xi\mathcal{O})$ for
$\mathcal{O}\in\Sigma(\pi)$.\bez
\end{lemma}

Let $T=T_{(\pi,\lambda)}$ be an IET satisfying Keane's condition.
For every $\mathcal{O}\in\Sigma(\pi)$ and
$\varphi\in\bv^{\lozenge}(\sqcup_{\alpha\in \mathcal{A}}
I_{\alpha})$ let
\[\mathcal{O}(\varphi)=\sum_{\alpha\in\mathcal{A},\pi_0(\alpha)\in\mathcal{O}}\varphi_-(r_{\alpha})-
\sum_{\alpha\in\mathcal{A},\pi_0(\alpha)-1\in\mathcal{O}}\varphi_+(l_{\alpha}).\]
Note that if $h\in \Gamma^{(0)}$ (i.e. $h$ is a function constant on
exchanged intervals), then
\[\mathcal{O}(h)=\sum_{\pi_0(\alpha)\in\mathcal{O}}h_{\alpha}-
\sum_{\pi_0(\alpha)-1\in\mathcal{O}}h_{\alpha}=\sum_{\alpha\in\mathcal{A}}
(\chi_{\mathcal{O}}(\pi_0(\alpha))-\chi_{\mathcal{O}}(\pi_0(\alpha)-1))h_\alpha=\langle
h,b(\mathcal{O})\rangle.\] Moreover,
\begin{equation}\label{szao}
|\mathcal{O}(\varphi)|\leq 2d\|\varphi\|_{\sup}\text{ for every
}\varphi\in\bv^{\lozenge}(\sqcup_{\alpha\in \mathcal{A}}
I_{\alpha})\text{ and }\mathcal{O}\in\Sigma(\pi).
\end{equation}

Let us consider
$T_{(\tilde{\pi},\tilde{\lambda})}=\mathcal{R}(T_{(\pi,\lambda)})$
and  the renormalized cocycle $\tilde{\varphi}:\tilde{I}\to\R$,
this is
\[\tilde{\varphi}(x)=\sum_{0\leq
i<\Theta_{\beta}(\pi,\lambda)}\varphi(T_{(\pi,\lambda)}^ix)\text{
for }x\in \tilde{I}_\beta.\] The proof of the following lemma is
straightforward and we leave it to the reader.
\begin{lemma}\label{invariophi}
If $\varphi\in\bv^{\lozenge}(\sqcup_{\alpha\in \mathcal{A}}
I_{\alpha})$ then
$\tilde{\varphi}\in\bv^{\lozenge}(\sqcup_{\alpha\in \mathcal{A}}
\tilde{I}_{\alpha})$ and
$(\xi\mathcal{O})(\tilde{\varphi})=\mathcal{O}(\varphi)$ for each
$\mathcal{O}\in\Sigma(\pi)$. \bez
\end{lemma}

Let $T=T_{(\pi,\lambda)}$ be an IET of periodic type and let $A$
be its periodic matrix.  By Lemma \ref{invario}, there exists a
bijection $\xi:\Sigma(\pi)\to\Sigma(\pi)$ such that
$A^{-1}b(\mathcal{O})=b(\xi\mathcal{O})$ for
$\mathcal{O}\in\Sigma(\pi)$. Since $\xi^N=Id_{\Sigma(\pi)}$ for
some $N\geq 1$, multiplying the period of $T$ by $N$, we can
assume that $\xi=Id_{\Sigma(\pi)}$. Therefore
$Ab(\mathcal{O})=b(\mathcal{O})$ for each
$\mathcal{O}\in\Sigma(\pi)$,  and hence
$A|_{\ker\Omega_{\pi}}=Id$. It follows that  the dimension of
$\Gamma^{(0)}_c=\{h\in\R^{\mathcal{A}}:A^th=h\}$ is greater or
equal than $\kappa-1$. Denote by
$\Gamma^{(0)}_{s}\subset\R^{\mathcal{A}}$ the direct sum of
invariant subspaces associated to Jordan blocks of $A^t$ with
negative Lyapunov exponents.

Assume that $T$ has non-degenerated spectrum, i.e.\ $\theta_g>0$.
Then $\dim\Gamma^{(0)}_s=\dim\Gamma^{(0)}_u=g$. Since
$2g+\kappa-1=d$ and $\dim\Gamma^{(0)}_c=\kappa-1$,
\[\R^{\mathcal{A}}=\Gamma^{(0)}=\Gamma^{(0)}_s\oplus\Gamma^{(0)}_c\oplus\Gamma^{(0)}_u\]
is an $A^t$--invariant decompositions. It follows that
$\Gamma^{(0)}_s\oplus\Gamma^{(0)}_c=\Gamma^{(0)}_{cs}\subset\Gamma^{(0)}_0$.
Therefore
\[\Gamma^{(0)}_0=\Gamma^{(0)}_s\oplus\Gamma^{(0)}_c\oplus(\Gamma^{(0)}_u\cap\Gamma^{(0)}_0).\]
Recall that $\Gamma^{(0)}_s\oplus\Gamma^{(0)}_u\subset H_{\pi}$.
As $T$ has non-degenerated spectrum,  these subspaces have the
same dimension, and so they are equal. Denote by $\Gamma^{(k)}_s$,
$\Gamma^{(k)}_c$ and  $\Gamma^{(k)}_u$ the subspaces of functions
on $I^{(k)}$  constant on intervals $I^{(k)}_\alpha$,
$\alpha\in\mathcal{A}$ corresponding to the vectors from
$\Gamma^{(0)}_s$, $\Gamma^{(0)}_c$ and  $\Gamma^{(0)}_u$
respectively. Then
\begin{equation}\label{trzydeco}
\Gamma^{(k)}=\Gamma^{(k)}_s\oplus\Gamma^{(k)}_c\oplus\Gamma^{(k)}_u,\;\;
H_{\pi}=\Gamma^{(k)}_s\oplus\Gamma^{(k)}_u,
\;\;\Gamma^{(k)}_0=\Gamma^{(k)}_s\oplus\Gamma^{(k)}_c\oplus(\Gamma^{(k)}_u\cap\Gamma^{(k)}_0)
\end{equation}
for $k\geq 0$ is a family of decomposition invariant with respect
to the renormalization operators $S(k,l)$ for $0\leq k<l$.

As $\xi=Id_{\Sigma(\pi)}$, by Lemma~\ref{invariophi}, for every
$\varphi\in\bv^{\lozenge}(\sqcup_{\alpha\in \mathcal{A}}
I^{(k)}_{\alpha})$ and $l\geq k$ we have
\begin{equation}\label{niezmo}
S(k,l)\varphi\in\bv^{\lozenge}(\sqcup_{\alpha\in \mathcal{A}}
I^{(l)}_{\alpha})\text{ and }
\mathcal{O}(S(k,l)\varphi)=\mathcal{O}(\varphi)\text{ for each
}\mathcal{O}\in\Sigma(\pi).
\end{equation}

\begin{proof}[Proof of Theorem~\ref{thmcorrecgener}]
Since
\[U^{(0)}\widehat{\varphi}= P^{(0)}\varphi=U^{(0)}\circ P_0^{(0)}\varphi-\Delta P^{(0)}
\varphi=U^{(0)}\varphi-U^{(0)}\circ C^{(0)}\varphi-\Delta P^{(0)}\varphi,\]
we have
\[\varphi-\widehat{\varphi}\in U^{(0)}\circ C^{(0)}\varphi+\Delta P^{(0)}\varphi \subset \Gamma_0^{(0)}.\]
In view of (\ref{splatanies}) and (\ref{przem}),
\[U^{(k)}\circ
S(k)\widehat{\varphi}=S_u(k)\circ U^{(0)}\widehat{\varphi}
=S_u(k)\circ P^{(0)}\varphi=P^{(k)}\circ S(k)\varphi.\] Therefore,
by (\ref{szapk}) and  (\ref{nave}), we have
\begin{eqnarray*}
\|U^{(k)}\circ
S(k)\widehat{\varphi}\|_{\sup/\Gamma^{(k)}_{cs}}&=&\|P^{(k)}(S(k)\varphi)\|_{\sup/\Gamma^{(k)}_{cs}}\\
&\leq&(1+K)\var(S(k)\varphi)\leq(1+K)\var\varphi.
\end{eqnarray*}
It follows that for every $k\geq 0$ there exists
$\varphi_k\in\bv_0(\sqcup_{\alpha\in \mathcal{A}}
I^{(k)}_{\alpha})$ and $h_k\in\Gamma^{(k)}_{cs}$ such that
\begin{equation}\label{szcsk1}
S(k)\widehat{\varphi}=\varphi_k+h_k\text{ and
}\|\varphi_k\|_{\sup}\leq(1+K)\var\varphi.
\end{equation}
As
\begin{equation}\label{roznren}
\varphi_{k+1}+h_{k+1}=S(k+1)\widehat{\varphi}=S(k,k+1)S(k)\widehat{\varphi}=S(k,k+1)\varphi_k+A^th_k,
\end{equation}
setting $\Delta h_{k+1}=h_{k+1}-A^th_k$ ($\Delta h_0=h_0$) we have
$\Delta h_{k+1}=-\varphi_{k+1}+S(k,k+1)\varphi_k$. Moreover, by
(\ref{szcsk1}),
\begin{eqnarray*}
\|\Delta
h_{k+1}\|&=&\|\varphi_{k+1}-S(k,k+1)\varphi_k\|_{\sup}\\&\leq&\|\varphi_{k+1}\|_{\sup}+\|S(k,k+1)\varphi_k\|_{\sup}\leq
(1+\|A\|)(1+K)\var\varphi.
\end{eqnarray*}
and
\[\|\Delta h_{0}\|=\|\widehat{\varphi}-\varphi_0\|_{\sup}\leq \|\widehat{\varphi}\|_{\sup}+(1+K)\var\varphi.\]
Since $h_k=\sum_{0\leq l\leq k}(A^t)^l\Delta h_{k-l}$ and $\Delta
h_l\in\Gamma^{(l)}_{cs}$, by (\ref{censtab}),
\begin{eqnarray*}
\|h_k\|&\leq&\sum_{0\leq l\leq k}\|(A^t)^l\Delta h_{k-l}\|\leq\sum_{0\leq l\leq k}Cl^{M-1}\|\Delta h_{k-l}\|\\
&\leq&
Ck^M(1+\|A\|)(1+K)\var\varphi+Ck^{M-1}\|\widehat{\varphi}\|_{\sup}.
\end{eqnarray*}
In view of (\ref{szcsk1}), it follows that
\[\|S(k)\widehat{\varphi}\|_{\sup}\leq\|{\varphi}_k\|_{\sup}+\|h_k\|
\leq Ck^M(2+\|A\|)(1+K)\var\varphi+Ck^{M-1}\|\widehat{\varphi}\|_{\sup}.\]

\vskip 3mm Since $\widehat{\varphi}-\varphi\in
\Gamma^{(0)}_0=(\Gamma^{(0)}_u\cap\Gamma^{(0)}_0)\oplus\Gamma^{(0)}_{cs}$,
there exist $h\in(\Gamma^{(0)}_u\cap\Gamma^{(0)}_0)$ and
$h'\in\Gamma^{(0)}_{cs}$ such that
$\varphi+h=\widehat{\varphi}+h'$. Hence
\[\varphi+h+\Gamma^{(0)}_{cs}=\widehat{\varphi}+\Gamma^{(0)}_{cs}=
P^{(0)}\varphi.\] Suppose that
$h_1,h_2\in\Gamma^{(0)}_u\cap\Gamma^{(0)}_0$ are vectors such that
\[\varphi+h_1+\Gamma^{(0)}_{cs}=\varphi+h_2+\Gamma^{(0)}_{cs}=
P^{(0)}\varphi.\] In view of (\ref{polygr}),
$\|S(k)(\varphi+h_1)\|_{\sup}$ and $\|S(k)(\varphi+h_2)\|_{\sup}$
have at most polynomial growth. Therefore,
$\|(A^t)^k(h_1-h_2)\|=\|S(k)(h_1-h_2)\|$ has at most polynomial
growth, as well. Since $h_1-h_2\in\Gamma_u^{(0)}$, it follows that
$h_1=h_2$.

\vskip 3mm Assume that  $T$ has non-degenerated spectrum. Then
$\Gamma^{(k)}_{cs}=\Gamma^{(k)}_{c}\oplus\Gamma^{(k)}_{s}$. Suppose
$\varphi_k$, $h_k\in\Gamma^{(k)}_{cs}$ satisfy (\ref{szcsk1}). Let
us decompose $h_k=h^s_k+h^c_k$, where $h^c_k\in\Gamma^{(k)}_{c}$ and
$h^s_k\in \Gamma^{(k)}_{s}\subset H_\pi$. By Remark~\ref{remiso},
$\Lambda^\pi(h^s_k)=0$. In view of (\ref{szcsk1}) and
(\ref{niezmo}), it follows that
\[\mathcal{O}(\widehat{\varphi})=\mathcal{O}(S(k)\widehat{\varphi})=\mathcal{O}(\varphi_k)+\mathcal{O}(h^c_k)
\text{ for every }\mathcal{O}\in\Sigma(\pi).\] Moreover, by
(\ref{szao}) and (\ref{szcsk1}),
\[|\mathcal{O}(\varphi_k)|\leq 2d\|\varphi_k\|_{\sup}\leq 2d(1+K)\var\varphi\text{ and }
|\mathcal{O}(\widehat{\varphi})|\leq
2d\|\widehat{\varphi}\|_{\sup}\] for every
$\mathcal{O}\in\Sigma(\pi)$. Therefore \[|\langle
h^c_k,b(\mathcal{O})\rangle|=|\mathcal{O}(h^c_k)|\leq
2d((1+K)\var\varphi+\|\widehat{\varphi}\|_{\sup})\text{ for every
}\mathcal{O}\in\Sigma(\pi),\] so that
\begin{equation}
\label{normlambda} \|\Lambda^\pi(h^c_k)\|\leq
2d((1+K)\var\varphi+\|\widehat{\varphi}\|_{\sup}).
\end{equation}
By (\ref{trzydeco}), we have
$\R^{\mathcal{A}}=\Gamma^{(k)}=\Gamma^{(k)}_{c}\oplus H_\pi$, so
in view of Remark~\ref{remiso}, there exists $K'\geq 1$ such that
$\|h\|\leq K'\|\Lambda^\pi h\|$ for every $h\in\Gamma^{(k)}_{c}$.
By (\ref{normlambda}), it follows that
\begin{equation}
\label{estcentral}
\|h^c_k\|\leq 2dK'((1+K)\var\varphi+\|\widehat{\varphi}\|_{\sup}).
\end{equation}
Let $\Delta h^s_{k+1}=h^s_{k+1}-A^th^s_k$  for $k\geq 0$ and
$\Delta h^s_0=h^s_0$. Then from (\ref{roznren}), we have \[\Delta
h^s_{k+1}=-\varphi_{k+1}+S(k,k+1)\varphi_k-h^c_{k+1}+A^th^c_k=-\varphi_{k+1}+S(k,k+1)\varphi_k-h^c_{k+1}+h^c_k.\]
Therefore, by (\ref{szcsk1}) and (\ref{estcentral}),
\begin{eqnarray*}
\|\Delta
h^s_{k+1}\|&\leq&\|\varphi_{k+1}\|_{\sup}+\|A\|\|\varphi_k\|_{\sup}+\|h^c_{k+1}\|+\|h^c_k\|\\&\leq&
(1+\|A\|+4dK')(1+K)\var\varphi+4dK'\|\widehat{\varphi}\|_{\sup}, \\
\|\Delta h^s_{0}\|&=& \|\widehat{\varphi}-
\varphi_0-h^c_0\|_{\sup}\leq (1+2dK')(\|\widehat{\varphi}\|_{\sup}
+(1+K)\var\varphi).
\end{eqnarray*}
Notice that for every $0<\theta_-<\theta_g$ there exists  $C\geq 1$
such that
\begin{equation*}
\|(A^t)^{n}h\|\leq C\exp(-n\theta_-)\|h\|\text{ for all
}h\in\Gamma^{(k)}_{s}\text{ and }n\geq 0.
\end{equation*}
Since $h^s_k=\sum_{0\leq l\leq k}(A^t)^l\Delta h^s_{k-l}$ and
$\Delta h^s_l\in\Gamma^{(l)}_{s}$, it follows that
\begin{eqnarray*}
\|h^s_k\|&\leq&\sum_{0\leq l\leq k}\|(A^t)^l\Delta h^s_{k-l}\|\leq\sum_{0\leq l\leq k}C\exp(-l\theta_-)\|\Delta h^s_{k-l}\|\\
&\leq&
\frac{C(1+\|A\|+4dK')}{1-\exp(-\theta_-)}((1+K)\var\varphi+\|\widehat{\varphi}\|_{\sup}).
\end{eqnarray*}
In view of (\ref{szcsk1}) and (\ref{estcentral}), it follows that
\begin{eqnarray*}
\|S(k)\widehat{\varphi}\|_{\sup}&\leq&\|{\varphi}_k\|_{\sup}+\|h^c_k\|+\|h^s_k\|\\&\leq&
\frac{C(2+\|A\|+6dK')}{1-\exp(-\theta_-)}((1+K)\var\varphi+\|\widehat{\varphi}\|_{\sup}),
\end{eqnarray*}
which completes the proof.
\end{proof}

\begin{theorem}
There exist $C_3,C_4>0$ such that
\[\|\widehat{\varphi}^{(n)}\|_{\sup}\leq C_3\log^{M+1}n\var\varphi+C_4\log^{M}n\|\widehat{\varphi}\|_{\sup}\]
for every natural $n$. If additionally $T$ has non-degenerated
spectrum then
\[\|\widehat{\varphi}^{(n)}\|_{\sup}\leq C_3\log n\var\varphi+C_4\log n\|\widehat{\varphi}\|_{\sup}.\]
\end{theorem}

\begin{proof}
By Proposition~\ref{propsn} and Theorem~\ref{thmcorrecgener}, for
every $x\in I$ we have
\begin{eqnarray*}
\|\widehat{\varphi}^{(n)}(x)\|&\leq& 2\|A\|\sum_{k=0}^m(C_1k^{M}\var\varphi+C_2k^{M-1}\|\widehat{\varphi}\|_{\sup})\\
&\leq&
2\|A\|(C_1m^{M+1}\var\varphi+C_2m^{M}\|\widehat{\varphi}\|_{\sup}),
\end{eqnarray*}
where $m=m(x,n,T)$. Now the assertion follows directly from
(\ref{mperio}).
\end{proof}

\section{Example of non-regular step cocycle} Let
$T=T_{(\pi,\lambda)}$ be an IET of periodic type with periodic
matrix is $A$. Then there exists $C>0$ and $\theta>0$ such that
\begin{equation*}\label{hypcenstab1}
\|(A^t)^{n}h\|\leq C\exp(-n\theta)\|h\|\text{ for all
}h\in\Gamma^{(0)}_{s}\text{ and }n\geq 0.
\end{equation*}

\begin{lemma}\label{kobniekob}
Suppose that $h\in\Gamma^{(0)}_0$ and $\varphi:I\to\R$ is the
associated step cocycle. If $h\in \Gamma^{(0)}_{s}$ then $\varphi$
is a coboundary. If $h\notin \Gamma^{(0)}_{cs}$ then $\varphi$ is
not a coboundary.
\end{lemma}

\begin{proof}
Assume that $h\in \Gamma^{(0)}_{s}$. Since
\[\|S(l)\varphi\|_{\sup}=\|(A^t)^lh\|\leq C\exp(-l\theta)\|h\|,\]
 by Proposition~\ref{propsn}, we have
\[\|\varphi^{(n)}\|_{\sup}\leq 2\sum_{l=0}^\infty\|Z(l+1)\|\|S(l)\varphi\|_{\sup}
\leq 2C\|A\|\|h\|\sum_{l=0}^\infty\exp(-l\theta) =
\frac{2C\|A\|\|h\|}{1-\exp(-\theta)}\] for every natural $n$. But
each bounded cocycle in $\R^\ell$ is a coboundary.

\vskip 3mm Now suppose that $h\in\Gamma^{(0)}_0$ and $\varphi$ is a
coboundary. Set
\[\vep=\inf\{\mu(C^{(n)}_\alpha):n\geq 0,\alpha\in\mathcal{A}\}\]
(see Section~\ref{step} for the definition of the tower
$C^{(n)}_\alpha$). In view of (\ref{dolcn}), $\vep>0$. Since
$\varphi$ is a coboundary, there exist $M>0$ and a sequence
$(B_k)_{k\geq 0}$ of measurable sets with $\mu(B_k)>1-\vep$ for $n
\geq 0$ such that $|\varphi^{(k)}(x)|\leq M$ for all $x\in B_k$ and
$k \geq 0$. Recall that for every $x\in C^{(n)}_\alpha$ we have
$\varphi^{(h^{(n+1)}_\alpha)}(x)
 =((A^t)^{n+1}h)_{\alpha}$. Since $C^{(n)}_\alpha\cap
B_{h^{(n+1)}_\alpha}\neq\emptyset$, it follows that
$|((A^t)^{n+1}h)_{\alpha}|\leq M$ for every $n\geq 0$ and
$\alpha\in\mathcal{A}$. Thus $\|(A^t)^{n+1}h\|\leq M$ for every
$n\geq 0$, and hence $h\in \Gamma^{(0)}_{cs}$.
\end{proof}

\begin{example}
Let us consider an IET $T=T_{(\pi^{sym}_5,\lambda)}$ of periodic
type whose periodic matrix is equal to
\[A=\left(
\begin{array}{ccccc}
18 & 28 & 31 & 38  & 18\\
10 & 16 & 8 & 9 & 6\\
13 & 20 & 36 & 46 & 18\\
2 & 3 & 16 & 22 & 6 \\
39 & 61 & 63 & 77 & 37
\end{array}
\right).\] The existence of such IET was shown in \cite{Si-Ul}. The
Perron-Frobenius eigenvalue of $A$ is $55+12\sqrt{21}$ and $\lambda$
is equal to
\[(1+\sqrt{21},2,1+\sqrt{21},2,7+\sqrt{21})\]
up to multiplication by a positive constant. Moreover, the
eigenvalues and eigenvectors of $A^t$ are as follows:
\[\begin{array}{ll}
\rho_1= 55+12\sqrt{21}, & v_1=(-1+\sqrt{21}, 1+\sqrt{21},
3+\sqrt{21}, 5+\sqrt{21}, 4)\\
\rho_2=9+4\sqrt{5}, & v_2=(-2, -1-1\sqrt{5}, 2, 1+\sqrt{5},0)\\
\rho_3=1, & v_3=(-1, -2, 0, -1, 1)\\
\rho_4=9-4\sqrt{5}, & v_4=(-2, -1+1\sqrt{5}, 2, 1-\sqrt{5},0)\\
\rho_5= 55-12\sqrt{21}, & v_5=(-1-\sqrt{21}, 1-\sqrt{21},
3-\sqrt{21}, 5-\sqrt{21}, 4).
\end{array}
\]
Note that $v_2,v_3,v_4,v_5\in \Gamma^{(0)}_0$. Denote by
$\varphi_i:I\to\R$ the step function corresponding to $v_i$ for
$1< i\leq 5$. Since $|\rho_2|>1>|\rho_4|$, by
Lemma~\ref{kobniekob}, $\varphi_4$ is a coboundary and $\varphi_2$
is not a coboundary.

We will show that $\varphi_2$ is a non-regular cocycle. Note that
the cocycles $\varphi_2+\varphi_4$ and $\varphi_2-\varphi_4$ take
values in $\Z$ and $\sqrt{5}\Z$ respectively. Since $\varphi_4$ is
a coboundary, it follows that $E(\varphi_2)\subset \Z$ and
$E(\varphi_2)\subset\sqrt{5}\Z$, and hence $E(\varphi_2)=\{0\}$.
Since $\varphi_2$ is not a coboundary,
$\overline{E}(\varphi_2)=\{0,\infty\}$, and hence  it is
non-regular.
\end{example}

\vskip 3mm {\it Acknowledgements}: This research was carried out
during visits of the first author to the Faculty of Mathematics and
Computer Science at Nicolaus Copernicus University, and of the
second author to the IRMAR at the University of Rennes 1. The
authors are grateful to their hosts for their hospitality and
support. They would like to thank Jacek Brzykcy for the numerical
computation of the example presented in Subsection
\ref{example7permut}.

\end{document}